\documentclass[a4paper,11pt]{article}

\usepackage[T1]{fontenc}
\usepackage{lmodern}

\usepackage{dsfont}

\usepackage[latin1]{inputenc}
\usepackage{amsmath}
\usepackage{amsthm}
\usepackage{amssymb}
\usepackage{mathrsfs}
\usepackage{graphicx}
\usepackage[all]{xy}
\usepackage{hyperref}

\usepackage{xcolor}

\usepackage{makeidx}

\usepackage{stmaryrd}
\usepackage{caption}

\usepackage{abstract} 

\newtheorem{thm}{Theorem}[section]
\newtheorem{cor}[thm]{Corollary}
\newtheorem{claim}[thm]{Claim}
\newtheorem{fact}[thm]{Fact}

\newtheorem{lemma}[thm]{Lemma}
\newtheorem{prop}[thm]{Proposition}

\theoremstyle{definition}
\newtheorem{definition}[thm]{Definition}

\newtheorem{remark}[thm]{Remark}
\newtheorem{question}[thm]{Question}

\def\rquotient#1#2{%
	\makeatletter
	\raise.3ex\hbox{$#1$}/\lower.3ex\hbox{$#2$}%
	\makeatother
}	

\makeatletter
\newcommand{\subjclass}[2][2010]{%
	\let\@oldtitle\@title%
	\gdef\@title{\@oldtitle\footnotetext{#1 \emph{Mathematics subject classification.} #2}}%
}
\newcommand{\keywords}[1]{%
	\let\@@oldtitle\@title%
	\gdef\@title{\@@oldtitle\footnotetext{\emph{Key words and phrases.} #1.}}%
}
\makeatother

\newcommand{\Address}{{
		\bigskip
		\small
		
				\textsc{Institut Montpellierain Alexander Grothendieck, 499-554 Rue du Truel, 34090 Montpellier, France.}\par\nopagebreak
		\textit{E-mail address}: \texttt{anthony.genevois@umontpellier.fr}
\medskip

		\textsc{Department of Mathematics, University of the
			Basque Country UPV/EHU,  Sarriena s/n, 48940 Leioa, Bizkaia, Spain. 
			IKERBASQUE, Basque Foundation for Science, Bilbao, Spain.}\par\nopagebreak
		\textit{E-mail address}: \texttt{anne.lonjou@ehu.eus}
\medskip

		\textsc{Department of Mathematics, 
			ETH Z\"urich,
			Rämistrasse 101, 
			8092 Z\"urich, Switzerland.}\par\nopagebreak
\textit{E-mail address}: \texttt{christian.urech@math.ethz.ch}
\medskip
		
}}

\makeindex

\title{Asymptotically rigid mapping class groups II: strand diagrams and nonpositive curvature}
\date{\today}
\author{Anthony Genevois, Anne Lonjou, and Christian Urech}


\begin{document}

\maketitle

\begin{abstract}
In this article, we introduce a new family of groups, called \emph{Chambord groups} and constructed from \emph{braided strand diagrams} associated to specific semigroup presentations. It includes the asymptotically rigid mapping class groups previously studied by the authors such as the braided Higman-Thompson groups and the braided Houghton groups. Our main result shows that polycyclic subgroups in Chambord groups are virtually abelian and undistorted. 
\end{abstract}

\tableofcontents

\section{Introduction}

\noindent
The groups $F$, $T$, and $V$ introduced by R. Thompson in the 1960s have a particular place in the history of group theory. First, $T$ and $V$ are the first examples of infinite finitely presented simple groups, and $F$ is the first example of a torsion-free group of type $F_\infty$ that is not of type $F$. But, since then, many groups have been constructed by varying the constructions of $F$, $T$, and $V$; see for instance \cite{HigmanV, Stein, MR1396957, Rover, Nekrashevych, BrinnV, FunarUniversal, DehornoybrV, BrinbrV, FunarKapoudjian, Sim, Monod, QV, MR4009393}. Although all these groups turn out to share similar properties, axiomatising the family of ``Thompson-like groups'' seems difficult; see \cite{Thumann, Witzel} for attempts in this direction. Nowadays, the investigation of Thompson-like groups is a subject on its own. Recent successes include the construction of new examples of non-amenable groups without non-abelian free subgroups \cite{Monod, LM} and the construction of simple groups distinguished by finiteness properties \cite{MR3910073, TwistedThompson}.

\medskip \noindent
In this article, we pursue the study initiated in \cite{GLU} of a particular family of braided Thompson-like groups. Our framework, largely inspired by \cite{FunarKapoudjian} (see the survey \cite{Survey} and the references therein for more background), is the following. Fix a locally finite tree $A$ embedded into the plane in such a way that its vertex-set is {closed and} discrete. The \emph{arboreal surface} $\mathscr{S}(A)$ is the oriented planar surface with boundary obtained by thickening $A$ in the plane. We denote by $\mathscr{S}^\sharp(A)$ the punctured arboreal surface obtained from $\mathscr{S}(A)$ by adding a puncture for each vertex of the tree. Now we fix a \emph{rigid structure} on $\mathscr{S}^\sharp(A)$, i.e. a decomposition into \emph{polygons} by means of a family of pairwise non-intersecting arcs whose endpoints are on the boundary of $\mathscr{S}(A)$ such that each polygon contains exactly one vertex of the underlying tree in its interior and such that each arc crosses once and transversely a unique edge of the tree. See for instance Figure~\ref{Dsharp}. We are interested in a specific subgroup $\mathfrak{mod}(A)$ of the big mapping class group of $\mathscr{S}^\sharp(A)$, corresponding to the (isotopy classes of the) homeomorphisms that send every polygon of the rigid structure to another polygon up to finitely many exceptions (loosely speaking, they preserve the rigid structure ``almost everywhere'').

\medskip \noindent
The main problem addressed in this article is the structure of the subgroups in asymptotically rigid mapping class groups. Our main result in this direction is the following statement:

\begin{thm}\label{thm:BigIntro}
Let $A$ be a locally finite planar tree. A polycyclic subgroup in $\mathfrak{mod}(A)$ is virtually abelian and undistorted in every finitely generated subgroup containing it.
\end{thm}

\noindent
This prevents many groups from belonging to our family of asymptotically rigid mapping groups, and even from embedding into one of them. It is worth mentioning that Theorem~\ref{thm:BigIntro} is essentially sharp since there are many examples containing distorted metabelian subgroups (namely, the lamplighter group $\mathbb{Z}\wr \mathbb{Z}$), see Proposition~\ref{prop:Wreath}.

\medskip \noindent
Along the road towards the proof of Theorem~\ref{thm:BigIntro}, we find a few other restrictions on the possible subgroups in asymptotically rigid mapping class groups. Our first by-product is the following statement:

\begin{thm}\label{thm:IntroFW}
Let $A$ be a locally finite planar tree. Every subgroup in $\mathfrak{mod}(A)$ satisfying the property $(FW_\omega)$ must be finite. In particular, this applies to finitely generated torsion groups and to groups satisfying Kazhdan's property $(T)$.
\end{thm}

\noindent
Here, $(FW_\omega)$ denotes the following fixed-point property: a group $G$ satisfies the property $(FW_\omega)$ if every action by automorphisms of $G$ on a CAT(0) cube complex with no infinite cube necessarily has bounded orbits. This is a particular case of the property $(FW)$ introduced in \cite{GuidoFW}, the latter encompassing Kazhdan's property $(T)$ according to \cite{MR1459140}, but, interestingly, finitely generated torsion groups satisfy the property $(FW_\omega)$ (see Lemma~\ref{lem:FixedPoint} below) but not {necessarily} the property $(FW)$ {(including the Grigorchuk group, because it admits Schreier graphs with at least two ends \cite{MR1841750, MR3027509}; and some Burnside groups \cite{MR3786300})}. 

\medskip \noindent
Our second by-product, {which is also a key step in the proof of Theorem~\ref{thm:BigIntro}}, shows that the natural copies of braid groups in asymptotically rigid mapping class groups are undistorted.

\begin{thm}\label{thm:IntroDistortion}
Let $A$ be a locally finite planar tree. For every admissible subsurface $\Sigma \subset \mathscr{S}^\sharp(A)$, the braid group $\mathrm{Mod}(\Sigma) \leq \mathfrak{mod}(A)$ is undistorted in every finitely generated subgroup containing it.
\end{thm}

\noindent
In the rest of the introduction, we describe in more details the strategy followed in order to prove Theorem~\ref{thm:BigIntro}.

\paragraph{Chambord groups.} First of all, inspired by the description of Thompson's groups as pairs of trees (see Section~\ref{section:warmup}), we develop a diagrammatic description of asymptotically rigid mapping class groups. More precisely, to every \emph{arboreal} semigroup presentation $\mathcal{P}= \langle \mathcal{A} \mid \mathcal{R} \rangle$ and to every baseword $w \in \mathcal{A}^+$, we construct the \emph{Chambord group}\footnote{The terminology refers to the \emph{ch\^ateau de Chambord}, a famous castle in France with a double-spiral staircase, similar to the wires turning around their cylinders in our diagrams.} $C(\mathcal{P},w)$ in terms of \emph{braided strand diagrams}. Then, we prove that asymptotically rigid mapping class groups are Chambord groups:

\begin{thm}\label{thm:IntroModChambord}
For every locally finite planar tree $A$, there exist an arboreal semigroup presentation $\mathcal{P}=\langle \mathcal{A} \mid \mathcal{R} \rangle$ and a letter $w \in \mathcal{A}$ such that the asymptotically rigid mapping class group $\mathfrak{mod}(A)$ is isomorphic to the Chambord group $C(\mathcal{P},w)$. 
\end{thm}

\noindent
{Our proof of Theorem~\ref{thm:BigIntro} is fundamentally based on this description of asymptotically rigid mapping class groups. But the interest of the formalism we introduce goes beyond Theorem~\ref{thm:BigIntro}, and we expect this point of view to be as fruitful as the diagrammatic representation of classical Thompson's groups. In our opinion, Theorem~\ref{thm:IntroModChambord} is the main contribution of the article: the introduction of a compact, two-dimensional, computable representation of asymptotically rigid mapping class groups.}

\medskip
\noindent
This contrasts drastically with the braided Thompson group $\mathrm{br}V$, introduced in \cite{BrinbrV, DehornoybrV}, whose definition already makes use of a diagrammatic representation. Asymptotically rigid mapping class groups are different in nature, and far from any obvious diagrammatic framework, which explains why Theorem 1.4 is surprising.

\paragraph{Cube complexes.} Inspired by the cubulation of \emph{diagram groups} \cite{MR1978047, MR2136028}, a family of groups also constructed from an extension of the representations of Thompson's groups in terms of pairs of trees \cite{MR1448329, MR1396957}, we construct CAT(0) cube complexes on which Chambord groups act with stabilisers isomorphic to finite extensions of braid groups. 

\begin{thm}\label{thm:CCintro}
Let $\mathcal{P}=\langle \mathcal{A} \mid \mathcal{R} \rangle$ be an arboreal semigroup presentation and $w \in \mathcal{A}^+$ a baseword. There exists a locally finite-dimensional CAT(0) cube complex $M(\mathcal{P},w)$ on which the Chambord group $C(\mathcal{P},w)$ acts with cube-stabilisers isomorphic to finite extensions of finitely generated braid groups.
\end{thm}

\noindent
Observe that Theorem~\ref{thm:IntroFW} follows almost immediately from Theorem~\ref{thm:CCintro}. Nevertheless, we emphasize that the key point in Theorem~\ref{thm:CCintro} is the actual construction of the cube complex. It is because its structure is tightly connected to the structure of the corresponding Chambord group that we are able to deduce valuable information on the structure of subgroups.

\medskip \noindent
It is worth noticing that, in our previous work \cite{GLU}, we already constructed contractible cube complexes on which asymptotically rigid mapping class groups act. However, it may or may not be CAT(0). For instance, it turns out to be CAT(0) for the braided Ptolemy-Thompson groups $\mathrm{br}T_n$, but it turns out not to be CAT(0) for the braided Houghton groups (as shown in \cite[Section~3.3]{GLU}) or some braided Higman-Thompson groups. The main obstruction comes from the absence of polygons with no puncture in the rigid structure. In the particular case of the braided Houghton and Higman-Thompson groups, one can overcome the difficulty by modifying artificially the planar trees and adding a ``ghost vertex'' that does not contribute to the puncture of the planar surface. Then, the construction of \cite{GLU} can adapted in a straightforward way, leading to a CAT(0) cube complex. Unfortunately, it seems unlikely that such a trick can be applied in full generality. For instance, difficulties already appear with the tree obtained from the right tree given by Figure~\ref{Graphes} by subdividing all the horizontal edges. This is why a new approach was needed, and why the formalism of braided strand diagrams is fundamental in our work.

\paragraph{Sketch of proof.} Let $A$ be a locally finite planar tree. Given an admissible subsurface $\Sigma \subset \mathscr{S}^\sharp(A)$, we exploit the description of $\mathfrak{mod}(A)$ as a Chambord group in order to define a natural projection $\mathfrak{mod}(A) \to \mathrm{Mod}(\Sigma)$. It turns out that the restriction of this projection to any finitely generated subgroup in $\mathfrak{mod}(A)$ induces a quasi-retraction, proving that $\mathrm{Mod}(\Sigma)$ is undistorted in $\mathfrak{mod}(A)$, as claimed by Theorem~\ref{thm:IntroDistortion}. Thus, $\mathfrak{mod}(A)$ acts on a CAT(0) cube complex with undistorted vertex-stabilisers. Because we know that abelian subgroups are undistorted in braid groups, we conclude from a general argument (see Proposition~\ref{prop:CCundistorted}) that finitely generated abelian subgroups in $\mathfrak{mod}(A)$ are undistorted, proving the second part of Theorem~\ref{thm:BigIntro}.

\medskip \noindent
Next, let $H \leq \mathfrak{mod}(A)$ be a polycyclic subgroup. As a consequence of \cite{CubicalFlat}, we know that, as soon as a polycyclic group acts on a CAT(0) cube complex, necessarily the action factorises (more or less) through an abelian quotient. By exploiting the structure of vertex-stabilisers in our cube complex and by using the fact that polycyclic subgroups in braid groups are virtually abelian, we manage to prove that $H$ must be nilpotent. Because a finitely generated nilpotent group that is not virtually abelian must contain distorted cyclic subgroups, we deduce from the previous paragraph that $H$ must be virtually abelian, concluding the proof of Theorem~\ref{thm:BigIntro}.

\paragraph{Organisation of the article.} The formalism of braided strand diagrams is developed in Section~\ref{section:Chambord}, where Chambord groups are defined. Section~\ref{section:MODA} is dedicated to the proof of Theorem~\ref{thm:IntroModChambord}. The construction of the cube complex from Theorem~\ref{thm:CCintro} is described in Section~\ref{section:CCC}, where it is also proved to be CAT(0). A few direct applications of this construction, including Theorem~\ref{thm:IntroFW}, are presented in Section~\ref{section:Afew}. Theorem~\ref{thm:BigIntro} is finally proved in Section~\ref{section:Polycyclic}. We conclude the article with a short discussion and a few open questions in Section~\ref{section:last}.

\paragraph{Acknowledgments.} The authors would like to thank the anonymous referee for helpful suggestions. The second author was partially supported by the Basque Government grant IT1483-22, by MCIN /AEI /10.13039/501100011033 / FEDER through the Spanish Governement grant PID2022-138719NA-I00, and by the French National grant through the project GOFR ANR-22-CE40-0004.

\section{Chambord groups}\label{section:Chambord}

\subsection{Warm up: Thompson's groups}\label{section:warmup}

\begin{figure}
\begin{center}
\includegraphics[width=\linewidth]{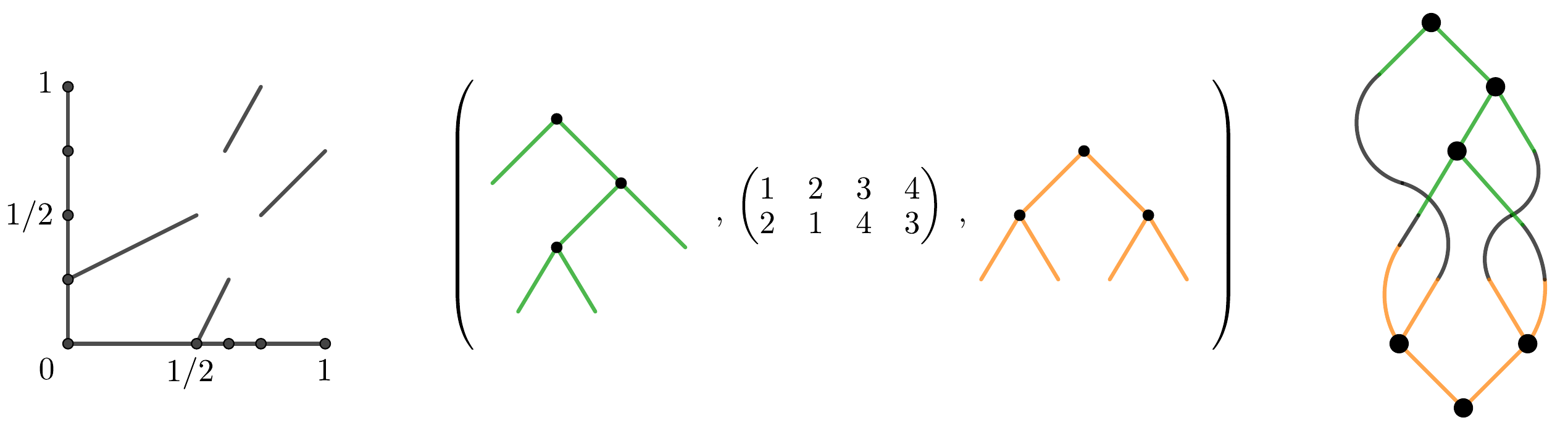}
\caption{Equivalent representations of an element of $V$.}
\label{Thompson}
\end{center}
\end{figure}
\noindent
A \emph{dyadic decomposition} of $[0,1]$ is a collection of intervals $(I_k)$ of the form $\left[ \frac{j}{2^m}, \frac{j+1}{2^m} \right]$ covering $[0,1]$ such that the intersection between any two intervals contains at most one point. Given two dyadic decompositions $\mathfrak{A}, \mathfrak{B}$ of $[0,1]$ and a bijection $\sigma : \mathfrak{A} \to \mathfrak{B}$, the map $\mathfrak{C} \to \mathfrak{C}$ defined on the Cantor set $\mathfrak{C} \subset [0,1]$ by sending $A \cap \mathfrak{C}$ to $\sigma(A) \cap \mathfrak{C}$ through an affine map induces a homeomorphism of $\mathfrak{C}$. \emph{Thompson's group $V$} is the group of the homeomorphisms of $\mathfrak{C}$ that decompose in this way. By adding the restriction that $\sigma$ has to preserve the left-right (resp. cyclic) order of $[0,1]$, one obtains \emph{Thompson's group} $F$ (resp. $T$). 

\medskip \noindent
Because a dyadic decomposition can be encoded as a rooted binary tree, in such a way that its leaves correspond the pieces of the decomposition (see Figure~\ref{Thompson}), an element of $V$ can be written as a triple $(R,\sigma,S)$ where $R,S$ are two rooted binary trees with the same number of vertices and where $\sigma$ is a bijection from the leaves of $R$ to the leaves of $S$. Following \cite{MR1396957, MR2706280}, it is convenient to think of $(R,\sigma,S)$ as a unique object, a \emph{strand diagram}, obtained by placing a reversed copy of $S$ below $R$ and by connecting with strands the leaves of $R$ to the leaves of $S$ according to $\sigma$. See Figure~\ref{Thompson}. 

\medskip \noindent
An element of $V$ can be represented by different strand diagrams. This is clear from the topological point of view: the identity coincides with the homeomorphism $(\mathfrak{C} \cap [0,1/2]) \cup (\mathfrak{C} \cap [1/2,1]) \to (\mathfrak{C} \cap [0,1/2]) \cup (\mathfrak{C} \cap [1/2,1])$ that sends the left interval to the left interval and the right interval to the right interval. In terms of diagrams, such a phenomenon corresponds to a \emph{dipole}, i.e. there exist a vertex $r$ in the top tree and a vertex $s$ in the bottom tree such that the children of $r$ are connected with strands to the children of $s$ from left to right. A dipole can be \emph{reduced} by removing the children of $r$ and $s$ from the top and bottom trees, and by connecting $r$ and $s$ with a strand. See the first equality of Figure~\ref{Product}. Now, an element of $V$ is represented by a unique \emph{reduced} diagram, i.e. a diagram without any dipole. 

\medskip \noindent
The group law of $V$ can be described in terms of diagrams. Let $a,b \in V$ be two elements respectively represented by $(R,\mu,S), (U,\nu,V)$. By adding dipoles, we can represent $a,b$ by $(R',\mu',S'), (U',\nu',V')$ in such a way that $S'=U'$. Then, it is clear that $(R',\nu' \mu',V')$ represents the product $ab$. See Figure~\ref{Product}. In other words, we add dipoles to the two diagrams until the bottom tree of the first diagram coincides with the top tree of the second diagram, then we remove these two trees, and we glue the second diagram below the first one. 

\medskip \noindent
In the rest of the section, our goal is to introduce a new kind of strand diagrams and to construct groups by mimicking what happens for Thompson's groups as described above. 
\begin{figure}
\begin{center}
\includegraphics[width=0.8\linewidth]{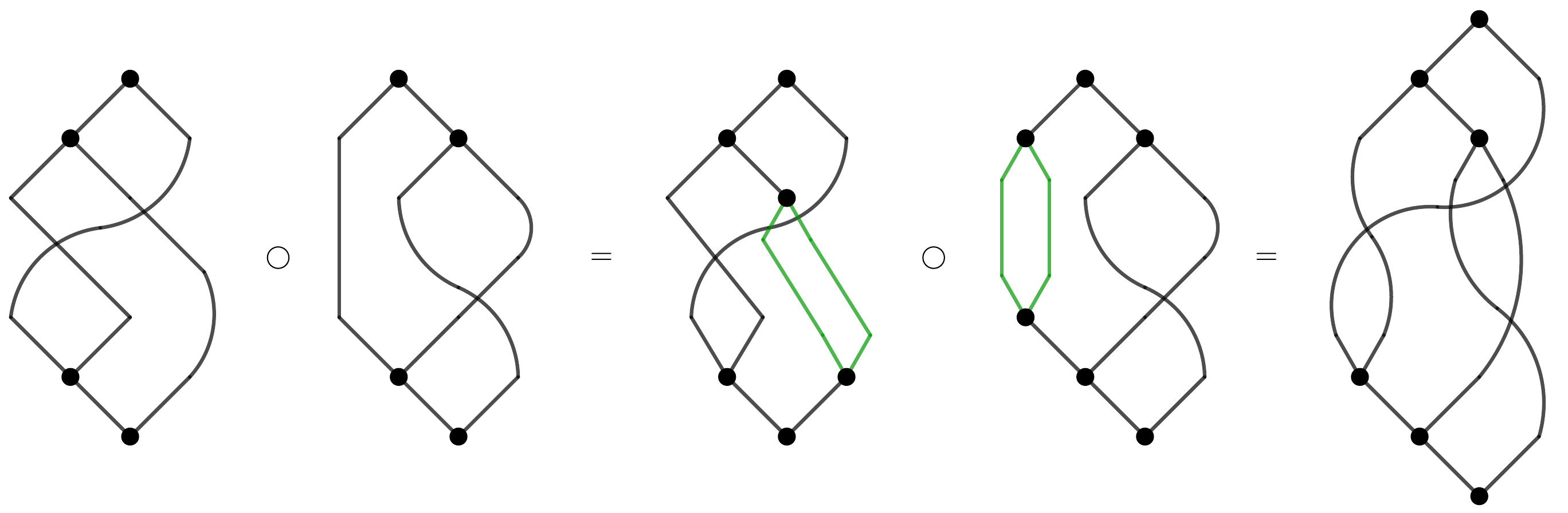}
\caption{A product in $V$.}
\label{Product}
\end{center}
\end{figure}

\subsection{Arboreal semigroup presentations}

\noindent
Recall that a \emph{semigroup presentation} $\langle \mathcal{A} \mid \mathcal{R} \rangle$ is the data of an \emph{alphabet} {$\mathcal{A}$} and a collection $\mathcal{R}$ of \emph{relations} of the form $u=v$ where $u,v$, are non-empty words written over $\mathcal{A}$, i.e. elements of $\mathcal{A}^+$. This section is dedicated to a specific family of semigroup presentations, which will be central in the construction of our groups. 

\begin{definition}
A semigroup presentation $\mathcal{P}= \langle \mathcal{A} \mid \mathcal{R} \rangle$ is \emph{arboreal} if the following conditions hold:
\begin{itemize}
	\item for every relation $u=v$ in $\mathcal{R}$, $u$ has length one;
	\item for all relations $u_1=v_1$ and $u_2=v_2$ in $\mathcal{R}$, $u_1$ and $u_2$ coincide if and only if so do $v_1$ and $v_2$.
\end{itemize}
\end{definition}

\noindent
The terminology is justified by a tight connection between arboreal semigroup presentations and planar labelled forests. More precisely, let $\mathcal{P}= \langle \mathcal{A}\mid \mathcal{R} \rangle$ be an arboreal semigroup presentation. Given a word $w \in \mathcal{A}^+$, we define a $\mathcal{A}$-labelled planar rooted forest $T(\mathcal{P},w)$ as follows:
\begin{itemize}
	\item $T(\mathcal{P},w)$ has $|w|$ roots, labelled from left to right by the letters of $w$;
	\item if $x$ is a vertex of $T(\mathcal{P},w)$ labelled by $u \in \mathcal{A}$ and if the relation $u=v$ belongs to $\mathcal{R}$, then $x$ has $|v|$ children labelled from left to right by the letters of $v$.
\end{itemize}
See Figure \ref{TP} for an example. Notice that $T(\mathcal{P},w)$ is a tree if and only if $w$ has length one. Conversely, let $\mathcal{A}$ be an alphabet and let $T$ be a $\mathcal{A}$-labelled planar locally finite rooted forest such that
\begin{itemize}
	\item the colour of a vertex depends only on the colours of its children (ordered from left to right);
	\item and, conversely, the colours of the children of a vertex $x$ (ordered from left to right) depend only on the colour of $x$.
\end{itemize}
Then $T$ coincides with $T(\langle \mathcal{A} \mid \mathcal{R} \rangle,w)$ where $w$ is the word obtained by reading from left to right the colours of the roots of $T$ and where $\mathcal{R}$ is the set of relations of the form $u=v$ where $u$ is the colour of a vertex $x$ and where $v$ is the word obtained by reading from left to right the colours of the children of $x$. Observe that $\langle \mathcal{A} \mid \mathcal{R} \rangle$ is an arboreal semigroup presentation as a consequence of our assumptions on the $\mathcal{A}$-colouring of $T$. For instance, the coloured tree given by Figure~\ref{TP} coincides with $T(\mathcal{P},a)$ where $\mathcal{P}= \langle a,\textcolor{red}{b}, \textcolor{green}{c} , \textcolor{blue}{d} \mid a = \textcolor{red}{b} \textcolor{green}{c}, \ \textcolor{green}{c} = \textcolor{green}{cc}, \ \textcolor{red}{b}= \textcolor{red}{b} \textcolor{blue}{d}  \rangle$.

\medskip \noindent
{It is worth noticing that, for every planar locally finite rooted forest $F$, there exists an arboreal semigroup presentation $\mathcal{P}= \langle \mathcal{A} \mid \mathcal{R} \rangle$ and a word $w \in \mathcal{A}^+$ such that $F$ is isomorphic to $T(\mathcal{P},w)$. Indeed, let our alphabet $\mathcal{A}$ be the vertex-set $V(F)$ of $F$; let $w$ be the word obtained by reading the roots of $F$ from left to right; and let $\mathcal{R}$ be the set of relations of the form $u=u_1 \cdots u_k$ where $u$ is a vertex of $F$ that is not a leaf and where $u_1, \ldots, u_k$ are its children ordered from left to right. Then map $F \to T(\mathcal{P},w)$ that sends every vertex $u$ of $F$ to the unique vertex of $T(\mathcal{P},w)$ labelled by the letter $u \in \mathcal{A}$ induces a graph isomorphism.}

\medskip \noindent
In the next sections, we will be interested in finite pieces of $T(\mathcal{P},w)$, as formalised by the following definition:
\begin{figure}
\begin{center}
\includegraphics[width=0.5\linewidth]{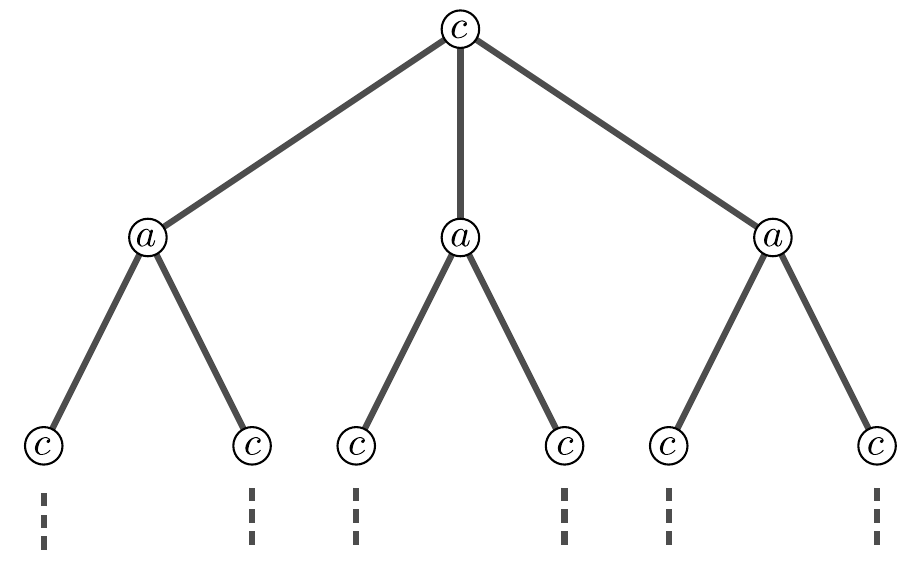} \hspace{1cm}
\includegraphics[width=0.37\linewidth]{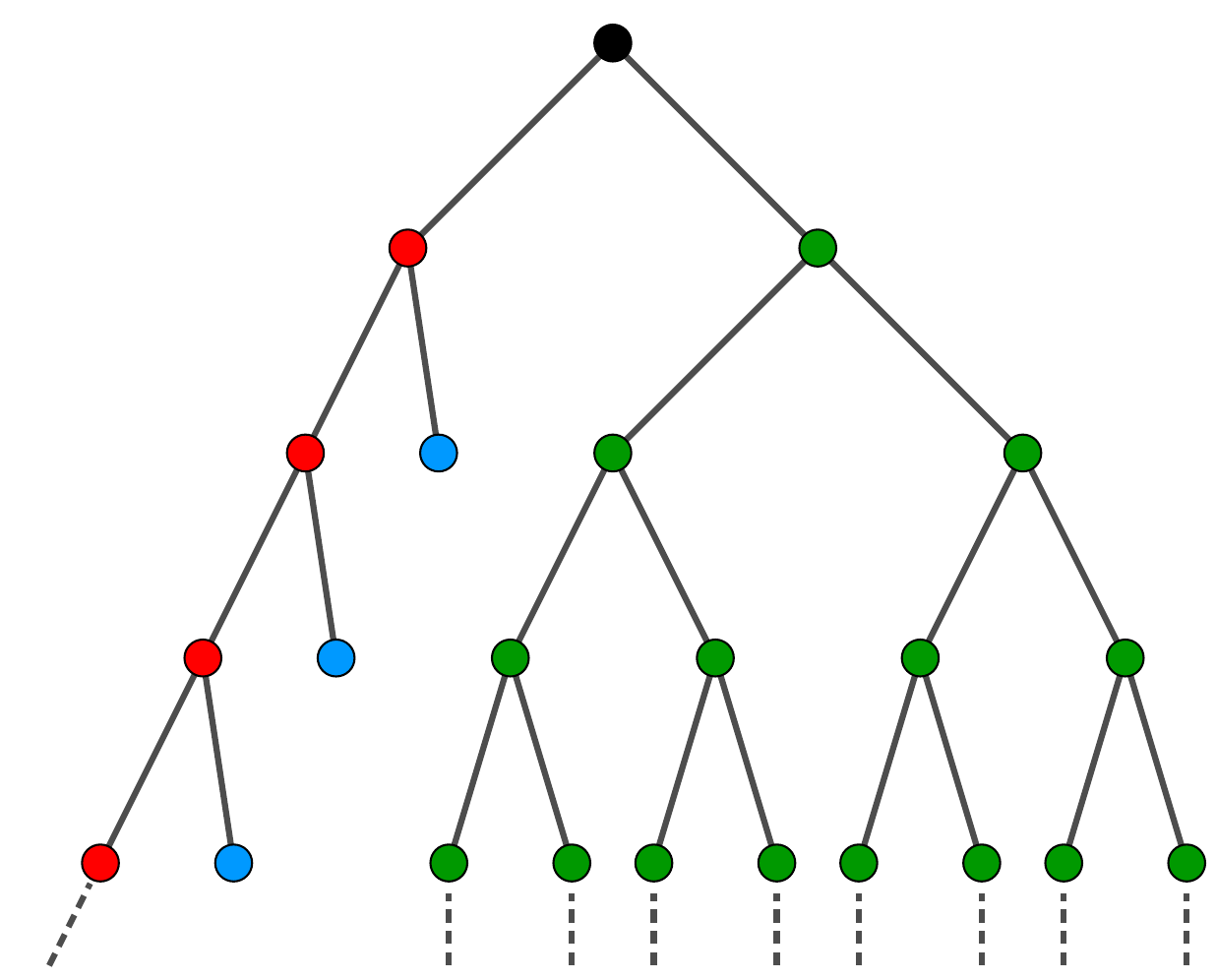}
\caption{On the left, the tree $T(\mathcal{P},c)$ where $\mathcal{P}=\langle a,c \mid a=c^2, c=a^3 \rangle$. On the right, a coloured planar rooted tree.}
\label{TP}
\end{center}
\end{figure}
\begin{definition}
Let $\mathcal{P}= \langle \mathcal{A} \mid \mathcal{R} \rangle$ be an arboreal semigroup presentation. A \emph{$\mathcal{P}$-forest} is a finite planar $\mathcal{A}$-labelled rooted forest $F$ such that, if $x \in F$ is a vertex labelled by a letter $u \in \mathcal{A}$, if $x$ is not a leaf of $F$, and if $u=v$ is a relation in $\mathcal{R}$, then $x$ has $|v|$ children labelled from left to right by the letters of $v$. If $w \in \mathcal{A}^+$ is the word obtained by reading from left to right the labels of the roots of $F$, then $F$ is a \emph{$(\mathcal{P},w)$-forest}.
\end{definition}

\noindent
In other words, a $\mathcal{P}$-forest $F$ is a rooted subforest of $T(\mathcal{P},w)$ such that, if $x$ is a vertex in $F$ having one child in $F$, then all its children in $T(\mathcal{P},w)$ lie in $F$. There is a natural partial order on the $(\mathcal{P},w)$-forests: given two such forests $F_1$ and $F_2$, $F_1$ is a \emph{prefix} of $F_2$ if $F_2$ can be obtained from $F_1$ by adding all the children to leaves iteratively. It may be useful to keep in mind that $(\mathcal{P},w)$-forests are rooted subforests of $T(\mathcal{P},w)$ containing the roots. In particular, the union $F_1 \cup F_2$ of two $(\mathcal{P},w)$-forest $F_1$ and $F_2$ makes sense. Formally, it coincides with the smallest prefix of $T(\mathcal{P},w)$ admitting $F_1$ and $F_2$ as prefixes.

\paragraph{Left-right orders on planar trees.} In all the article, vertices of planar rooted trees are ordered by left-right orders in the following way. First, up to isomorphism (in the category of planar rooted trees), we can always assume that our tree $T$ is drawn on the plane in such a way that, for every vertex $x$ of $T$, the vertical line passing through $x$ separates the leftmost child of $x$ and its descendants from the other children of $x$ and their descendants. Then, we order the vertices of $T$ through the obvious left-right order on the vertical lines of the plane.

\subsection{Braided strand diagrams and dipole reduction}\label{sec:braidedstrand}

\noindent
Before describing the braided strand diagrams we are interested in, we begin with a general discussion about punctured discs.

\medskip \noindent
For all $p,q \geq 0$, let $\mathscr{D}_{p,q}$ be a closed disc with $p$ punctures in its interior and $q$ marked points on its boundary; and let $\mathcal{B}_{p,q}$ denote the group of the orientation-preserving homeomorphisms of $\mathscr{D}_{p,q}$ that preserve (setwise) the punctures and marked points up to isotopy fixing pointwise the punctures and the marked points. As usual, any element of $\mathcal{B}_{p,q}$ can be represented by the data, in the cylinder $\mathscr{D}_{p,q} \times [0,1]$, of $p$ \emph{strands} in $\mathrm{int}(\mathscr{D}_{p,q}) \times [0,1]$ connecting the punctures of $\mathscr{D}_{p,q}\times \{0\}$ to the punctures of $\mathscr{D}_{p,q} \times \{1\}$ and $q$ \emph{wires} in $\partial \mathscr{D}_{p,q} \times [0,1]$ connecting the marked points of $\mathscr{D}_{p,q} \times \{0\}$ to the marked points of $\mathscr{D}_{p,q} \times \{1\}$ preserving the cyclic order on the wires. See Figure~\ref{Disque} for an example. We emphasize that, in $\mathcal{B}_{p,q}$, a full twist along the boundary coincides with a full twist of the punctures, as represented by Figure~\ref{Disque}. Algebraically speaking, this amounts to saying that $\mathcal{B}_{p,q}$ satisfies the short exact sequence
$$1 \to \mathcal{B}_p \to \mathcal{B}_{p,q} \to \mathbb{Z}_q \to 1$$
where $\mathcal{B}_p$ denotes the subgroup of $\mathcal{B}_{p,q}$ corresponding to the homeomorphisms fixing pointwise the boundary, and which coincides with the usual braid group on $p$ strands, and where the morphism $\mathcal{B}_{p,q} \to \mathbb{Z}_q$ is given by the action by cyclic permutations on the marked points. 
\begin{figure}
\begin{center}
\includegraphics[trim={1cm 3cm 1cm 1cm},clip,width=0.45\linewidth]{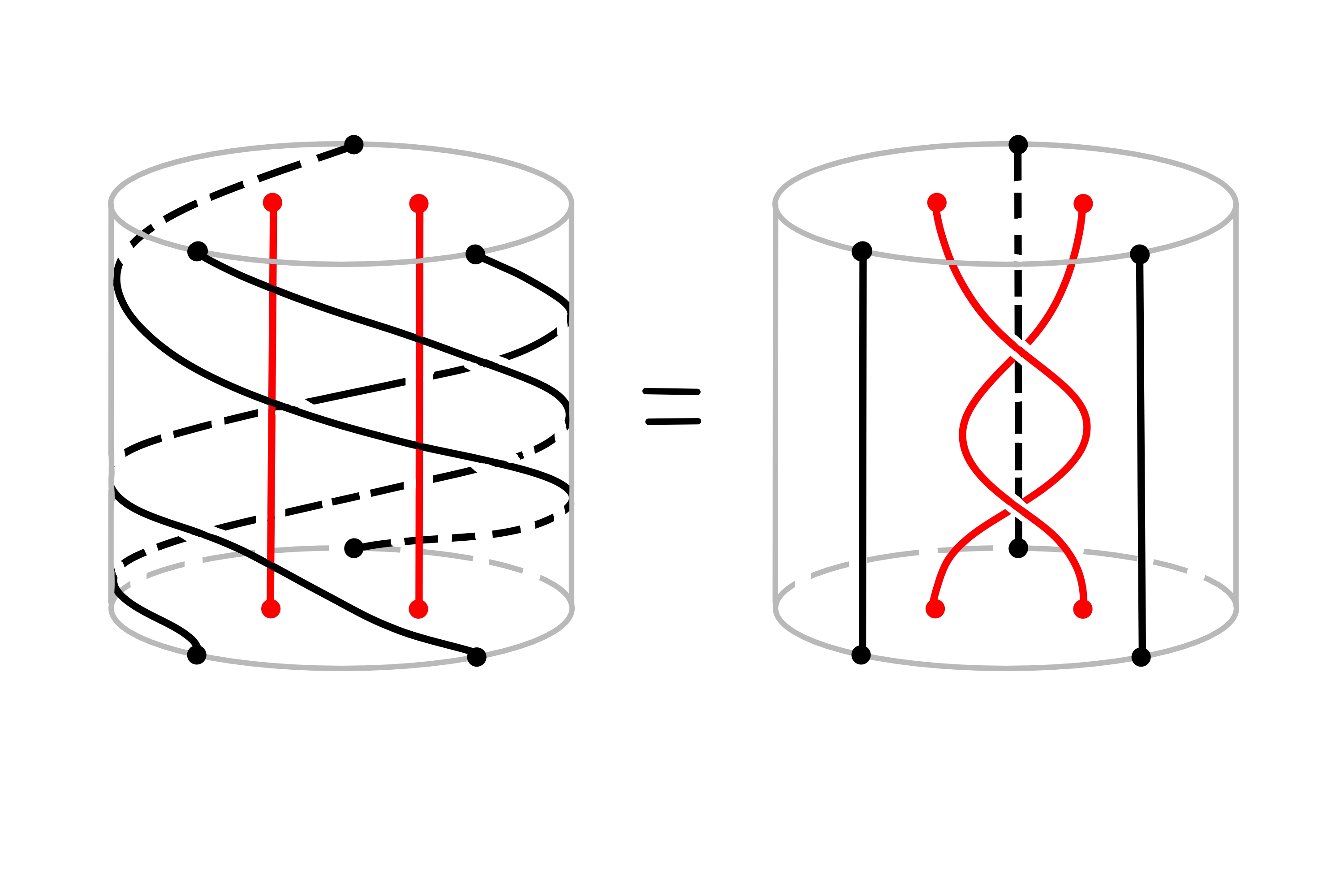}
\includegraphics[trim={0 2cm 2cm 2cm},clip,width=0.45\linewidth]{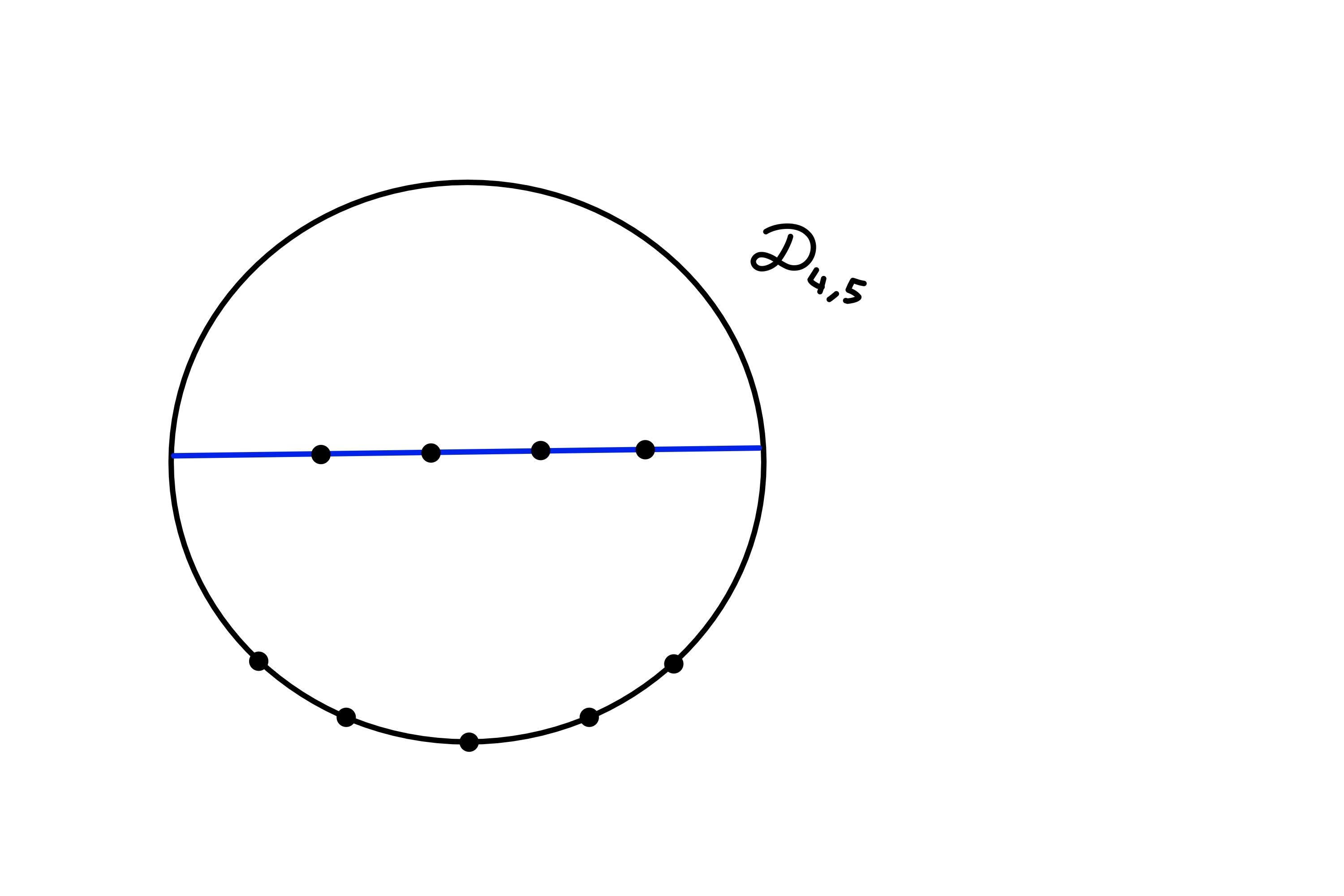}
\caption{An element of $\mathcal{B}_{p,q}$ and a representation of $\mathscr{D}_{4,5}$.}
\label{Disque}
\end{center}
\end{figure}

\medskip \noindent
\textbf{Convention:} In all the article, the punctures of $\mathscr{D}_{p,q}$ will be thought of as lying on the same horizontal line and its marked points as lying below this line; see Figure~\ref{Disque}. 

\medskip \noindent
We are now ready to define our braided strand diagrams. In the rest of this section, we fix an arboreal semigroup presentation $\mathcal{P}=\langle \mathcal{A}\mid \mathcal{R} \rangle$ and a baseword $w \in \mathcal{A}^+$. 

\begin{definition}\label{def:Diag}
A \emph{braided strand diagram over $(\mathcal{P},w)$} is a triple $(A, \beta, B)$ where $A,B$ are two  $(\mathcal{P},w)$-forests, with the same numbers of \emph{$\mathcal{P}$-leaves} $q$ (i.e. vertices that are leaves in $A,B$ but not in $T(\mathcal{P},w)$) and \emph{$\mathcal{P}$-interior vertices} $p$ (i.e. vertices that are not $\mathcal{P}$-leaves), and where $\beta \in \mathcal{B}_{p,q}$ is a braid that induces a bijection from the $\mathcal{P}$-leaves of $A$ to the $\mathcal{P}$-leaves of $B$ preserving the $\mathcal{A}$-labelling.
\end{definition}

\noindent
In this definition, we think of the marked points of $\mathscr{D}_{p,q}$ as being indexed from left to right by the $\mathcal{P}$-leaves of both $A$ and $B$. Therefore, a braid $\beta \in \mathcal{B}_{p,q}$, which induces a permutation of the marked points, naturally defines a bijection from the $\mathcal{P}$-leaves of $A$ to the $\mathcal{P}$-leaves of $B$. In our definition, we require this bijection to preserve the $\mathcal{A}$-labelling. 

\medskip \noindent
Graphically, a braided strand diagram $(A,\beta,B)$ can be thought of as follows. Draw a copy of $A$ above a cylinder $\mathscr{D}_{p,q} \times [0,1]$ and connect its $\mathcal{P}$-leaves (resp. its $\mathcal{P}$-interior vertices) to the marked points (resp. the punctures) of $\mathscr{D}_{p,q} \times \{0\}$ from left to right; similarly, draw an inverted copy of $B$ below the cylinder and connect its $\mathcal{P}$-leaves (resp. its $\mathcal{P}$-interior vertices) to the marked points (resp. the punctures) of $\mathscr{D}_{p,q} \times \{1\}$ from left to right. Now, following the given braid $\beta$, connect the marked points of $\mathscr{D}_{p,q} \times \{0\}$ to the marked points of $\mathscr{D}_{p,q} \times \{1\}$ with \emph{wires} in $\partial \mathscr{D}_{p,q} \times [0,1]$ and connect the punctures of $\mathscr{D}_{p,q} \times\{0\}$ to the punctures in $\mathscr{D}_{p,q} \times \{1\}$ with \emph{strands} in the interior of the cylinder $\mathscr{D}_{p,q} \times [0,1]$. See Figure~\ref{BraidDiag}. 
\begin{figure}
\begin{center}
\includegraphics[width=0.85\linewidth]{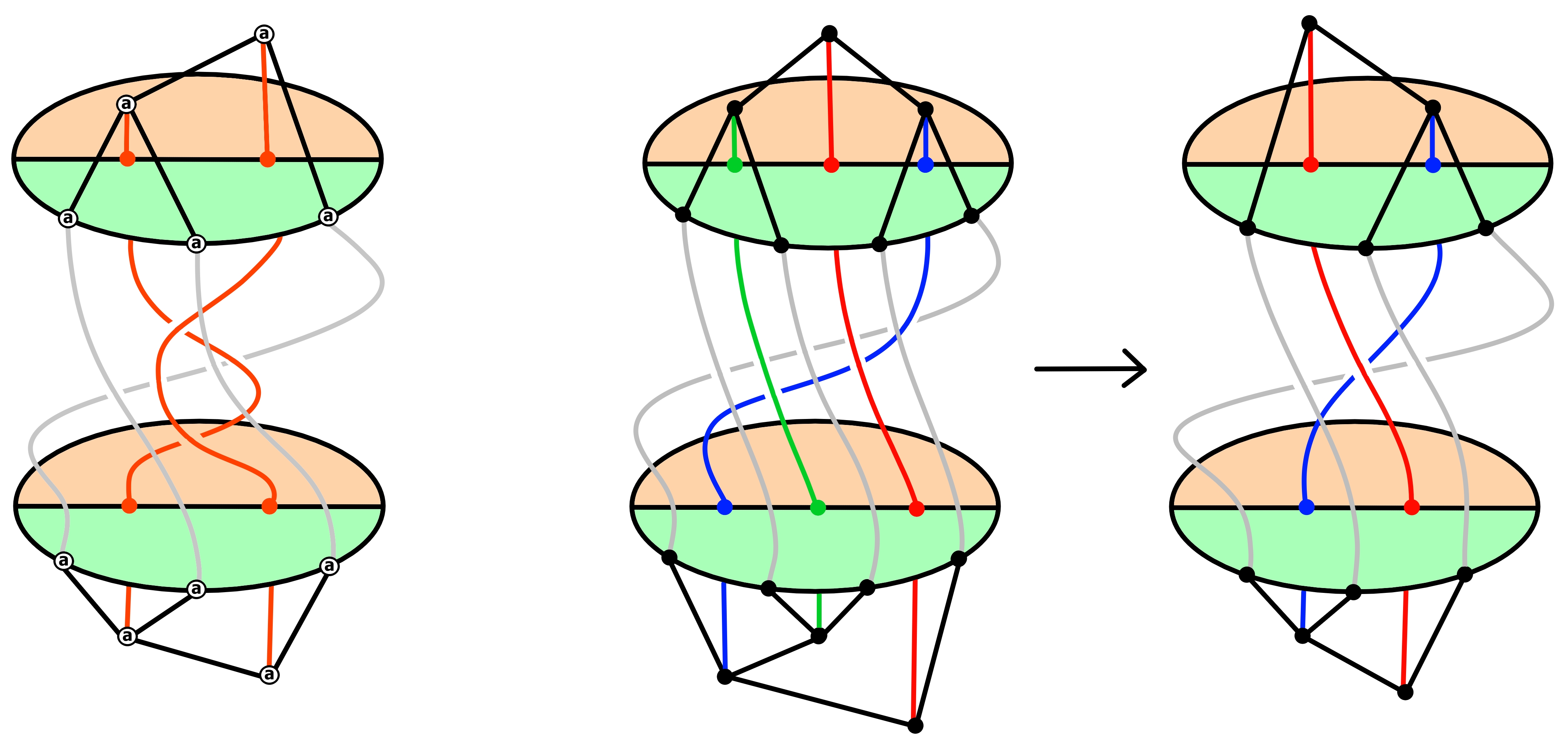}
\caption{On the left, a braided strand diagram over $\left( \langle a \mid a=a^2 \rangle, a \right)$; on the right, a dipole reduction.}
\label{BraidDiag}
\end{center}
\end{figure}

\medskip \noindent
A trivial example of a braided strand diagram over $(\mathcal{P},w)$ is $\epsilon(w):= (\eta(w), \mathrm{id}, \eta(w))$ where $\eta(w)$ denotes the $(\mathcal{P},w)$-forest with no edges. Graphically, it has no strands. This diagram will play a central role in the sequel since it will yield the neutral element of our group.

\begin{definition}
In a braided strand diagram $(A, \beta,B)$, a strand $\mu$ is \emph{parallel} to a wire $\nu$ if the latter is homotopic to the arc $\mu'$ in the cylinder $\mathscr{D}_{p,q} \times [0,1]$ minus the strands distinct from $\mu$, where $\mu'$ is obtained by following a straight line in $\mathscr{D}_{p,q} \times \{0\}$ from the starting point of $\nu$ to the starting point of $\mu$, then $\mu$, and finally a straight line in $\mathscr{D}_{p,q} \times \{1\}$ from the ending point of $\mu$ to the starting point of $\nu$. 
\end{definition}

\noindent
The right part of Figure~\ref{BraidDiag} illustrates a dipole reduction. In the left diagram, the blue strand is parallel to the back wire. Roughly speaking, this is because it can be ``pushed'' to the back wire without crossing the red and green strands. But the blue strand is not parallel to the other wires.

\begin{definition}\label{def:Dipoles}
In a braided strand diagram $(A, \beta, B)$, a \emph{dipole} is the data of two $\mathcal{P}$-interior vertices $a \in A$, $b \in B$ such that:
\begin{itemize}
	\item the children of $a$ and $b$ are $\mathcal{P}$-leaves;
	\item if $\ell_1, \ldots, \ell_k$ denote the children of $a$ written from left to right, then the children of $b$ are $\beta(\ell_1), \ldots, \beta(\ell_k)$ from left to right;
	\item the strand starting from $a$ ends at $b$ and is parallel to the wires connecting the children of $a$ to the children of $b$. 
\end{itemize}
The diagram $(A',\beta',B')$ is obtained from $(A,\beta,B)$ by \emph{reducing the dipole} if $A'=A \backslash \{ \text{children of $a$} \}$, if $B'=B \backslash \{ \text{children of $b$}\}$, and if $\beta' \in \mathcal{B}_{p-1, q-k+1}$ is obtained from $\beta$ by replacing the strand connecting $a,b$ and the wires connecting their children with a single wire. A diagram is \emph{reduced} if it does not contain any dipole.
\end{definition}

\noindent
See Figure~\ref{BraidDiag} for an example of a dipole reduction. Observe that the fact that the green strand passes in front of the blue strand is primordial: if we modify the crossing, letting the blue strand pass in front of the green strand, then the dipole does not exist anymore and the diagram becomes reduced.

\begin{definition}\label{def:equipotent}
Two braided strand diagrams $\Phi,\Psi$ are \emph{equipotent} if there exists a sequence of diagrams 
$$\Delta_1 = \Phi, \ \Delta_2, \ldots, \ \Delta_{k-1}, \ \Delta_k=\Psi$$
such that, for every $1 \leq i \leq k-1$, one of $\Delta_i,\Delta_{i+1}$ is obtained from the other by reducing a dipole. 
\end{definition}

\noindent
In the rest of this section, our goal is to show the following statement:

\begin{prop}\label{prop:reduction}
A braided strand diagram is equipotent to a unique reduced diagram. 
\end{prop}

\noindent
Consequently, it makes sense to refer to the \emph{reduced form} of a diagram. The proof of Proposition~\ref{prop:reduction} relies on the following observation:  

\begin{lemma}[\cite{MR7372}]\label{lem:Confluent}
If a directed graph is terminating and locally confluent, then it must be confluent.
\end{lemma}

Recall that a directed graph $X$ is \emph{terminating} if it has no infinite directed path; it is \emph{confluent} if, for all vertices $a,b,c$ such that $X$ contains directed paths from $a$ to $b$ and $c$, there exists a fourth vertex $d$ such that $X$ contains directed paths from $b$ and $c$ to $d$; and it is \emph{locally confluent} if, for all vertices $a,b,c$ such that $X$ contains oriented edges from $a$ to $b$ and $c$, there exists a fourth vertex $d$ such that $X$ contains directed paths from $b$ and $c$ to $d$.

\begin{proof}[Proof of Proposition \ref{prop:reduction}.]
Fix a diagram $\Delta$ and consider the directed graph $X$ whose vertices are all the diagrams equipotent to $\Delta$ and whose oriented edges connect two diagrams when we can pass from one to the other by reducing a dipole. In order to conclude that there exists a unique reduced diagram equipotent to $\Delta$, it suffices to show that $X$ is confluent. First, notice that $X$ is clearly terminating since reducing a dipole makes the total number of vertices in the forests decrease. Next, let $\Phi,\Psi_1,\Psi_2$ be diagrams equipotent to $\Delta$ such that $\Psi_1,\Psi_2$ are obtained from $\Phi$ by reducing two distinct dipoles. {We claim that one can reduce simultaneously these two dipoles, producing a new diagram $\Psi$ that is the terminal vertex in $X$ of two oriented edges starting from $\Psi_1$ and $\Psi_2$. This will show that $X$ is locally confluent. Concretely,} we can write $\Phi =(A,\beta,B)$, for some $(\mathcal{P},w)$-forests $A,B$ and some braid $\beta \in \mathcal{B}_{p,q}$, and we can find two $\mathcal{P}$-interior vertices $a_1,a_2$ in $A$ (resp. $b_1,b_2$ in $B$) such that:
\begin{itemize}
	\item the children of $a_1,a_2,b_1,b_2$ are $\mathcal{P}$-leaves;
	\item if $m_1, \ldots, m_r$ denote the children of $a_1$ written from left to right, then the children of $b_1$ are $\beta(m_1), \ldots, \beta(m_r)$;
	\item the strand starting from $a_1$ ends at $b_1$ and is parallel to the wires connecting the children of $a_1$ to the children of $b_1$;
	\item $\Psi_1=(A_1,\beta_1,B_1)$ where $A_1:= A \backslash \{\text{children of $a_1$}\}$, where $B_1:= B \backslash \{\text{children of $b_1$}\}$, and where $\beta_1 \in \mathcal{B}_{p-1,q-r+1}$ is obtained from $\beta$ by replacing the strand connecting $a_1,b_1$ and the wires connecting their children with a single wire;
	\item if $n_1, \ldots, n_s$ denote the children of $a_2$ written from left to right, then the children of $b_2$ are $\beta(n_1),\ldots, \beta(n_r)$;
	\item the strand starting from $a_2$ ends at $b_2$ and is parallel to the wires connecting the children of $a_2$ to the children of $b_2$;
	\item $\Psi_2=(A_2,\beta_2,B_2)$ where $A_2:= A \backslash \{\text{children of $a_2$}\}$, where $B_2:= B \backslash \{\text{children of $b_2$}\}$, and where $\beta_2 \in \mathcal{B}_{p-2,q-r+2}$ is obtained from $\beta$ by replacing the strand connecting $a_2,b_2$ and the wires connecting their children with a single wire.
\end{itemize}
Then, set the diagram $\Psi := (A',\beta',B')$ where $A':= A \backslash \{\text{children of $a_1,a_2$}\}$, where $B':= B \backslash \{\text{children of $b_1,b_2$}\}$, and where $\beta'$ is obtained from $\beta$ by replacing the strand connecting $a_1,b_1$ (resp. $a_2,b_2$) and the wires connecting their children with a single wire. In other words, $\Psi$ is obtained from $\Phi$ by reducing our two dipoles simultaneously and independently. Then, by construction, $\Psi$ can be obtained from both $\Psi_1,\Psi_2$ by reducing a dipole, i.e. there exist two oriented edges in our directed graph from $\Psi_1$ and $\Psi_2$ to $\Psi$. 

\medskip \noindent
Thus, we have shown that $X$ is locally confluent, and we conclude that $X$ is confluent thanks to Lemma~\ref{lem:Confluent}, as desired.
\end{proof}

\noindent
\textbf{Notation:} In the rest of the article, given two diagrams $\Delta_1=(A_1,\beta_1,B_1)$ and $\Delta_2 = (A_2, \beta_2,B_2)$, we write $\Delta_1= \Delta_2$ if $A_1=A_2$, $B_1=B_2$ and $\beta_1=\beta_2$; and we write $\Delta_1 \equiv \Delta_2$ if $\Delta_1$ and $\Delta_2$ are equipotent, which amounts to saying, according to Proposition~\ref{prop:reduction}, that they have the same reduction.

\subsection{Group structure on braided strand diagrams}\label{section:GroupLaw}

\noindent
In this section, we fix an arboreal semigroup presentation $\mathcal{P}= \langle \mathcal{A} \mid \mathcal{R} \rangle$ and a baseword $w \in \mathcal{A}^+$. Our goal is to define a group law on the set of braided strand diagrams over $(\mathcal{P},w)$ up to equipotence. We mimic the group law defined on Thompson's groups as described in Section~\ref{section:warmup}. For this purpose, the following lemma will be needed:

\begin{lemma}\label{lem:Bigger}
Let $\Delta= (A, \beta,B)$ be a diagram. For every $(\mathcal{P},w)$-forest $B'$ containing $B$ as a prefix, there exist $A'$ and $\beta'$ such that $\Delta$ is equipotent to $(A',\beta',B')$. Similarly, for every $(\mathcal{P},w)$-forest $A''$ containing $A$ as a prefix, there exist $\beta''$ and $B''$ such that $\Delta$ is equipotent to $(A'', \beta'',B'')$. 
\end{lemma}

\noindent
Roughly speaking, the lemma {and its proof} explain how to add dipoles in diagrams, instead of reducing them. This construction will be often used in the sequel.

\begin{proof}[Proof of Lemma \ref{lem:Bigger}.]
Let $A''$ be a $(\mathcal{P},w)$-forest obtained from $A$ by adding all the children of a $\mathcal{P}$-leaf $a$. In the diagram $(A,\beta,B)$, to the $\mathcal{P}$-leaf $a$ of $A$ is associated a wire $\mu$, which ends at some $\mathcal{P}$-leaf $b \in B$. Define the diagram $(A'', \beta'',B'')$ where $B'':= B \cup \{ \text{children of $b$}\}$ and where $\beta''$ is obtained from $\beta$ by replacing the single wire $\mu$ with a collection of wires of size the common number of children of $a,b$ (indeed, they have the same number of children since they have the same $\mathcal{A}$-label) and by adding a strand parallel to these wires connecting $a,b$ (which are $\mathcal{P}$-interior vertices respectively in $A,B$). See Figure~\ref{DipoleJ}. The second assertion of our lemma follows by iteration.
\begin{figure}
\begin{center}
\includegraphics[width=0.5\linewidth]{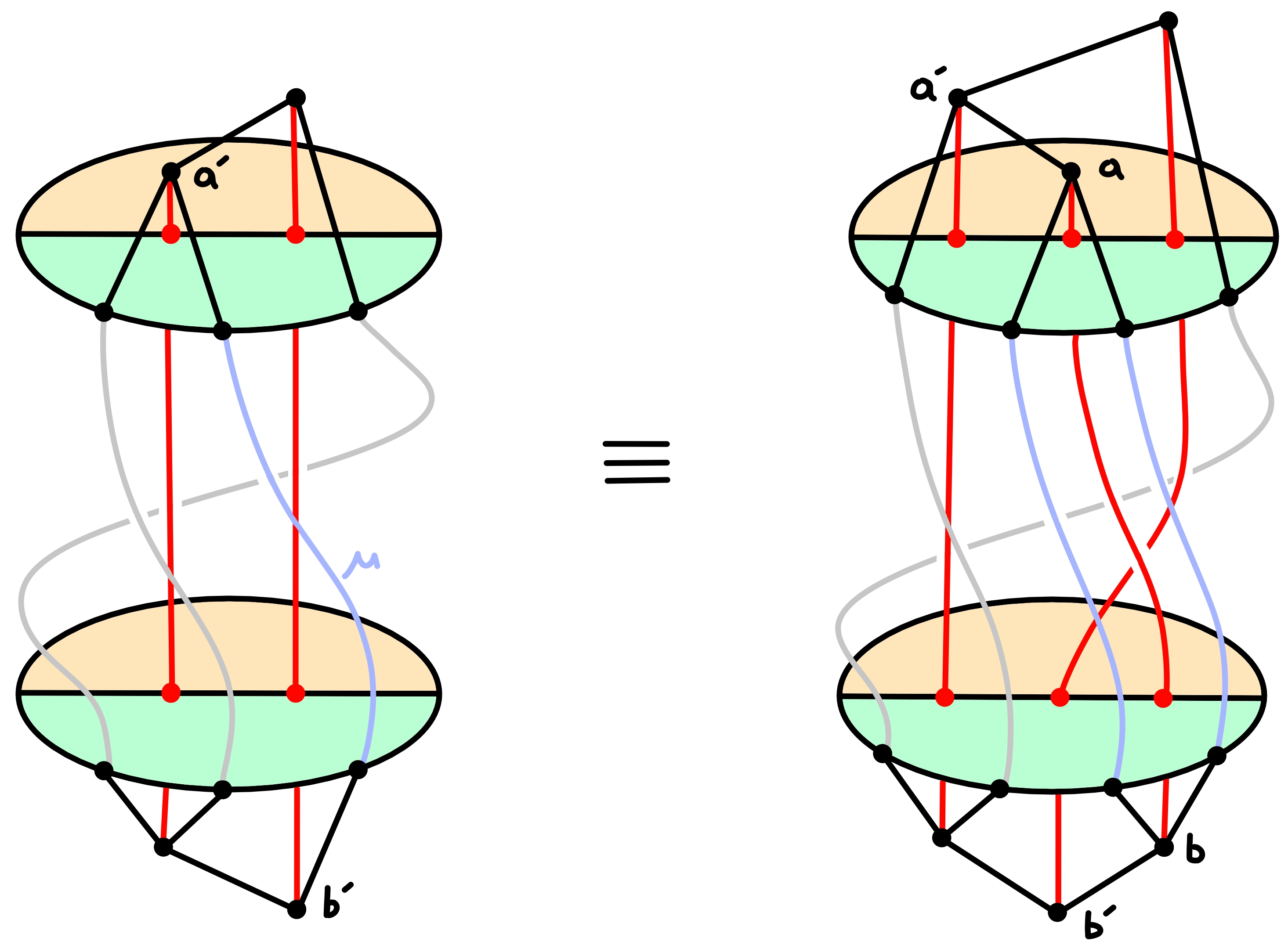}
\caption{Adding a dipole.}
\label{DipoleJ}
\end{center}
\end{figure}

\medskip \noindent
Observe that the map $(R,\tau,S) \mapsto (S,\tau^{-1},R)$, which turns upside down a diagram, sends equipotent diagrams to equipotent diagrams. Therefore, the first assertion of our lemma follows from the second one.
\end{proof}

\noindent
We are now ready to define our product between reduced braided strand diagrams:

\begin{definition}
Let $\Delta_1,\Delta_2$ be two braided strand diagrams over $(\mathcal{P},w)$. As a consequence of Lemma~\ref{lem:Bigger}, there exist $A,B,C$ and $\beta_1,\beta_2$ such that $\Delta_1 \equiv (A, \beta_1, B)$ and $\Delta_2 \equiv (B, \beta_2, C)$. The \emph{concatenation} $\Delta_1 \circ \Delta_2$ is the reduced form of $(A, \beta_2 \beta_1, C)$. 
\end{definition}

\noindent
Of course, it is not clear that our definition does not depend on the choices made in the construction. The next statement shows that this is indeed the case, and, moreover, that the operation we get is a group law.

\begin{prop}\label{prop:GroupLaw}
The concatenation $\circ$ defines a group law on the set of reduced braided strand diagrams over $(\mathcal{P},w)$.
\end{prop}

\noindent
We begin by proving the following observation: 

\begin{lemma}\label{lem:Equal}
If two diagrams $(A,\beta,B)$ and $(A',\beta',B)$ are equipotent, then $A=A'$ and $\beta=\beta'$. Similarly, if two diagrams $(A,\beta,B)$ and $(A,\beta',B')$ are equipotent, then $\beta=\beta'$ and $B=B'$. 
\end{lemma}

\begin{proof}
Fix two equipotent diagrams $(A,\beta,B)$ and $(A',\beta',B)$. Let $(R,\alpha,S)$ denote their common reduction. Observe that $S$ is a prefix of $B$. Because one can obtain $(R,\alpha,S)$ from $(A,\beta,B)$ and $(A',\beta',B)$ by reducing dipoles, it follows that $(A,\beta,B)$ and $(A',\beta',B)$ can be obtained from $(R,\alpha,S)$ by adding dipoles. Adding such dipoles amounts to adding all the children to some $\mathcal{P}$-leaves of $S$. After finitely many iterations we obtain $B$. But, in the two processes, we add the vertices in different orders. Therefore, in order to conclude that $(A,\beta,B)$ and $(A',\beta',B)$ coincide, it suffices to notice that adding two dipoles in one order or another yields the same diagram, which is clear. 

\medskip \noindent
The second assertion of our lemma is symmetric to the first one.
\end{proof}

\begin{proof}[Proof of Proposition \ref{prop:GroupLaw}.]
Let $\Delta_1=(A_1,\beta_1,B_2)$ and $\Delta_2=(A_2,\beta_2,B_2)$ be two reduced diagrams. As a consequence of Lemma~\ref{lem:Bigger}, there exist $A_1,\beta_1,\beta_2',B_2'$ such that $\Delta_1 \equiv (A_1' ,\beta_1', A_2 \cup B_2)$ and $\Delta_2 \equiv (A_2 \cup B_2, \beta_2', B_2')$. Given $A'',\beta_1'',B'',\beta_2'',C''$ such that $\Delta_1 \equiv (A'',\beta_1'',B'')$ and $\Delta_2 \equiv (B'', \beta_2'',C'')$, we want to prove that $(A_1', \beta_2'\beta_1',B_2') \equiv (A'', \beta_2'' \beta_1'',C'')$. This will prove that our product is well-defined.

\medskip \noindent
Because $\Delta_1$ can be obtained from $(A'',\beta_1'',B'')$ by reducing dipoles, necessarily $B''$ contains $B_2$ as a prefix. Similarly, because $\Delta_2$ can be obtained from $(B'',\beta_2'',A'')$ by reducing dipoles, $B''$ has to contain $A_2$ as a prefix. Thus, $B''$ contains the union $A_2 \cup B_2$ as a prefix. 

\begin{claim}\label{claim:ProductWellDefined}
Let $(R,\alpha,S),(S,\beta,T)$ and $(U,\mu,V),(V,\nu,W)$ be four diagrams such that $(R,\alpha,S) \equiv (U,\mu,V)$, such that $(S,\beta,T) \equiv (V,\nu,W)$, and such that $V$ contains $S$ as a prefix. Then $(R, \beta \alpha,T) \equiv (U, \nu \mu, V)$.
\end{claim}

\noindent
We assume that $V$ is obtained from $S$ by adding all the children of a $\mathcal{P}$-leaf $v \in S$, the general case following by iteration. As $v$ is a $\mathcal{P}$-leaf of $S$, it indexes a wire in $(R,\alpha,S)$, which is also labelled by a $\mathcal{P}$-leaf $r$ of $R$. By adding a dipole like in the proof of Lemma~\ref{lem:Bigger}, we obtain a diagram $(R \cup \{\text{children of $r$}\}, \alpha', V)$ equipotent to $(R,\alpha,S)$. It follows from Lemma~\ref{lem:Equal} that this new diagram must coincide with $(U,\beta,V)$. Similarly, as $v$ is a $\mathcal{P}$-leaf of $S$, it indexes a wire in $(S,\beta,T)$, which is also labelled by a $\mathcal{P}$-leaf $t$ of $T$. By adding a dipole like in the proof of Lemma~\ref{lem:Bigger}, we obtain a diagram $(V, \beta', T \cup \{\text{children of $t$}\})$ equipotent to $(S,\beta,T)$. It follows from Lemma~\ref{lem:Equal} that this new diagram must coincide with $(V,\nu,W)$. So $(U,\nu \mu ,V)= (R \cup \{\text{children of $r$}\}, \beta' \alpha' , T \cup \{\text{children of $t$}\})$. In this diagram, the strand connecting $r$ and $t$ belongs to a dipole (coming from the two dipoles added in $(R,\alpha,S)$ and $(S,\beta,T)$), and reducing this dipole yields $(R, \beta \alpha, T)$. See Figure~\ref{Composition} for an illustration. Thus, our claim is proved.
\begin{figure}
\begin{center}
\includegraphics[width=\linewidth]{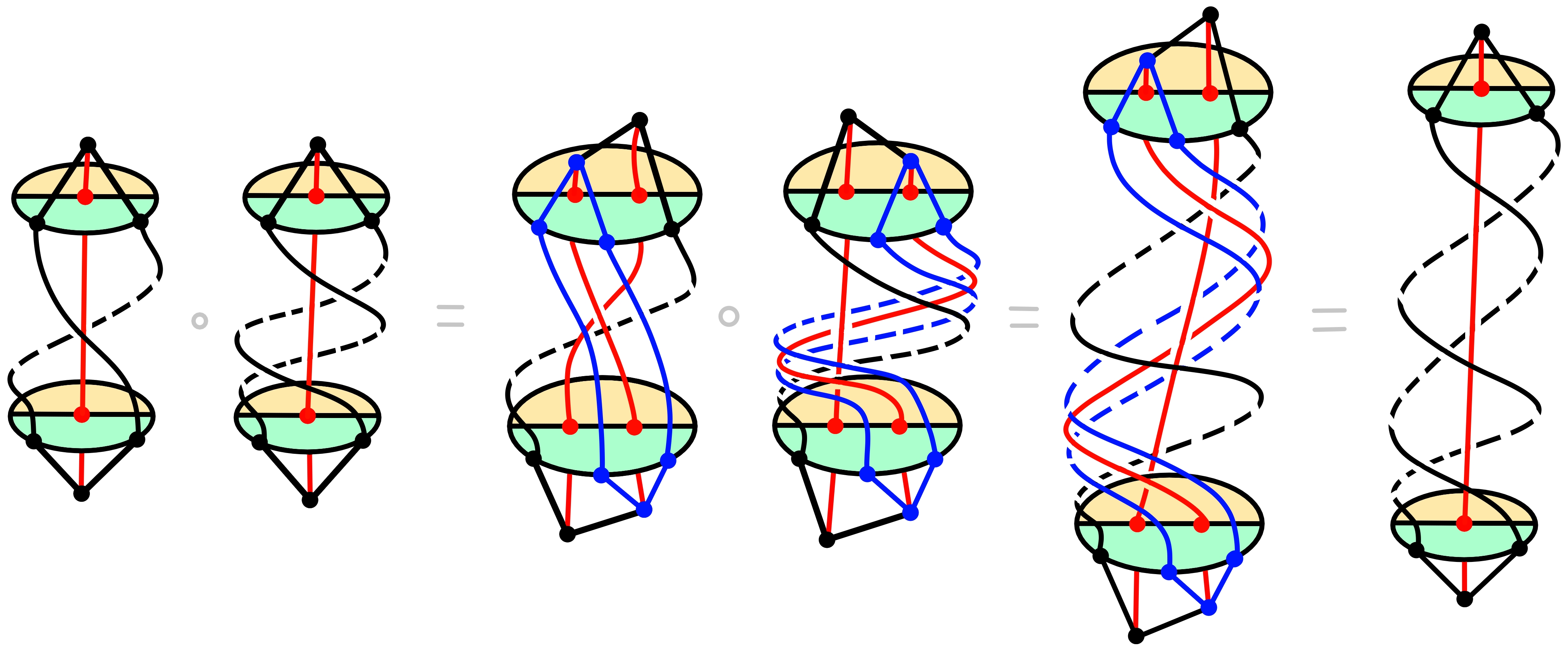}
\caption{Adding a dipole in a concatenation does not modify the result.}
\label{Composition}
\end{center}
\end{figure}

\medskip \noindent
Now, we know that our product is well-defined and that we can choose representatives of diagrams at our best convenience without modifying the result. It remains to show that the concatenation is a group law.

\medskip \noindent
It is clear that $\epsilon(w):=(\eta(w),\mathrm{id},\eta(w))$ is a neutral element. (Recall that $\eta(w)$ denotes the $(\mathcal{P},w)$-forest that has no edges.) Moreover, for every reduced diagram $(A,\beta,B)$, we have
$$(A,\beta,B) \circ (B, \beta^{-1},A) = (A, \mathrm{id},A) \equiv (\eta(w), \mathrm{id}, \eta(w))= \epsilon(w),$$
so every reduced diagram has an inverse (which amounts to reversing upside down the diagram). Finally, let $\Delta_1,\Delta_2,\Delta_3$ be three reduced diagrams. As a consequence of Lemma~\ref{lem:Bigger}, there exist $A,B,C,D$ and $\beta_1,\beta_2,\beta_3$ such that $\Delta_1 \equiv (A,\beta_1,B)$, $\Delta_2 \equiv (B,  \beta_2, C)$, and $\Delta_3 \equiv (C, \beta_3,D)$. Then
$$\Delta_1 \circ (\Delta_2 \circ \Delta_3) \equiv (A, (\beta_3 \beta_2) \beta_1, D) = (A, \beta_3 (\beta_2 \beta_1), D) \equiv (\Delta_1 \circ \Delta_2) \circ \Delta_3.$$
Thus, the concatenation is associative, concluding the proof of our proposition.
\end{proof}

\noindent
Proposition \ref{prop:GroupLaw} allows us to give the following definition:

\begin{definition}
Given an arboreal semigroup presentation $\mathcal{P}=\langle \mathcal{A} \mid \mathcal{R} \rangle$ and a baseword $w \in \mathcal{A}^+$, the \emph{Chambord group} $C(\mathcal{P},w)$ is the set of reduced braided strand diagrams endowed with the concatenation. 
\end{definition}

\noindent
For all $p,q \geq 0$, we say that a braid in $\mathcal{B}_{p,q}$ is \emph{pure} if it can be represented by a homeomorphism fixing pointwise the punctures and the marked points. In other words, every strand (resp. wire) in $\mathscr{D}_{p,q} \times [0,1]$ connects a puncture (resp. marked point) in $\mathscr{D}_{p,q} \times \{0\}$ to the same puncture (resp. marked point) in $\mathscr{D}_{p,q} \times \{1\}$. Observe that, if $(A,\beta,B)$ and $(A',\beta',B')$ are two diagrams such that one is obtained from the other by adding or reducing a dipole, then $\beta$ is pure if and only if $\beta'$ is pure. As a consequence,
$$PC(\mathcal{P},w):= \left\{ (A,\beta,B) \in C(\mathcal{P},w) \mid \beta \text{ pure} \right\} \leq C(\mathcal{P},w)$$
defines a subgroup, which we refer to as the \emph{pure subgroup} of $C(\mathcal{P},w)$.

\subsection{Diagram groups and the forgetful short exact sequence}\label{section:DiagramGroups}

\noindent
In the same way that that the Chambord group $C(\mathcal{P},w)$ is defined from triples $(A,\beta,B)$, it is possible to define the \emph{annular diagram group} $D_a(\mathcal{P},w)$ from triples $(A,\sigma,B)$ where $\sigma$ is a cyclic permutation from the $\mathcal{P}$-leaves of $A$ to the $\mathcal{P}$-leaves of $B$. This group will not be used in the rest of the article, so we leave the details to the reader. The point is that there is a \emph{forgetful map}
$$\mathfrak{f} : \left\{ \begin{array}{ccc} C(\mathcal{P},w) & \to &  D_a(\mathcal{P},w) \\ (A,\beta,B) & \mapsto & (A, \varsigma(\beta),B) \end{array} \right.$$
where, for all $p,q \geq 0$, $\varsigma$ sends every element of $\mathcal{B}_{p,q}$ to the cyclic permutation it induces on the marked points on $\partial \mathscr{D}_{p,q}$. This morphism may not be surjective. In fact, a triple $(A,\sigma,B)$ belongs to the image of $\mathfrak{f}$ if and only if it is \emph{balanced}, i.e. $A$ and $B$ have the same number of vertices. Let $D_a^b(\mathcal{P},w)$ denote the subgroup of $D_a(\mathcal{P},w)$ consisting of balanced triples. Then we have a short exact sequence
\begin{equation}\label{shortexactsequence}
1 \to B_\infty \to C(\mathcal{P},w) \overset{\mathfrak{f}}{\to} D^b_a(\mathcal{P},w) \to 1,
\end{equation}
where $B_\infty$ denotes the subgroup in $C(\mathcal{P},w)$ consisting of the triples of the form $(A,\beta,A)$, which is isomorphic to the group of finitely supported braids on infinitely many strands. 

\medskip \noindent
The annular diagram group $D_a(\mathcal{P},w)$ contains the \emph{planar diagram group} $D_p(\mathcal{P},w)$, consisting of the triples $(A,\mathrm{id},B)$. 

\medskip \noindent
It is worth noticing that $D_a(\mathcal{P},w)$ can be described as a group of partial isometries of $T(\mathcal{P},w)$. More precisely, if $(A,\sigma,B)$ represents an element of $D_a(\mathcal{P},w)$, then it yields an isometry from the complement of the prefix $A \leq T(\mathcal{P},w)$ to the complement of the prefix $B \leq T(\mathcal{P},w)$ by permuting the components according to $\sigma$ via translations of the ambient plane.  

\medskip \noindent
In case $\mathcal{R}$ does not contain relations of the form $u=u$, then $D_p(\mathcal{P},w)$ coincides with the \emph{diagram groups} studied in~\cite{MR1448329, MR1396957}; and $D_a(\mathcal{P},w)$ coincides with the \emph{picture groups} introduced in~\cite{MR1396957} and studied in more details in~\cite{MR2136028}. For instance, if $\mathcal{P}=\langle x \mid x=x^2 \rangle$, then $D_p(\mathcal{P},x)$ is isomorphic to Thompson's group $F$ and $D_a(\mathcal{P},x)$ is isomorphic to Thompson's group $T$. However, these diagram groups are well-defined even if the semigroup presentation $\mathcal{P}$ is not arboreal, providing a larger class of groups.

\medskip \noindent
As we will show, for every locally finite planar tree $A$, the asymptotically rigid mapping class group $\mathfrak{mod}(A)$, whose definition is recalled in Section~\ref{section:ModA} below, is isomorphic to some Chambord group $C(\mathcal{P},w)$. Interestingly, this isomorphism induces an isomorphism between $D_a(\mathcal{P},w)$ and $\mathfrak{mod}(\mathscr{S}(A))$, sending $D_a^b(\mathcal{P},w)$ to $\mathfrak{mod}_f(A)$, and the short exact sequence mentioned above turns out to coincide with the forgetful exact sequence satisfied by $\mathfrak{mod}(A)$ (see Section~\ref{section:ModA} below). More precisely, there exists a commutative diagram

\centerline{\xymatrix{
1 \ar[ddd] \ar[r] & B_\infty \ar[ddd] \ar[r] & C(\mathcal{P},w) \ar[ddd] \ar[r] & D_a^b(\mathcal{P},w) \ar@{^{(}->}[dr] \ar[ddd] \ar[rr] && 1 \ar[ddd] \\ &&&& D_a(\mathcal{P},w) \ar[d] \\ &&&& \mathfrak{mod}(\mathscr{S}(A)) \\ 1 \ar[r] & B_\infty \ar[r] & \mathfrak{mod}(A) \ar[r] & \mathfrak{mod}_f(A) \ar@{^{(}->}[ur] \ar[rr] && 1
}}

\medskip \noindent
where vertical arrows are all isomorphisms. 

\medskip \noindent
Finally, notice that the short exact sequence~(\ref{shortexactsequence}) splits over $D_p(\mathcal{P},w) \cap D_a^b(\mathcal{P},w)$, generalising the observation \cite[Propositions~2.8 and~2.11]{FunarKapoudjian}.

\section{Links with asymptotically rigid mapping class groups}\label{section:MODA}

\subsection{Asymptotically rigid mapping class groups...}\label{section:ModA}

\medskip \noindent
Fix a locally finite tree $A$ embedded into the plane in such a way that its vertex-set is closed and discrete. The \emph{arboreal surface} $\mathscr{S}(A)$ is the oriented planar surface with boundary obtained by thickening $A$ in the plane. We denote by $\mathscr{S}^\sharp(A)$ the punctured arboreal surface obtained from $\mathscr{S}(A)$ by adding a puncture for each vertex of the tree. Following \cite{FunarKapoudjian}, we fix a \emph{rigid structure} on $\mathscr{S}^\sharp(A)$, i.e. a decomposition into \emph{polygons} by means of a family of pairwise non-intersecting arcs whose endpoints are on the boundary of $\mathscr{S}(A)$ such that each polygon contains exactly one vertex of the underlying tree in its interior and such that each arc crosses once and transversely a unique edge of the tree. See for instance Figure~\ref{Dsharp}.
\begin{figure}
\begin{center}
\includegraphics[trim={0 0 16cm 0},clip,width=0.6\linewidth]{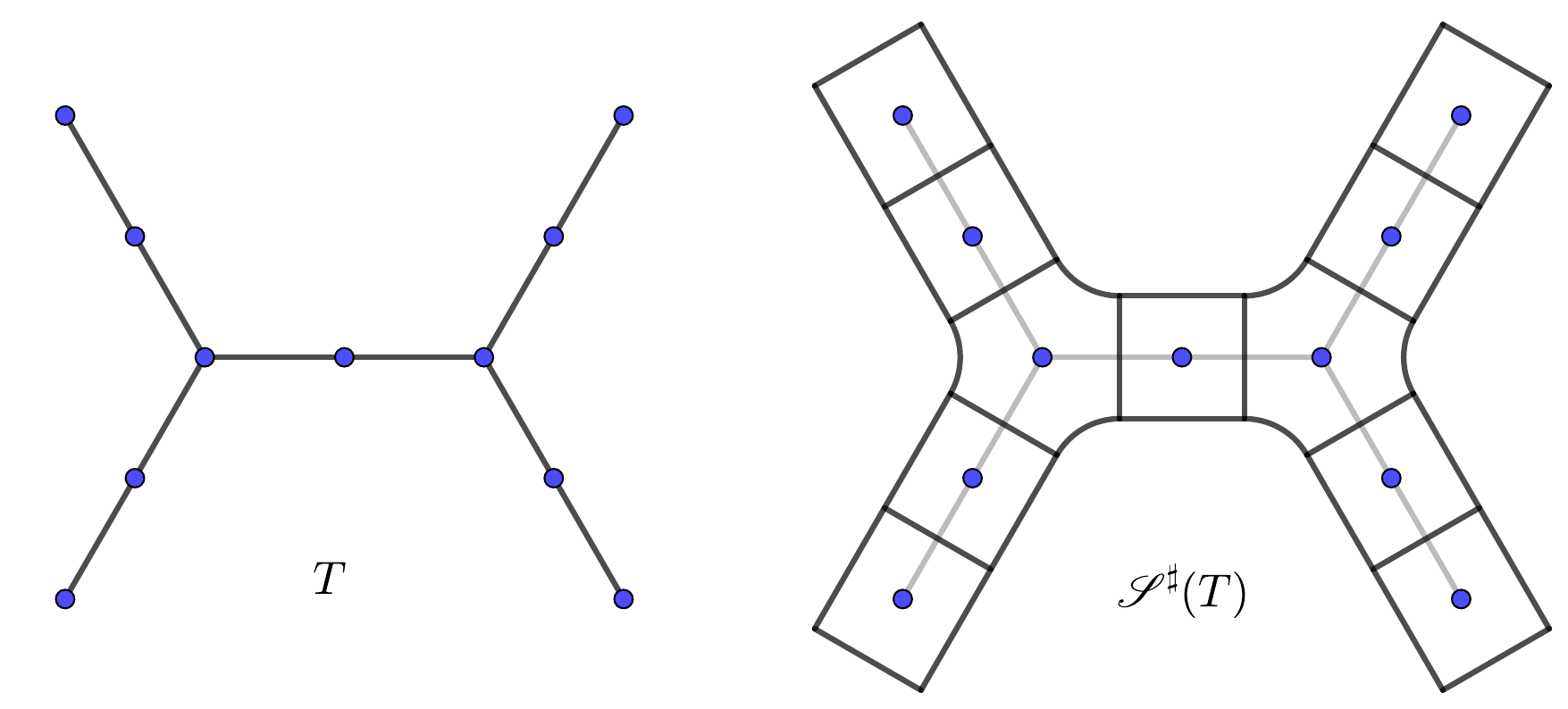}
\caption{Surfaces with rigid structures associated to a simplicial tree.}
\label{Dsharp}
\end{center}
\end{figure}

\medskip \noindent
A subsurface of $\mathscr{S}^\sharp(A)$ is \emph{admissible} if it is a non-empty connected finite union of polygons belonging to the rigid structure. A homeomorphism $\varphi : \mathscr{S}^\sharp(A) \to \mathscr{S}^\sharp(A)$ is \emph{asymptotically rigid} if the following conditions are satisfied:
\begin{itemize}
	\item there exists an admissible subsurface $\Sigma \subset \mathscr{S}^\sharp(A)$ such that $\varphi(\Sigma)$ is also admissible;
	\item the homeomorphism $\varphi$ is \emph{rigid outside $\Sigma$}, i.e. the restriction \[\varphi : \mathscr{S}^\sharp(A) \backslash \Sigma \to \mathscr{S}^\sharp(A) \backslash \varphi( \Sigma)\]
	 respects the rigid structure, mapping polygons to polygons. Such a surface $\Sigma$ is called a \emph{support} for $\varphi$.
\end{itemize}
We denote by $\mathfrak{mod} (A)$ the group of isotopy classes of orientation-preserving asymptotically rigid homeomorphisms of $\mathscr{S}^\sharp(A)$. We emphasize that isotopies have to fix each puncture. 

\medskip \noindent
In the sequel, we refer to the \emph{frontier} $\mathrm{Fr}(\Sigma)$ of an admissible subsurface $\Sigma$ as the union of the arcs defining the rigid structure that are contained in the boundary. Also, a polygon is called \emph{adjacent} to $\Sigma$ if it is not contained in $\Sigma$ but shares an arc with the frontier of~$\Sigma$.

\medskip \noindent
Similarly, one can define asymptotically rigid homeomorphisms of $\mathscr{S}(A)$ and the group $\mathfrak{mod}(\mathscr{S}(A))$. Then, there exists a natural \emph{forgetful map} $\mathfrak{mod}(A) \to \mathfrak{mod}(\mathscr{S}(A))$, which ``forgets'' the punctures of $\mathscr{S}^\sharp(A)$. We denote by $\mathfrak{mod}_f(A)$ the image of $\mathfrak{mod}(A)$ in $\mathfrak{mod}(\mathscr{S}(A))$. The kernel of the forgetful map corresponds to the homeomorphisms $\mathscr{S}^\sharp(A) \to \mathscr{S}^\sharp(A)$ that are the identity outside an admissible subsurface. In other words, we have a short exact sequence
$$1 \to B_\infty \to \mathfrak{mod}(A) \to \mathfrak{mod}_f(A) \to 1,$$
where $B_\infty$ denotes the limit $\bigcup\limits_{\text{$\Sigma$ admissible}} \mathrm{Mod}(\Sigma)$. Notice that $B_\infty$ is isomorphic to the group of finitely supported braids with infinitely many strands. We refer to the previous exact sequence as the \emph{forgetful short exact sequence} satisfied by $\mathfrak{mod}(A)$. 

\medskip \noindent
Another short exact sequence satisfied by $\mathfrak{mod}(A)$, not needed here and referred to as the \emph{arboreal short exact sequence}, is described in \cite[Section~2]{GLU}. Some of these exact sequences are given in Section~\ref{section:Examples}.

\subsection{...as Chambord groups}

\noindent
This section is dedicated to the proof of the following statement, which shows that the asymptotically rigid mapping class groups of our surfaces and the Chambord groups are closely related (see also Remark~\ref{remark:ChambordInMod} below):

\begin{thm}\label{thm:BigChambordMCG}
For every infinite locally finite planar tree $A$, there exist an arboreal semigroup presentation $\mathcal{P}= \langle \mathcal{A} \mid \mathcal{R} \rangle$, a letter $w \in \mathcal{A}$, and an isomorphism $\Phi : C(\mathcal{P},w) \overset{\sim}{\longrightarrow} \mathfrak{mod}(A)$. Moreover, there exists a vertex $x_0 \in A$ such that $\Phi$ induces a bijection
$$\begin{array}{c} \left\{ \left\{ g \in C(\mathcal{P},w) \left| \begin{array}{c} g \equiv (R,\beta,R) \\ \text{ $\beta$ a braid} \end{array} \right. \right\}, \text{ $R$ a $(\mathcal{P},w)$-forest} \right\} \\ \updownarrow \\ \left\{ \mathrm{Mod}(\Sigma, \partial \Sigma) \left| \begin{array}{c} \Sigma \subset \mathscr{S}^\sharp(A) \text{ admissible} \\ \text{containing the puncture $x_0$} \end{array} \right. \right\} \end{array}$$
between braid subgroups in $C(\mathcal{P},w)$ and in $\mathfrak{mod}(A)$, where $\mathrm{Mod}(\Sigma, \partial \Sigma)$ describes the mapping class group of $\Sigma$ preserving setwise the frontier. 
\end{thm}

\noindent
Fix a planar infinite locally finite tree $A$ and a root $x_0 \in A$. For every vertex $x \in A$, let $A(x)$ denote the isomorphism type of the rooted planar tree containing $x$ and all its descendants. Our alphabet is $\mathcal{A}:= \{ A(x) \mid x \in A\}$ and our set of relations $\mathcal{R}$ is
$$\{ A(x)=A(x_1) \cdots A(x_k) \mid x \in A, \text{ the children of $x$ are $x_1,\ldots, x_k$ from left to right} \}.$$
We emphasize that, if $x \in A$ is a leaf, then there is no relation of the form ``$A(x)=\text{something}$'' in $\mathcal{R}$. The semigroup presentation $\mathcal{P}:= \langle \mathcal{A} \mid \mathcal{R} \rangle$ is clearly arboreal. Finally, set $w:=A(x_0)$. Now, Theorem~\ref{thm:BigChambordMCG} is a direct consequence of the following statement:

\begin{prop}\label{prop:AsMCGareChambord}
The groups $\mathfrak{mod}(A)$ and $C(\mathcal{P},w)$ are isomorphic. 
\end{prop}

\begin{proof}
We begin by noticing that the planar rooted trees $A$ and $T(\mathcal{P},w)$ are isomorphic.

\begin{claim}\label{claim:TreeIso}
There exists an isomorphism $\varphi : A \to T(\mathcal{P},w)$ that sends $x_0$ to the root of $T(\mathcal{P},w)$, that preserves the left-right order, and such that, for every vertex $x \in A$, the label of $\varphi(x)$ is $A(x)$. 
\end{claim}

\noindent
We construct the isomorphism inductively by following the instructions below:
\begin{itemize}
	\item First, send $x_0$ to the root of $T(\mathcal{P},w)$.
	\item If $x \in A$ is a vertex such that $\varphi(x)$ is well-defined, let $x_1, \ldots, x_k$ denote its children from left to right. Because the label of $\varphi(x)$ is $A(x)$, by construction of $T(\mathcal{P},w)$ the children $y_1, \ldots, y_k$ (written from left to right) of $\varphi(x)$ are respectively labelled by $A(x_1),\ldots, A(x_k)$. Set $\varphi(x_i)=y_i$ for every $1 \leq i \leq k$.
\end{itemize}
By construction, $\varphi$ defines an isomorphism $A \to T(\mathcal{P},w)$ satisfying the desired properties. In the sequel, we identify $A$ and $T(\mathcal{P},w)$ for convenience. As a consequence, a $(\mathcal{P},w)$-forest, which can be thought of as a prefix of $T(\mathcal{P},w)$, can also be thought of as a subtree of $A$ (containing the root $x_0$). 

\medskip \noindent
From now on, we assume that $A$ and $\mathscr{S}^\sharp(A)$ are drawn on the plane in such a way that, for every vertex $x \in A$, the vertical line passing through $x$ has connected intersection with $\mathscr{S}^\sharp(A)$ and separates the leftmost child of $x$ and its descendants from its other children and their descendants. Also, for each disc $\mathscr{D}_{p,q}$, we assume that each marked point has a fixed small closed neighbourhood in the boundary such that all these arcs are pairwise disjoint and do not meet the horizontal line passing through the punctures. 

\medskip \noindent
Given an admissible subsurface $\Sigma \subset \mathscr{S}^\sharp(A)$ with $p$ punctures and $q$ arcs in its frontier, we define a homeomorphism $\psi_\Sigma : \Sigma \to \mathscr{D}_{p,q}$ as follows: first, draw a line $L_\Sigma$ in $\Sigma$ that starts from the left of the leftmost arc in $\mathrm{Fr}(\Sigma)$, that visits all the punctures in $\Sigma$ in the left-right order, that is monotonic in the horizontal direction, and that ends to the right of the rightmost arc in $\mathrm{Fr}(\Sigma)$; next, define $\psi_\Sigma$ by sending $L_\Sigma$ to the horizontal line passing through the punctures in $\mathscr{D}_{p,q}$, by sending the component of $\Sigma \backslash L_\Sigma$ containing $\mathrm{Fr}(\Sigma)$ to the inferior half-disc of $\mathscr{D}_{p,q}$ in such a way that the midpoints of the arcs in $\mathrm{Fr}(\Sigma)$ are sent to the marked points of $\mathscr{D}_{p,q}$ and that the arcs in $\mathrm{Fr}(\Sigma)$ are sent to the corresponding neighbourhoods of marked points, and by the sending the other component of $\Sigma \backslash L_\Sigma$ to the superior half-disc of $\mathscr{D}_{p,q}$. See Figure~\ref{HomeoDisque}. 
\begin{figure}
\begin{center}
\includegraphics[trim={1cm 4cm 0 5cm},clip,width=0.8\linewidth]{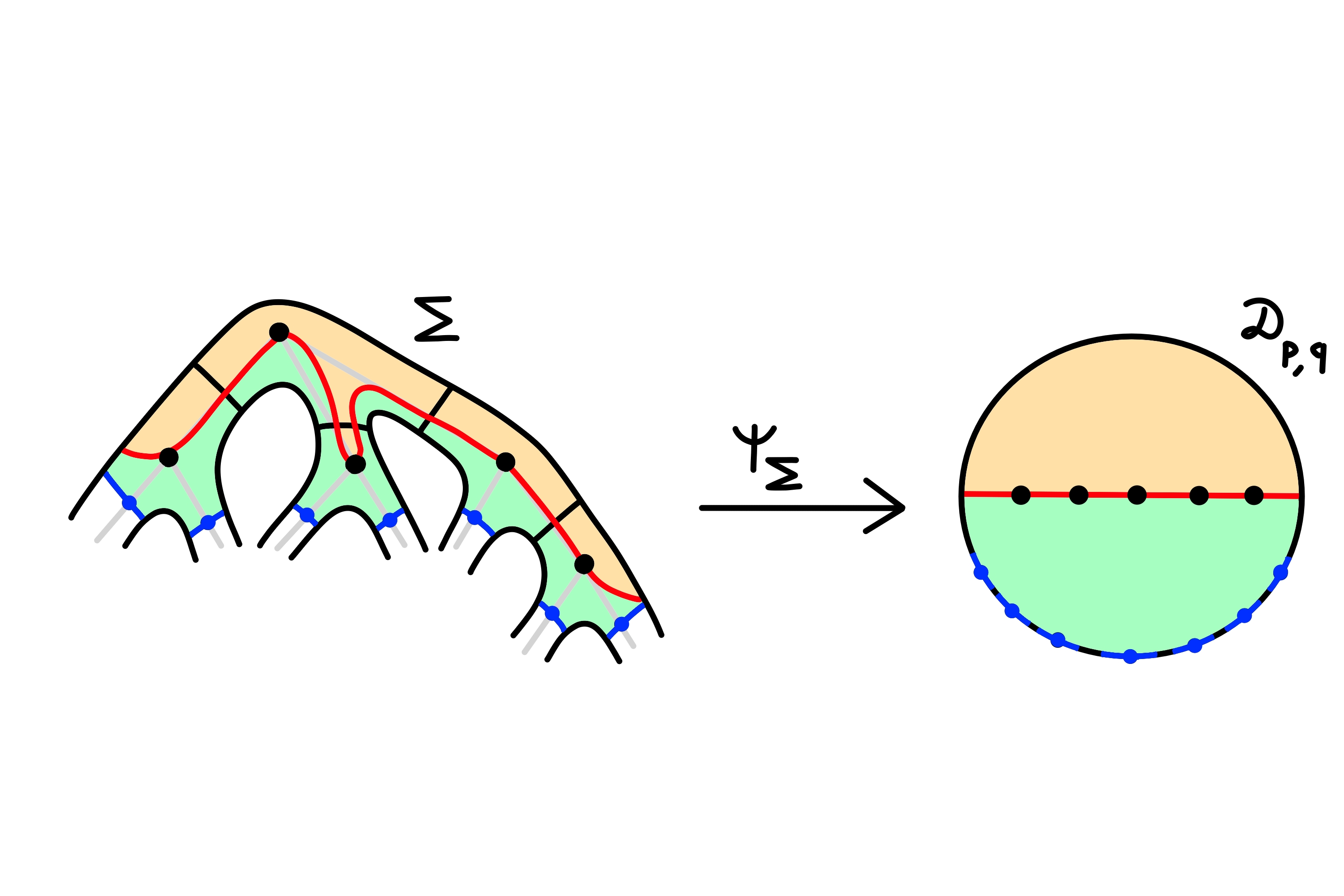}
\caption{Canonical homeomorphism $\psi_\Sigma$.}
\label{HomeoDisque}
\end{center}
\end{figure}

\medskip \noindent
As a consequence of Claim~\ref{claim:TreeIso}, a $(\mathcal{P},w)$-tree $T$ provides a subtree in $A$. We denote by $\Sigma(T)$ the admissible subsurface in $\mathscr{S}^\sharp(A)$ corresponding to this subtree. Observe that the number of punctures in $\Sigma(T)$ coincides with the number of $\mathcal{P}$-interior vertices in $T$ and that the number of arcs in the frontier of $\Sigma(T)$ coincides with the number of $\mathcal{P}$-leaves in $T$. In the sequel, we denote by $\psi_T$ the homeomorphism $\psi_{\Sigma(T)}$. 

\medskip \noindent
Given a diagram $(R,\beta,S)$ distinct from $\epsilon(w)$ (which amounts to requiring that $R$ and $S$ are not single vertices), we define an element $\Phi(R,\beta,S)$ of $\mathfrak{mod}(A)$ as the rigid extension of $\psi^{-1}_{R} \circ \beta^{-1} \circ \psi_{S} : \Sigma(S) \to \Sigma(R)$. Two remarks are necessary here:
\begin{itemize}
	\item In the expression $\psi^{-1}_{R} \circ \beta^{-1} \circ \psi_{S}$, $\beta$ is identified with one of its representatives that sends $\psi_R(\mathrm{Fr}(\Sigma(R)))$ to $\psi_S(\mathrm{Fr}(\Sigma(S)))$. Since our goal is to define an element of $\mathfrak{mod}(A)$, this abuse of notation is justified.
	\item The rigid extension of $\psi^{-1}_{R} \circ \beta^{-1} \circ \psi_{S}$ is possible since the fact that $\beta$ permutes the marked points of $\mathscr{D}_{p,q}$ by preserving the $\mathcal{A}$-labels induced by $R,S$ implies that, if $\psi^{-1}_R \circ \beta^{-1} \circ \psi_S$ sends an arc $\zeta \subset \mathrm{Fr}(\Sigma(S))$ to an arc $\xi \subset \mathrm{Fr}(\Sigma(R))$, then the components of $\mathscr{S}^\sharp(A) \backslash \Sigma(S)$ and $\mathscr{S}^\sharp(A) \backslash \Sigma(S)$ respectively delimited by $\zeta,\xi$ are homeomorphic as planar surfaces.
\end{itemize}
The key observation is the following:

\begin{claim}\label{claim:EquipotentGood}
If $(R,\beta, S)$ and $(R',\beta',S')$ are two equipotent diagrams, then $\Phi(R,\beta,S)$ and $\Phi(R',\beta',S')$ coincide in $\mathfrak{mod}(A)$.
\end{claim}

\noindent
Before turning to the proof of this claim, we need to introduce some vocabulary. Given a puncture $a$ and a set of marked points $M$ in $\mathscr{D}_{p,q}$, an \emph{extremal disc} in $\mathscr{D}_{p,q}$ \emph{spanned} by $a$ and $M$ is a topological closed disc $D \subset \mathscr{D}_{p,q}$ such that:
\begin{itemize}
	\item $D \cap \partial\mathscr{D}_{p,q}$ is connected, contains the neighbourhoods of the points in $M$ in its interior but does not meet any other neighbour;
	\item $a$ is the only puncture in $D$, and it belongs to its interior;
	\item $\partial D \backslash \partial \mathscr{D}_{p,q}$ is an arc that intersects at most twice the horizontal line passing through the punctures.
\end{itemize}
The point is that extremal discs can be used in order to recover extremal polygons in admissible subsurfaces, where a polygon in an admissible subsurface is referred to as \emph{extremal} if it is adjacent to at most one other polygon. More precisely:

\begin{fact}\label{fact:ExtremalDisc}
Let $\Sigma$ be an admissible subsurface with $p$ punctures and $q$ arcs in its frontier, and let $H \subset \Sigma$ be an extremal polygon. If $D \subset \mathscr{D}_{p,q}$ is an extremal disc spanned by $\psi_\Sigma( \text{puncture of $\Sigma$})$ and $\psi_\Sigma( \text{midpoints of arcs in $H \cap \mathrm{Fr}(\Sigma)$})$, then $\psi_{\Sigma}^{-1}(D)$ is homotopic to $H$ in $\Sigma$. 
\end{fact}

\noindent
The statement is clear by construction of $\psi_\Sigma$. We leave the details to the reader and turn to the proof of Claim~\ref{claim:EquipotentGood}.

\medskip \noindent
Assume that $(R',\beta',S')$ is obtained from $(R,\beta,S)$ by reducing a dipole, the general case follows by iteration. So there exist two $\mathcal{P}$-interior vertices $r \in R$ and $s \in S$ such that:
\begin{itemize}
	\item the children of $r,s$ are leaves;
	\item if $\ell_1, \ldots, \ell_k$ denote the children of $r$ written from left to right, then the children of $s$ are $\beta(\ell_1), \ldots, \beta(\ell_k)$ from left to right;
	\item the strand starting from $r$ ends at $s$ and is parallel to the wires connecting the children of $r$ to the children of $s$;
	\item $R'= R\backslash \{\text{children of $r$}\}$, $S'= S \backslash \{\text{children of $s$}\}$, and $\beta' \in \mathcal{B}_{p-1,q-k+1}$ is obtained from $\beta \in \mathcal{B}_{p,q}$ by replacing the strand connecting $a,b$ and the wires connecting their children with a single wire.
\end{itemize}
Necessarily, there exists an extremal disc $D \subset \mathscr{D}_{p,q}$ spanned by the puncture indexed by $r$ and the marked points indexed by the children of $r$ such that $\beta(D)$ is an extremal disc of $\mathscr{D}_{p,q}$ spanned by the puncture indexed by $s$ and the marked points indexed by the children of $s$. See Figure~\ref{ExtremeDisque}. As a consequence of Fact~\ref{fact:ExtremalDisc}, $\psi_R^{-1}(D)$ is homotopic to an extremal polygon $H_R$ of $\Sigma(R)$ and $\psi_S^{-1}(D)$ to an extremal polygon $H_S$ of $\Sigma(S)$. Up to modifying $\beta$ up to isotopy, we assume that $\psi_S^{-1} \circ \beta \circ \psi_R$ sends $H_R$ to $H_S$. 
\begin{figure}
\begin{center}
\includegraphics[trim={2cm 1cm 6cm 2cm},clip,width=0.35\linewidth]{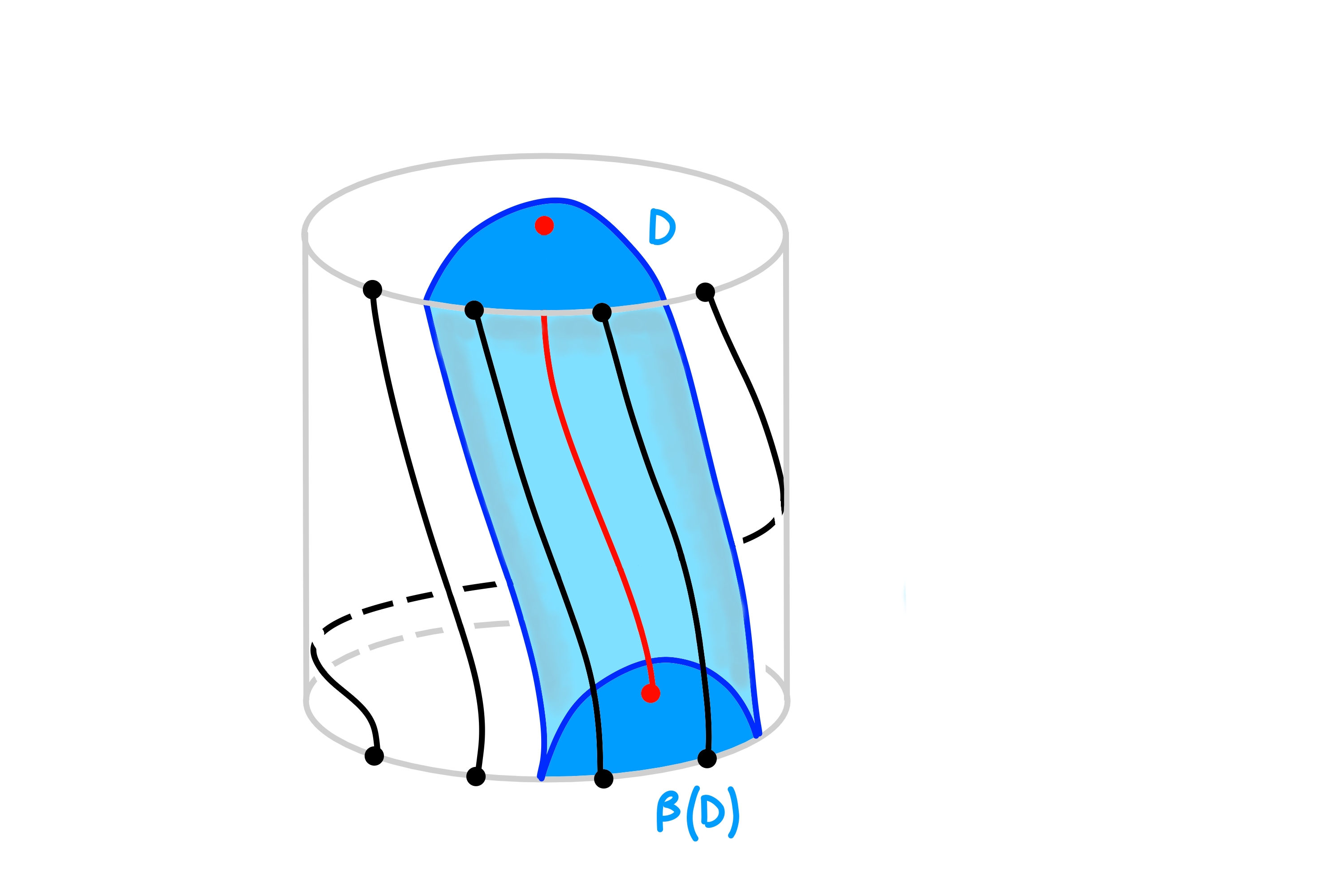}
\caption{Extremal discs associated to a dipole.}
\label{ExtremeDisque}
\end{center}
\end{figure}

\medskip \noindent
Thus, we have proved that $\psi_S^{-1} \circ \beta \circ \psi_R$ sends $\Sigma(R)\backslash H_R$ to $\Sigma(S) \backslash H_S$ and is rigid outside. Observe that $\Sigma(R')= \Sigma(R) \backslash H_R$, that $\Sigma(S')= \Sigma(S) \backslash H_S$, and that the homeomorphisms $\psi_{S'}^{-1} \circ \beta' \circ \psi_{R'}$ and $\psi_S^{-1} \circ \beta \circ \psi_R$ agree on $\Sigma(R')$. See Figure~\ref{Diagramme}. Since $\psi_{S'}^{-1} \circ \beta' \circ \psi_{R'}$ sends $\Sigma(R')$ to $\Sigma(S')$ and is rigid outside, we conclude that $\Phi(R,\beta,S)=\Phi(R',\beta',S')$ as desired. The proof of Claim~\ref{claim:EquipotentGood} is complete.
\begin{figure}
\begin{center}
\includegraphics[width=0.6\linewidth]{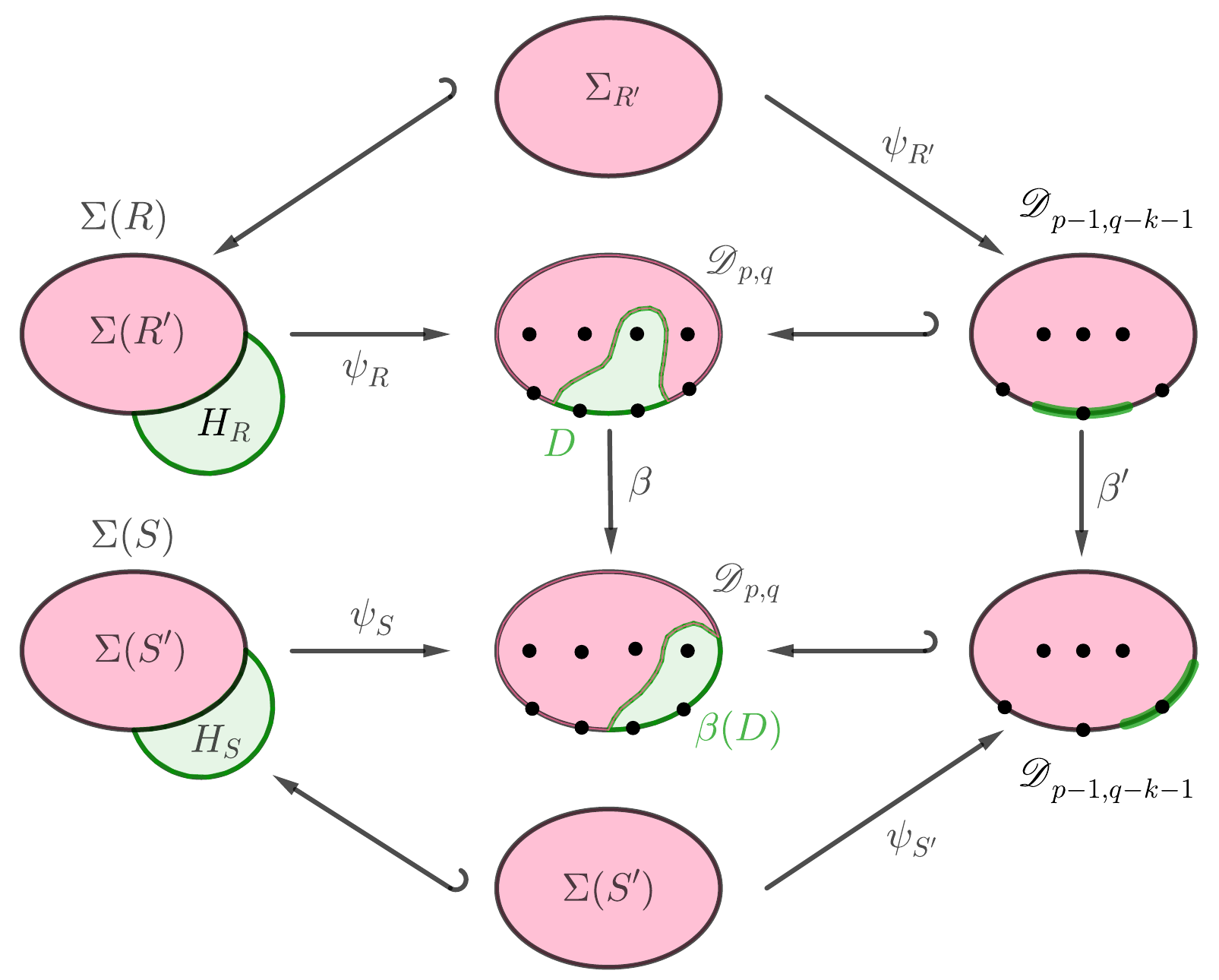}
\caption{Stability of $\Phi$ under dipole reduction.}
\label{Diagramme}
\end{center}
\end{figure}

\medskip \noindent
From now on, we can forget the specific construction of our homeomorphisms $\psi_T$: knowing that Claim~\ref{claim:EquipotentGood} holds will be sufficient in order to deduce that $\Phi$ is an isomorphism.

\begin{claim}
$\Phi$ is a morphism.
\end{claim}

\noindent
Let $\Delta_1,\Delta_2$ be two diagrams. Write $\Delta_1 \equiv (R,\beta_1,S)$ and $\Delta_2 \equiv (S,\beta_2,T)$. Then
\begin{itemize}
	\item $\Phi(\Delta_1)$ sends $\Sigma(S)$ to $\Sigma(R)$ via $\psi_R^{-1} \circ \beta_1^{-1} \circ \psi_S$ and is rigid outside;
	\item $\Phi(\Delta_2)$ sends $\Sigma(T)$ to $\Sigma(S)$ via $\psi_S^{-1} \circ \beta_2^{-1} \circ \psi_T$ and is rigid outside,
\end{itemize}
so $\Phi(\Delta_1) \circ \Phi(\Delta_2)$ sends $\Sigma(T)$ to $\Sigma(R)$ via $\psi_R^{-1} \circ \beta_1^{-1} \beta_2^{-1} \circ \psi_T$ and is rigid outside. Therefore, $\Phi(\Delta_1) \circ \Phi(\Delta_2)= \Phi((R,\beta_2\beta_1,T)) = \Phi(\Delta_1 \circ \Delta_2)$. Since $\Phi(\epsilon(w))= 1$, we conclude that $\Phi$ is a morphism.

\begin{claim}
$\Phi$ is surjective.
\end{claim}

\noindent
Let $\varphi$ be an asymptotically rigid homeomorphism. So it sends an admissible subsurface $\Sigma_1$ to some admissible subsurface $\Sigma_2$ and it is rigid outside. Without loss of generality, we assume that $\Sigma_1$ and $\Sigma_2$ contain the polygon associated to the root $x_0 \in A$. Let $T_1,T_2$ be $(\mathcal{P},w)$-trees such that $\Sigma_1=\Sigma(T_1)$ and $\Sigma_2=\Sigma(T_2)$. Set $\beta:= \psi_{T_2} \circ \varphi_{|\Sigma_1} \circ \psi^{-1}_{T_1}$ and $\Delta:= (T_2,\beta^{-1},T_1)$. Then $\Phi(\Delta)$ sends $\Sigma(T_1)=\Sigma_1$ to $\Sigma(T_2)=\Sigma_2$ via $\psi_{T_2}^{-1} \circ \beta \circ \psi_{T_1}= \varphi_{|\Sigma_1}$ and is rigid outside. We conclude that $\Phi(\Delta)=\varphi$.

\begin{claim}
$\Phi$ is injective.
\end{claim}

\noindent
Let $\Delta = (R,\beta,S)$ be a reduced diagram in the kernel of $\Phi$. We know that $\Phi(\Delta)$ sends $\Sigma(S)$ to $\Sigma(R)$ and is rigid outside. Since $\Phi(\Delta)$ is trivial in $\mathfrak{mod}(A)$, we must have $\Sigma(S)= \Sigma(R)$, hence $R=S$. So $\Phi(\Delta)$ sends $\Sigma(R)$ to $\Sigma(R)$ through $\psi_R^{-1 } \circ \beta^{-1} \circ \psi_R$ and is rigid outside. We must have $\beta=1$, hence $\Delta = (R,1,R) \equiv \epsilon(w)$, as desired.
\end{proof}

\begin{remark}\label{remark:ChambordInMod}
For every arboreal semigroup presentation $\mathcal{P}= \langle \mathcal{A} \mid \mathcal{R} \rangle$ and every letter $w \in \mathcal{A}^+$, one can reproduce the construction of the morphism $\Phi$ from the proof of Proposition~\ref{prop:AsMCGareChambord} in order to get an injective morphism $\Psi : C(\mathcal{P},w) \hookrightarrow \mathfrak{mod}(T(\mathcal{P},w))$. The justifications are the same. However, $\Psi$ may not be surjective. The main obstruction comes from the fact that the punctures of $\mathscr{S}^\sharp(T(\mathcal{P},w))$ (which correspond to the vertices of $T(\mathcal{P},w)$) are indexed by $\mathcal{A}$ and that a homeomorphism in the image of $\Psi$ has to preserve this colouring for all but finitely many punctures. For instance, if $\mathcal{P}=\langle a,b,c \mid a=bc, b=b, c=c \rangle$ and $w=a$, then $T(\mathcal{P},w)$ is a bi-infinite line but the image of $\Psi$ does not contain a homeomorphism permuting the two ends of $\mathscr{S}^\sharp(T(\mathcal{P},w))$. If $w$ is not a letter but a word of arbitrary length, observe that $C(\mathcal{P},w)$ is isomorphic to $C(\mathcal{P}',\epsilon)$ where $\epsilon$ is a new letter that does not belong to $\mathcal{A}$ and where $\mathcal{P}'= \langle \mathcal{A} \cup \{\epsilon\} \mid \mathcal{R} \cup \{ \epsilon=w\} \rangle$. It follows that every Chambord group embeds into some asymptotically rigid mapping class group. 
\end{remark}

\subsection{Examples}\label{section:Examples}

\begin{figure}
\begin{center}
\includegraphics[width=0.8\linewidth]{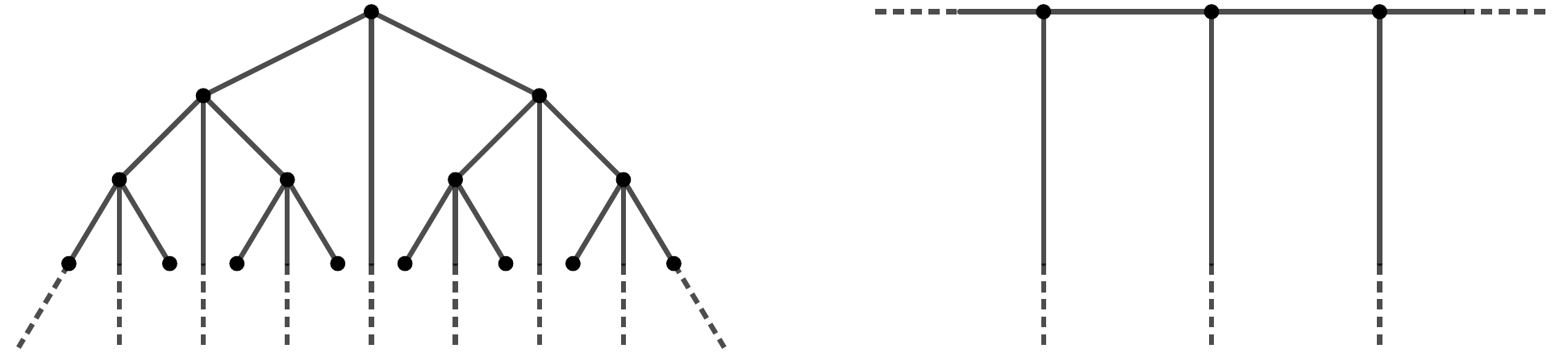}
\caption{The tree defining the Greenberg-Sergiescu group and the braided lamplighter group.}
\label{Graphes}
\end{center}
\end{figure}

\noindent
In this section, we record some examples of interest of asymptotically rigid mapping class groups and we apply Proposition~\ref{prop:AsMCGareChambord} in order to describe them as Chambord groups.

\paragraph{Braided Higman-Thompson groups.} Let $A_{n,m}$ denote the tree having one vertex of degree $m$ while all its other vertices have degree $n+1$, where $m \geq 1$ and $n \geq 2$. In \cite{GLU}, we refer to $\mathrm{br}T_{n,m}:= \mathfrak{mod}(A_{n,m})$ as the \emph{braided Higman-Thompson group}, the terminology being justified by the short exact sequence
$$1 \to B_\infty \to \mathrm{br}T_{n,m} \to T_{n,m} \to 1$$
satisfied by $\mathrm{br}T_{n,m}$, where $B_\infty$ denotes the group of finitely supported braids on infinitely many strands and where $T_{n,m}$ denotes Higman's generalisation of Thompson's group $T$. This family generalises the \emph{Ptolemy-Thompson group} $\mathrm{br}T_{2,3}$ introduced in \cite{FunarKapoudjian}. By applying Proposition~\ref{prop:AsMCGareChambord}, one sees that $\mathrm{br}T_{n,m}$ is isomorphic to the Chambord group $C(\mathcal{P}_{n,m},a)$ where $\mathcal{P}_{n,m}= \langle a,b \mid a = b^m, b=b^n \rangle$ if $m \neq n+1$ and $\mathcal{P}_{n,n+1}= \langle a \mid a=a^n \rangle$ otherwise. 

\paragraph{Braided Houghton groups.} Let $R_n$ denote the tree that is the union of $n$ infinite rays sharing a common origin. Following \cite{Degenhardt, FunarHoughton}, we refer to the finite-index subgroup $\mathrm{br}H_n \leq \mathfrak{mod}(R_n)$ that preserves the ends at infinity of $\mathscr{S}^\sharp (R_n)$ as the \emph{braided Houghton group}, the terminology being justified by the short exact sequence
$$1 \to PB_\infty \to \mathrm{br}H_n \to H_n \to 1$$
satisfied by $\mathrm{br}H_n$, where $PB_\infty$ is the pure subgroup of $B_\infty$ and where $H_n$ is the \emph{Houghton group} introduced in \cite{Houghton}. By applying Proposition~\ref{prop:AsMCGareChambord}, one sees that $\mathfrak{mod}(R_n)$ is isomorphic to the Chambord group $C(\mathcal{P}_n,a)$ where $\mathcal{P}_n:= \langle a,b \mid a=b^n, b=b \rangle$. The image of $\mathrm{br}H_n$ in $C(\mathcal{P}_n,a)$ corresponds to the diagrams in which the wires are not twisted around the cylinder. As a consequence, $\mathrm{br}H_n$ itself can be described as a Chambord group, namely $C(\mathcal{Q}_n,a)$ where $\mathcal{Q}_n:= \langle a,b_1, \ldots, b_n \mid a=b_1 \cdots b_n, b_1=b_1, \ldots, b_n=b_n \rangle$.

\paragraph{Greenberg-Sergiescu group.} Let $A$ denote the tree obtained from the rooted binary tree by gluing an infinite descending ray at each vertex. See Figure~\ref{Graphes}. The group $\mathfrak{mod}(A)$ was introduced in \cite{MR1090167}, where it is proved that it satisfies the short exact sequence
$$1 \to B_\infty \to \mathfrak{mod}(A) \to F' \to 1,$$
where $F'$ denotes the derived subgroup of $F$. Observe that, as a consequence, $\mathfrak{mod}(A)$ is not finitely generated. By applying Proposition~\ref{prop:AsMCGareChambord}, one sees that $\mathfrak{mod}(A)$ is isomorphic to the Chambord group $C(\mathcal{P},a)$ where $\mathcal{P}= \langle a,b \mid a=aba, b=b \rangle$.

\paragraph{Braided lamplighter group.} Let $A$ denote the tree obtained from a bi-infinite horizontal line by gluing an infinite descending ray at each vertex. See Figure~\ref{Graphes}. In \cite{GLU}, we refer to $\mathrm{br}\mathcal{L}:= \mathfrak{mod}(A)$ as the \emph{braided lamplighter group}, the terminology being justified by the short exact sequence
$$1 \to B_\infty \to \mathrm{br} \mathcal{L} \to \mathcal{L}:= \mathbb{Z} \wr \mathbb{Z} \to 1$$
satisfied by $\mathrm{br}\mathcal{L}$, where $\mathcal{L}$ is a \emph{lamplighter group}, defined as the \emph{wreath product} $\mathbb{Z} \wr \mathbb{Z}:= \left( \bigoplus_\mathbb{Z} \mathbb{Z} \right) \rtimes \mathbb{Z}$ where $\mathbb{Z}$ acts on the direct product by shifting the coordinates. As proved below by Proposition~\ref{prop:notFP}, $\mathrm{br} \mathcal{L}$ is finitely generated but not finitely presented. By applying Proposition~\ref{prop:AsMCGareChambord}, one sees that $\mathrm{br}\mathcal{L}$ is isomorphic to the Chambord group $C(\mathcal{P},a)$ where $\mathcal{P}= \langle a,b,c,p \mid a=bpc, b=bp, c=pc, p=p \rangle$.

\section{Nonpositive curvature}\label{section:CCC}

\subsection{Warm up: cubulation of Thompson's group $V$}\label{section:CCV}

\begin{figure}
\begin{center}
\includegraphics[width=0.8\linewidth]{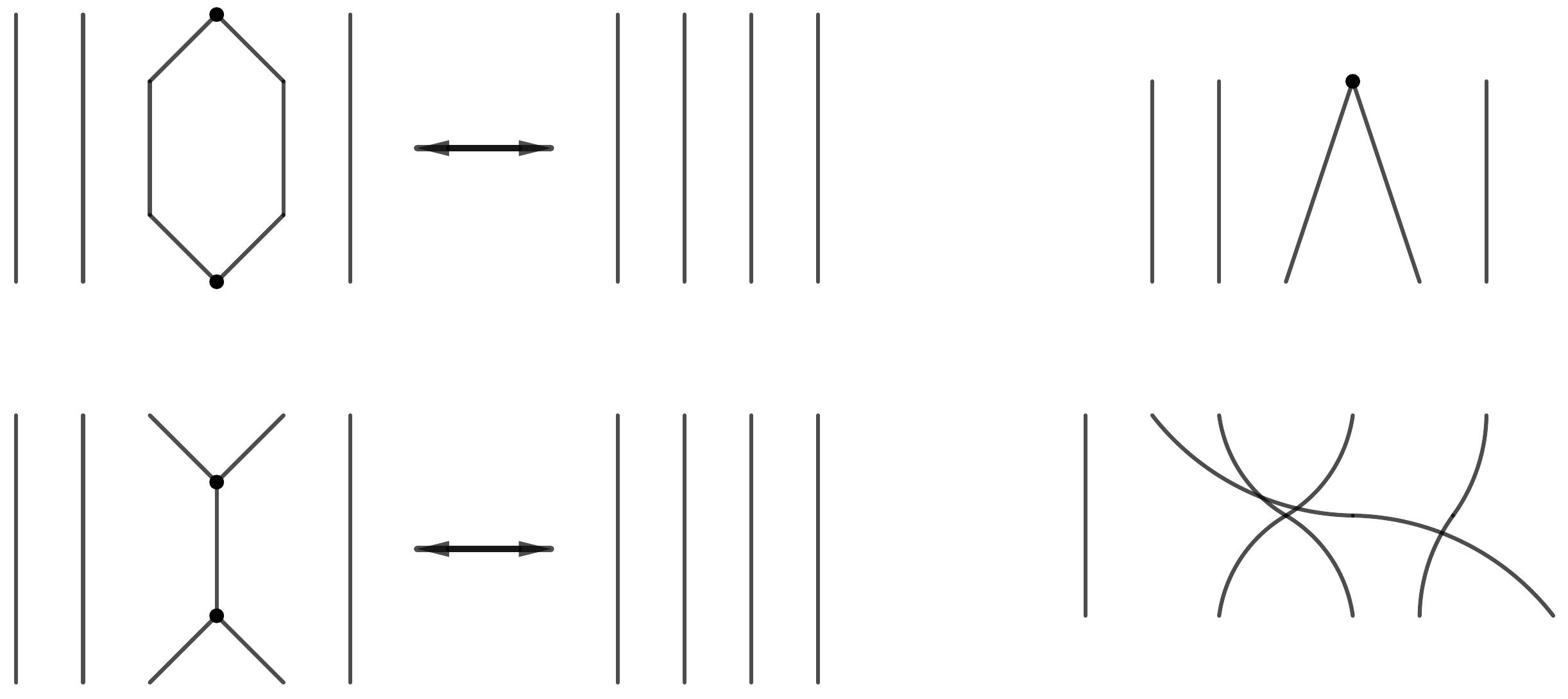}
\caption{On the left, the cancellations in the groupoid; on the right, illustration of the generators.}
\label{Cancel}
\end{center}
\end{figure}
\begin{figure}
\begin{center}
\includegraphics[width=0.9\linewidth]{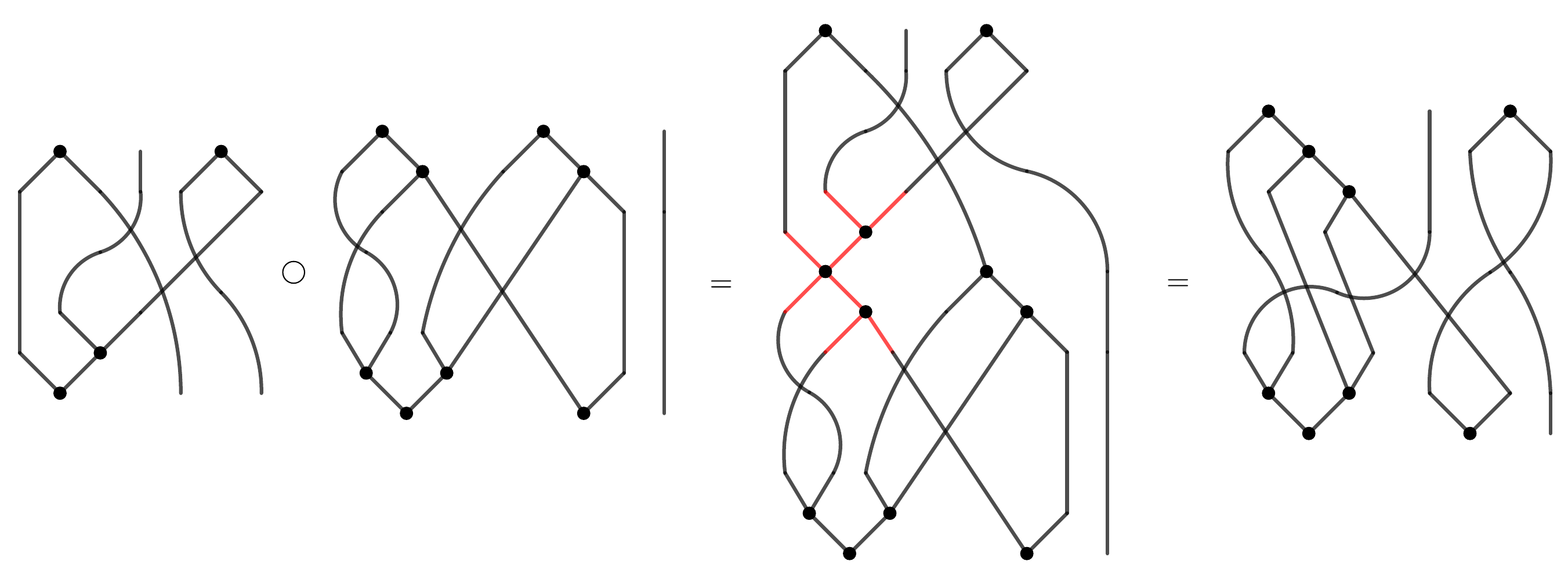}
\caption{A product between two open diagrams.}
\label{Produit}
\end{center}
\end{figure}
\noindent
As described in Section~\ref{section:warmup}, the elements of Thompson's groups $F \leq T \leq V$ can be represented by strand diagrams, a formalism we adapted in order to define our Chambord groups. It is worth noticing that strand diagrams can be similarly defined when rooted binary trees are replaced with rooted binary forests. We refer to such an \emph{open diagram} as an \emph{$(m,n)$-diagram} if its top forest has $m$ trees and its bottom forest $n$ trees. So the elements of Thompson's groups are represented by $(1,1)$-diagrams. 

\medskip \noindent
This more general family of diagrams is naturally endowed with a groupoid structure. Indeed, given a $(p,q)$-diagram $A$ and $(q,r)$-diagram $B$, we define the $(p,r)$-diagram $A \circ B$ by gluing $B$ below $A$ and by applying the cancellations illustrated by Figure~\ref{Cancel} as much as possible. See Figure~\ref{Produit}. Then, the corresponding Thompson group coincides with an isotropy group of the groupoid. 

\medskip \noindent
For Thompson's group $F$ (i.e.\ when the strands are not permuted), the groupoid of open diagrams is generated by diagrams with no edges in their bottom forests and with top forests that are unions of isolated vertices with a single caret. See the right hand side of Figure~\ref{Cancel}. In the Cayley graph of our groupoid constructed from this generating set, the $(1,\ast)$-diagrams define a connected component $X(F)$ on which $F$ naturally acts freely by left-multiplication. In other words, the vertices of $X(F)$ are the $(1,\ast)$-diagrams and one passes from one vertex to another by adding or removing a caret in the diagram. The key observation is that this graph is a \emph{median graph}, i.e. the one-skeleton of a CAT(0) cube complex (sometimes referred to as the \emph{Stein-Farley complex}). 

\medskip \noindent
For Thompson's group $V$, the groupoid is generated by the previous generators plus the diagrams with no edges in their top and bottom forests (so, roughly speaking, they just permute the strands), referred to as \emph{permutation diagrams}. See Figure~\ref{Cancel}. The Cayley graph of the groupoid with respect to this generating set is not a median graph. This is because right-multiplying a given diagram with permutation diagrams produces pairwise adjacent vertices, but there is no complete subgraphs of size $\geq 3$ in median graphs. Nevertheless, one can ``collapse'' these complete subgraphs to vertices in order to get a median graph; which amounts to considering, instead of the Cayley graph, the (connected component containing the $(1,\ast)$-diagrams of the) Schreier graph $X(V)$ with respect to the subgroups given by the permutation diagrams. Concretely, the vertices of $X(V)$ are the classes of $(1,\ast)$-diagrams up to right-multiplication by permutation diagrams and the edges connect two classes whenever there are two representatives such that one can be obtained from the other by right-multiplying by a generator that is not a permutation diagram (i.e.\ we essentially add or remove a caret at the bottom of the diagram). See Figure~\ref{complexe}.

\medskip \noindent
Thus, one obtains an action of Thompson's group $V$ on some CAT(0) cube complex. Details of the construction can be found in \cite{MR1978047, MR2136028}. Our goal, in the rest of the section, is to adapt this strategy to our Chambord groups. Unfortunately, there does not seem to be a natural groupoid structure here, but $(1,\ast)$-diagrams can be defined as well as a left product by diagrams representing elements of our Chambord group.
\begin{figure}
\begin{center}
\includegraphics[width=0.8\linewidth]{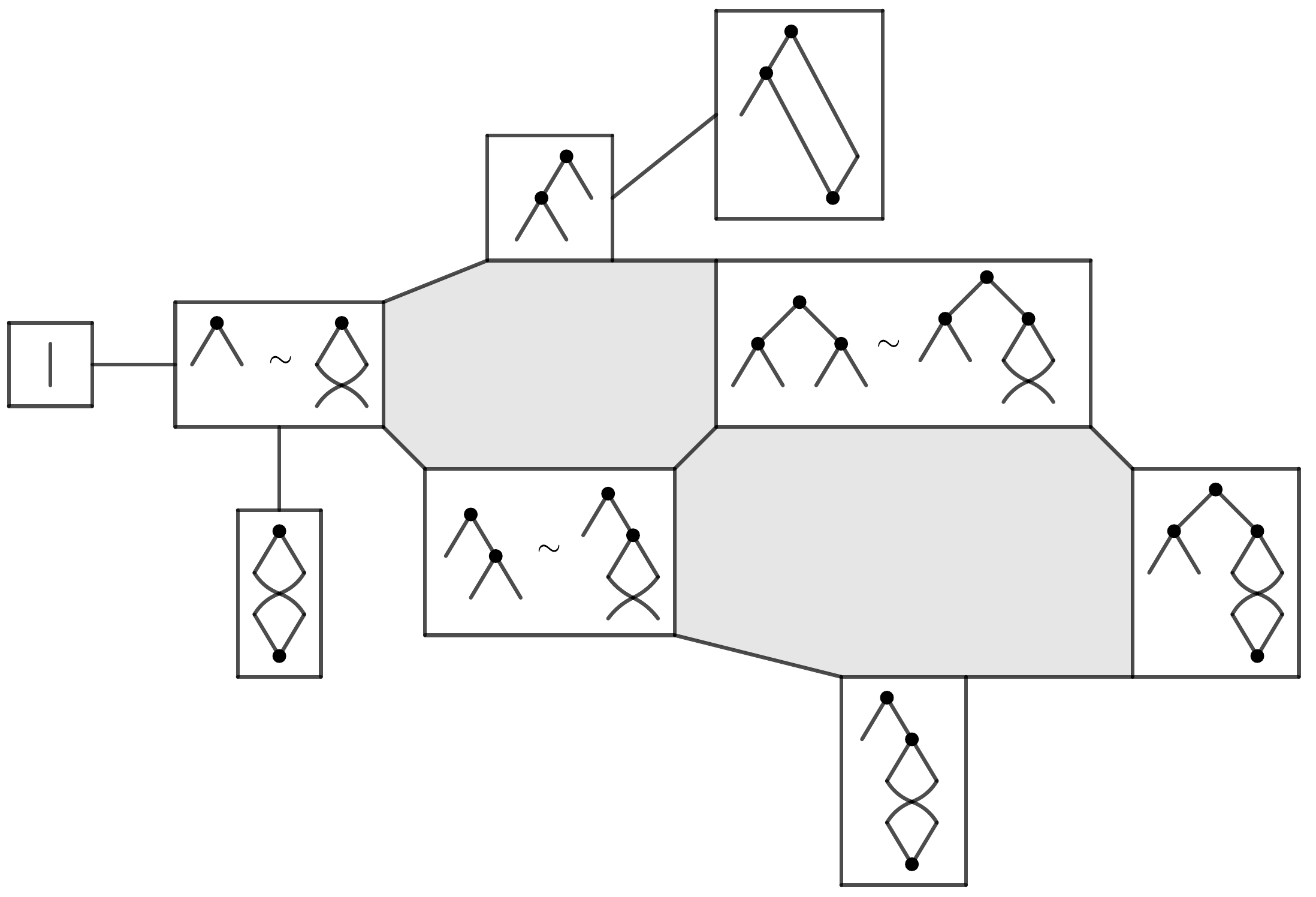}
\caption{A piece of the CAT(0) cube complex on which $V$ acts.}
\label{complexe}
\end{center}
\end{figure}

\subsection{Open braided strand diagrams}

\noindent
In all this section, we fix an arboreal semigroup presentation $\mathcal{P}= \langle \mathcal{A} \mid \mathcal{R} \rangle$ and a baseword $w \in \mathcal{A}^+$. Given a $\mathcal{P}$-forest $F$, we denote by $\iota(F)$ its set of $\mathcal{P}$-interior points, $\lambda(F)$ its set of $\mathcal{P}$-leaves, and $i(F):= |\iota(F)|$, $\ell(F):=|\lambda(F)|$. 

\begin{definition}
An \emph{open braided strand diagram $(A,\beta,\sigma,B)$} is the data of
\begin{itemize}
	\item two $\mathcal{P}$-forests $A,B$ satisfying $i(A) \geq i(B)$ and $\ell(A)=\ell(B)$;
	\item a braid $\beta \in \mathcal{B}_{i(A),\ell(A)}$ that induces a bijection $\lambda(A) \to \lambda(B)$ preserving the $\mathcal{A}$-labelling;
	\item an injective map $\sigma : \iota(B) \hookrightarrow \iota(A)$ preserving the left-right-orders. 
	
\end{itemize}
If $u \in \mathcal{A}^+$ (resp. $v \in \mathcal{A}^+$) denotes the word obtained by reading the labels of the roots of $A$ (resp. $B$) from left to right, one says that $(A,\beta,\sigma,B)$ is a \emph{$(u,v)$-diagram}. If we do not want to specify $v$, one says that it is a \emph{$(u,\ast)$-diagram}.
\end{definition}

\noindent
Similarly to braided strand diagrams (see Definition~\ref{def:Diag}), here we think of the marked points of $\mathscr{D}_{i(A),\ell(A)}$ as being indexed from left to right by the $\mathcal{P}$-leaves of both $A$ and $B$. Therefore, the braid $\beta \in \mathcal{B}_{i(A),\ell(A)}$, which induces a permutation on the marked points, naturally defines a bijection from the $\mathcal{P}$-leaves of $A$ to the $\mathcal{P}$-leaves of $B$, and we require this bijection to preserve the $\mathcal{A}$-labelling. 

\medskip \noindent
Graphically, we think of $(A, \beta, \sigma,B)$ as follows. We draw a copy of $A$ above a cylinder $\mathscr{D}_{i(A),\ell(A)} \times [0,1]$ and we connect its $\mathcal{P}$-leaves (resp. its $\mathcal{P}$-interior vertices) to the marked points (resp. the punctures) of $\mathscr{D}_{i(A),\ell(A)} \times \{0\}$ from left to right; similarly, we draw an inverted copy of $B$ below the cylinder, we connect its $\mathcal{P}$-leaves to the marked points of $\mathscr{D}_{i(A),\ell(A)} \times \{1\}$ from left to right, and we connect its $\mathcal{P}$-interior vertices to the punctures of $\mathscr{D}_{i(A),\ell(A)} \times \{1\}$ according to $\sigma$. Now, following the given braid $\beta$, we connect the marked points of $\mathscr{D}_{i(A),\ell(A)} \times \{0\}$ to the marked points of $\mathscr{D}_{i(A),\ell(A)} \times \{1\}$ with \emph{wires} in $\partial \mathscr{D}_{i(A),\ell(A)} \times [0,1]$ and we connect the punctures of $\mathscr{D}_{i(A),\ell(A)} \times \{0\}$ to the punctures in $\mathscr{D}_{i(A),\ell(A)} \times \{1\}$ with \emph{strands} in the interior of the cylinder $\mathscr{D}_{i(A),\ell(A)} \times [0,1]$. See Figure~\ref{Iso}.
\begin{figure}
\begin{center}
\includegraphics[trim={2.5cm 1.5cm 2.5cm 1cm},clip,width=0.5\linewidth]{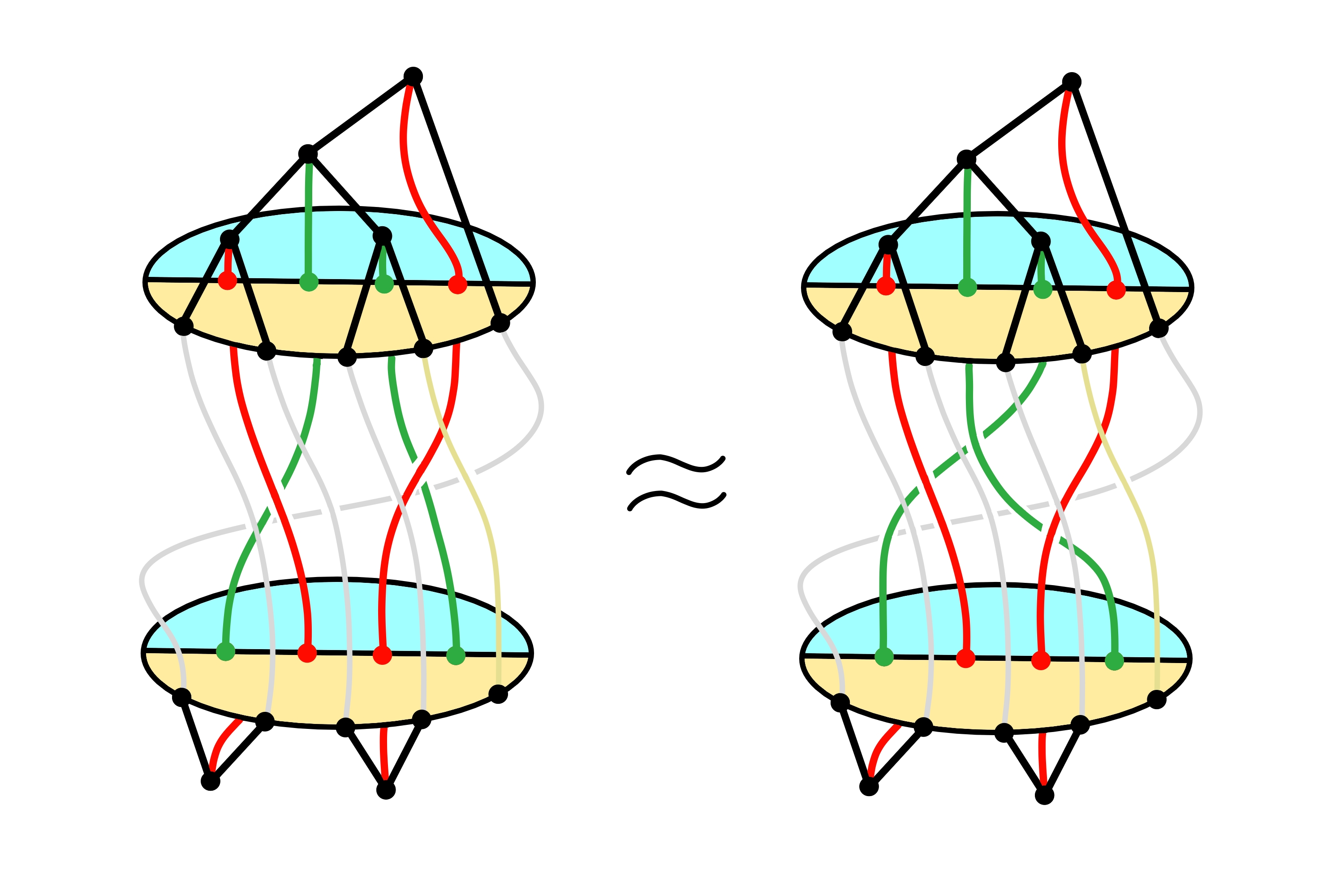}
\caption{Two isocephalese open braided strand diagrams over $\mathcal{P}= \langle x \mid x=x^2 \rangle$. The yellow wire and green strands are free.}
\label{Iso}
\end{center}
\end{figure}

\medskip \noindent
The major difference with (closed) braided strand diagrams is that a strand starting from a $\mathcal{P}$-interior vertex of the top forest may not be connected to a  vertex of the bottom tree. We refer to such a strand as a \emph{free strand}. We also refer to a wire that ends at a vertex of the bottom tree without ancestors as a \emph{free wire}. By extension, we refer to the correspond punctures and marked points as \emph{free punctures} and \emph{free marked points}.

\begin{definition}
Let $\Delta_1= (A,\beta,\sigma,B)$ be an open braided strand diagram. A second diagram $\Delta_2$ is \emph{isocephalese} to $\Delta_1$, written $\Delta_1 \approx \Delta_2$, if $\Delta_2 = (A,\beta' \beta,\sigma,B)$ for some $\beta'$ represented by a strand diagram such that the non-free strands are vertical and always in front of the free strands. Equivalently, $\beta'$ fixes (up to isotopy) a collection of straight arcs in the disc starting from the non-free punctures and ending on the lower part of the boundary.
\end{definition}

\noindent
Figure \ref{Iso} gives two examples of isocephalese open diagrams. Observe that, on open diagrams, being isocephalese defines an equivalence relation. 

\medskip \noindent
\emph{Dipoles}, \emph{dipole reductions}, and \emph{equipotence} are defined in the same way as for braided strand diagrams. See Definitions~\ref{def:Dipoles} and~\ref{def:equipotent}. By reproducing the proof of Proposition~\ref{prop:reduction} word by word, one also shows that an open braided strand diagram is equipotent to a unique reduced diagram, allowing us to define the \emph{reduction} of an open diagram.
Note that free strands and free wires are never part of a dipole. 

\medskip \noindent
In the rest of the article, we are not interested in all the open diagrams but only in those that come from closed diagrams. More formally:

\begin{definition}
A \emph{closable open $(\mathcal{P},w)$-diagram}, or \emph{clopen} for short, is an open $(w,\ast)$-diagram $(A,\beta,\sigma,B)$ such that there exists a $(\mathcal{P},w)$-diagram $(A,\beta, B')$ - which we consider as an open  $(\mathcal{P},w)$-diagram $(A,\beta,\sigma',B')$ by letting $\sigma'$ be the bijection induced by $\beta$ - such that $B$ can be obtained from $B'$ by removing $\mathcal{P}$-interior vertices and where $\sigma,\sigma'$ agree on the $\mathcal{P}$-interior vertices of $B$.
\end{definition}

\noindent
Graphically, a clopen $(\mathcal{P},w)$-diagram is just a (closed) $(\mathcal{P},w)$-diagram from which one removes $\mathcal{P}$-interior vertices in the bottom forest. Figure~\ref{Clopen} provides examples of open diagrams that are not closable. Observe that an open diagram obtained from a clopen diagram by adding or removing dipoles is still closable. Also, if two open diagrams are isocephalese and if one of them is closable, then the other one must be closable as well. 
\begin{figure}
\begin{center}
\includegraphics[width=0.7\linewidth]{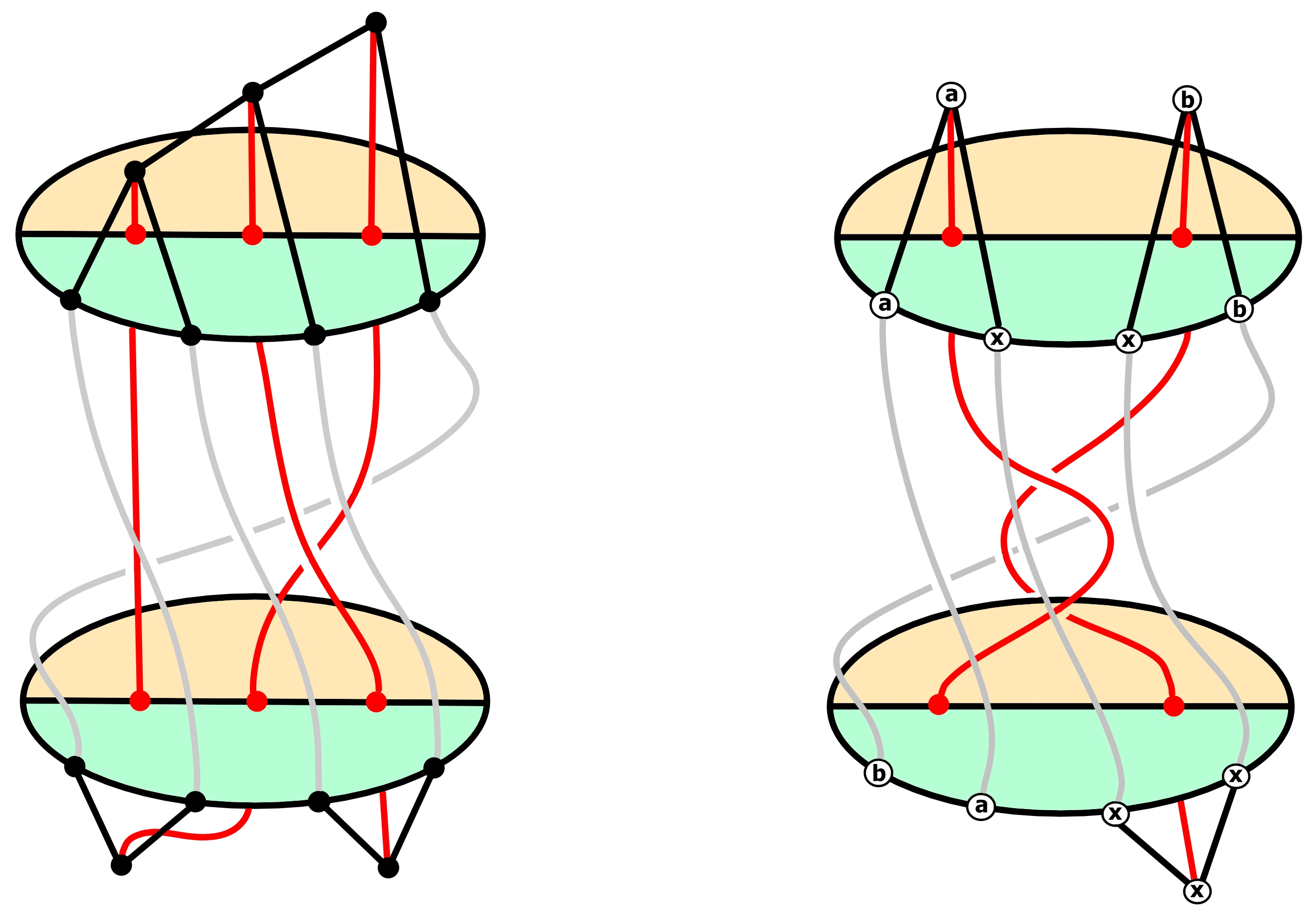}
\caption{Open diagrams over $\langle x \mid x=x^2 \rangle$ and $\langle a,b,x \mid a=ax, b=xb, x=x^2 \rangle$ that are not closable.}
\label{Clopen}
\end{center}
\end{figure}

\subsection{Cube complex and action}

\noindent
In all this section, we fix an arboreal semigroup presentation $\mathcal{P}= \langle \mathcal{A} \mid \mathcal{R} \rangle$ and a baseword $w \in \mathcal{A}^+$. Our goal is to construct a cube complex on which the Chambord group acts naturally. 

\begin{definition}\label{def:MCC}
Let $M(\mathcal{P},w)$ be the cube complex
\begin{itemize}
	\item whose vertices are the classes of clopen $(\mathcal{P},w)$-diagrams given by the equivalence relation generated by being equipotent or isocephalese;
	\item whose edges link two diagrams $(A,\beta,\sigma,B)$ and $(A,\beta,\sigma',B')$ when $B'$ is obtained from $B$ by removing a root and when $\sigma,\sigma'$ agree on the $\mathcal{P}$-interior vertices of $B'$;
	\item whose $k$-cubes fill in all the subgraphs in the one-skeleton isomorphic to the one-skeleton of a $k$-cube, $k \geq 2$.
\end{itemize}
If $(A,\beta,\sigma,B)$ is an open $(\mathcal{P},w)$-diagram, we denote by $[A,\beta,\sigma,B]$ the corresponding vertex of the cube complex $M(\mathcal{P},w)$.
\end{definition}

\noindent
Definition~\ref{def:MCC} gives a short description of our cube complex, but we need more precise characterisations of its vertices and edges. First, we observe that two open diagrams define the same vertex in $M(\mathcal{P},w)$ if and only if their reductions are isocephalese.

\begin{lemma}\label{lem:Classes}
Let $\Phi,\Psi$ be two clopen $(\mathcal{P},w)$-diagrams. If $[\Phi]=[\Psi]$ then there exist two isocephalese diagrams $\Phi',\Psi'$ respectively obtained from $\Phi,\Psi$ by reducing dipoles.
\end{lemma}

\begin{proof}
Our lemma is a straightforward consequence of the following observation, where the arrow $\to$ means that the right diagram is obtained from the left diagram by reducing a dipole. The first two points give us information on the commutation between the reduction operations and being isocephalese, and the last two show local confluence properties. 

\begin{fact}\label{fact:CommutationRules}
Let $A,B,C$ be three clopen $(\mathcal{P},w)$-diagrams. The following assertions hold:
\begin{itemize}
	\item If $A \approx B \to C$, then there exists some $B'$ such that $A \to B' \approx C$.
	\item If $A \leftarrow B \approx C$, then there exists some $B'$ such that $A \approx B' \leftarrow C$.
	\item If $A \neq C$ and $A \to B \leftarrow C$, then there exists some $B'$ such that $A \leftarrow B' \to C$.
	\item If $A \neq C$ and $A \leftarrow B \to C$, then there exists some $B'$ such that $A \to B' \leftarrow C$. 
\end{itemize}
\end{fact}

\noindent
Indeed, if $[\Phi]= [\Psi]$, then there exists a sequence $\Delta_0=\Phi, \Delta_1, \ldots, \Delta_{k-1},\Delta_k=\Psi$ such that, for every $0 \leq i \leq k-1$, $\Delta_i \approx \Delta_{i+1}$ or $\Delta_i \to \Delta_{i+1}$ or $\Delta_i \leftarrow \Delta_{i+1}$. By applying the commutation rules given by Fact~\ref{fact:CommutationRules}, we find a new sequence $\Xi_0= \Phi, \Xi_1, \ldots, \Xi_{r-1},\Xi_r=\Psi$ and two indices $0 \leq p \leq q \leq r$ such that $\Delta_i \to \Delta_{i+1}$ for every $0 \leq i \leq p-1$, $\Delta_i \approx \Delta_{i+1}$ for every $p \leq i \leq q-1$, and $\Delta_i \leftarrow \Delta_{i+1}$ for every $q \leq i \leq r-1$. Setting $\Phi':= \Xi_p$ and $\Psi':= \Xi_q$ yields the desired diagrams.
\end{proof}

\begin{cor}\label{cor:Classes}
Let $\Phi,\Psi$ be two clopen $(\mathcal{P},w)$-diagrams. Then $[\Phi]=[\Psi]$ if and only if the reductions of $\Phi$ and $\Psi$ are isocephalese.
\end{cor}

\begin{proof}
If $[\Phi]=[\Psi]$ then, according to Lemma~\ref{lem:Classes}, there exist two isocephalese diagrams $\Phi',\Psi'$ respectively obtained from $\Phi,\Psi$ by reducing dipoles, i.e. $\Phi \to \cdots \to \Phi' \approx \Psi' \leftarrow \cdots \leftarrow \Psi$. Let $\Phi''$ denote the reduction of $\Phi'$. Because $\Phi'' \leftarrow \cdots \leftarrow \Phi' \approx \Psi'$, it follows from Fact~\ref{fact:CommutationRules} that there exists some $\Psi''$ such that $\Phi'' \approx \Psi'' \leftarrow \cdots \leftarrow \Psi'$. Observe that if, among two isocephalese diagrams, one of them is reduced, then the other must be reduced as well. Consequently, $\Psi''$ has to be reduced. But we know that $\Psi''$ can be obtained from $\Psi$ by reducing dipoles, so $\Psi''$ coincides with the reduction of $\Psi$. Thus, we have proved that the reductions $\Phi'',\Psi''$ of $\Phi,\Psi$ are isocephalese. The converse is clear.
\end{proof}

\noindent
Next, we identify a few elementary operations on open diagrams that allow us to pass from one vertex to one of its neighbours. Given an open $(\mathcal{P},w)$-diagram $\Phi = (A,\beta,\sigma,B)$, we say that a diagram $\Psi$ is:
\begin{itemize}
	\item a \emph{top expansion} of $\Phi$ if there exist a $\mathcal{P}$-leaf $a \in A$, corresponding to a free wire, such that $\Psi$ is obtained by adding the $k$ children of $a$ to $A$, by replacing the isolated vertex $b$ of $B$ attached by a wire to $a$ with a collection of $k$ isolated vertices, by replacing the wire connecting $a$ and $b$ with $k$ wires connecting the children of $a$ to the $k$ new isolated vertices of $B$ from left to right, and by adding a {free} strand starting from $a$ parallel to the previous wires;
	\item a \emph{bottom expansion} of $\Phi$ if $\Psi=(A,\beta,\sigma',B')$ such that $B$ can be obtained from $B'$ by removing a $\mathcal{P}$-interior vertex and such that $\sigma,\sigma'$ agree on the $\mathcal{P}$-interior vertices of~$B$; 
	\item a \emph{top contraction} of $\Phi$ if $\Phi$ is a top expansion of $\Psi$;  
	\item a \emph{bottom contraction} of $\Phi$ if $\Phi$ is a top expansion of $\Psi$.
\end{itemize}

We refer to Figure~\ref{Op} for an illustration of these elementary operations. Notice that an open diagram obtained from a clopen diagram by a top expansion or a bottom contraction is necessarily closable. However, the same observation may not hold for top contractions and bottom expansions. Note also that a top contraction can be obtained as a reduction of a suitable bottom extension.
\begin{figure}
\begin{center}
\includegraphics[width=\linewidth]{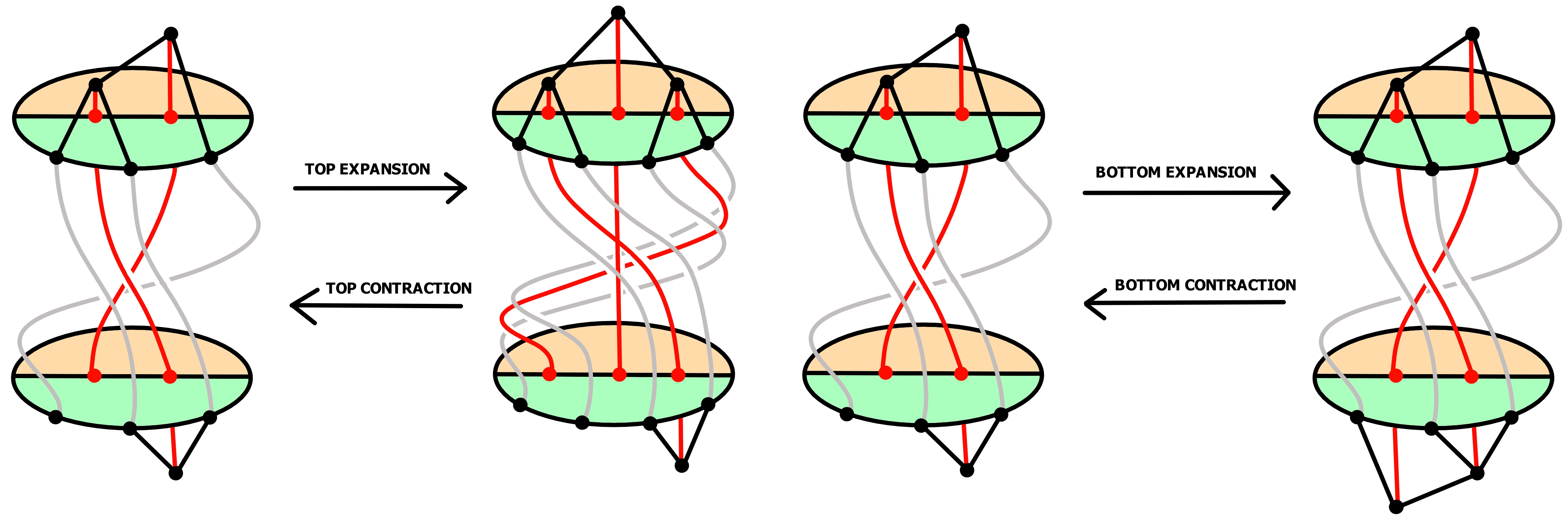}
\caption{Expansions and contractions of open diagrams.}
\label{Op}
\end{center}
\end{figure}

\begin{lemma}\label{lem:Adjacent}
Let $\Phi,\Psi$ be two reduced clopen $(\mathcal{P},w)$-diagrams. Assume that $[\Phi]$ and $[\Psi]$ are adjacent vertices in $M(\mathcal{P},w)$. Then exactly one of the following assertions occurs:
\begin{itemize}
	\item $\Psi$ is isocephalese to a top expansion of $\Phi$;
	\item $\Psi$ is isocephalese to a top contraction of $\Phi'$ for some $\Phi'$ isocephalese to $\Phi$;
	\item $\Psi$ is isocephalese to a bottom contraction of $\Phi$;
	\item $\Psi$ is isocephalese to a bottom expansion of $\Phi'$ for some $\Phi'$ isocephalese to $\Phi$.
\end{itemize}
\end{lemma}

\begin{proof}
First, assume that there exist clopen $(\mathcal{P},w)$-diagrams $(A,\beta,\sigma,B)$ and $(A,\beta,\sigma',B')$ such that $B'$ is obtained from $B$ by removing a $\mathcal{P}$-interior vertex and such that $[\Phi]=[A,\beta,\sigma,B]$ and $[\Psi]=[A,\beta,\sigma',B']$. As a consequence of Corollary~\ref{cor:Classes}, there exists some $\Phi'$ (resp. $\Psi'$) isocephalese to $\Phi$ (resp. $\Psi$) that is the reduction of $(A,\beta,\sigma,B)$ (resp. $(A,\beta,\sigma',B')$).
Let $b \in B$ denote the vertex that does not belong to $B'$ and write $\Phi'=(R, \alpha, \varepsilon,S)$. We distinguish two cases.

\medskip \noindent
Assume that $b \in S$. 
In other words, we add dipoles to $\Phi'$ to get $(A,\beta,\sigma,B)$, we remove the vertex $b$ (which does not belong to these dipoles), and we reduce the diagram to get $\Psi'$. Clearly, the reduction kills all the dipoles added to $\Phi'$, so $\Psi'=(R,\alpha,\varsigma,S\backslash \{b\})$. In other words, $\Psi'$ is the bottom contraction of $\Phi'$ over $b$. 

\medskip \noindent
Assume that $b \notin S$. 
In other words, we add dipoles to $\Phi'$ to get $(A,\beta,\sigma,B)$, we remove the bottom vertex $b$ of a dipole, and we reduce the diagram to get $\Psi'$. Clearly, the reduction kills all the dipoles added to $\Phi'$ but the one containing $b$. So passing from $\Phi'$ to $\Psi'$ amounts to adding a single dipole and removing its bottom vertex, i.e. $\Psi'$ is a top expansion of $\Phi'$. 

\medskip \noindent
Thus, we have proved that $\Psi$ is isocephalese to $\Psi'$, which is either a bottom contraction or a top expansion of $\Phi' \approx \Phi$. By symmetry, it follows that, if our clopen $(\mathcal{P},w)$-diagrams $(A,\beta,\sigma,B)$ and $(A,\beta,\sigma',B')$ are such that $[\Phi]=[A,\beta,\sigma',B']$ and $[\Psi]=[A,\beta,\sigma,B]$, then $\Psi$ is isocephalese to $\Psi'$, which is either a bottom expansion or a top contraction of~$\Phi' \approx \Phi$. 

\medskip \noindent
In order to conclude the proof of our lemma, it suffices to observe that applying a top expansion or a bottom contraction commutes with the relation of being isocephalese. More precisely, if $\Theta$ is an open diagram, if $p$ (resp. $q$) is a vertex in its top forest (resp. bottom forest), and if $\Theta'$ is a diagram isocephalese to $\Theta$, then the diagrams obtained from $\Theta$ and $\Theta'$ by applying a top expansion over $p$ (resp. a bottom contraction over $q$) are isocephalese.
\end{proof}

\noindent
Now, we want to describe how the Chambord group $C(\mathcal{P},w)$ acts on $M(\mathcal{P},w)$. Let $(R,\alpha,S) \in C(\mathcal{P},w)$ be an element and $[A,\beta,\sigma,B] \in M(\mathcal{P},w)$ a vertex. By adding dipoles, we find two diagrams $(R',\alpha',S') \equiv (R, \alpha,S)$ and $(A',\beta',\sigma',B') \equiv (A,\beta,\sigma,B)$ such that $S'=A'$. Then we set
$$(R, \alpha, S) \cdot [A,\beta,\sigma,B] = [R', \beta' \alpha', \sigma',B'].$$
Of course, we need to justify that the vertex thus defined does not depend on the choices we made and we need to verify that such a definition yields an action of $C(\mathcal{P},w)$ on $M(\mathcal{P},w)$ by automorphisms. This can be done by following almost word for word the arguments from Section~\ref{section:GroupLaw}, where we showed that the concatenation between elements of $C(\mathcal{P},w)$ is well-defined and yields a group law. Some explanation is given below, but we refer to Section~\ref{section:GroupLaw} when the arguments are identical.

\medskip \noindent
\textbf{A well-defined...} First, given our group element $(R,\alpha,S) \in C(\mathcal{P},w)$ and our vertex $[A,\beta,\sigma,B] \in M(\mathcal{P},w)$, the diagrams $(R',\alpha',S')$ and $(A',\beta',\sigma',B')$ can be constructed by adding dipoles like in Lemma~\ref{lem:Bigger}. By following Claim~\ref{claim:ProductWellDefined}, we know that adding more dipoles than needed leads to dipoles in $(R',\beta'\alpha',\sigma,B')$ so that it does not modify the result up to equipotence. Moreover, because $(A',\beta',\sigma',B')$ is closable, $(R',\beta',\alpha',\sigma',B')$ must be a clopen $(\mathcal{P},w)$-diagram. Finally, it is clear that replacing $(A,\beta,\sigma,B)$ with a diagram isocephalese to it yields a diagram isocephalese to our result $(R',\beta'\alpha',\sigma',B')$. Thus, the vertex $[R',\beta' \alpha',\sigma',B']$ is well-defined and it does not depend on the choices we made.

\medskip \noindent
\textbf{...action...} Given two elements $a,b \in C(\mathcal{P},w)$ and a vertex $x \in M(\mathcal{P},w)$, we can write $a \equiv (A, \alpha,B)$, $b\equiv (B,\beta,C)$, and $x=[C, \gamma,\sigma,D]$, and we have
$$a \cdot (b \cdot x) = [A, (\gamma \cdot \beta) \cdot \alpha, \sigma, D] = [A, \gamma \cdot (\beta \alpha), \sigma, D] = (ab) \cdot x.$$
Moreover, the neutral element $\epsilon(w)$ of $C(\mathcal{P},w)$ acts trivially since
$$\epsilon(w) \cdot x = (C,\mathrm{id},C) \cdot [C,\gamma,\sigma,D]= [C,\gamma,\sigma,D]=x.$$
Thus, we have defined an action by permutations on the vertices of $M(\mathcal{P},w)$.

\medskip \noindent
\textbf{...by automorphisms.} In order to show that our action on the vertices of $M(\mathcal{P},w)$ extends to an action on the whole cube complex by automorphisms, it suffices to verify that adjacency between vertices is preserved. So let $x,y \in M(\mathcal{P},w)$ be two adjacent vertices and let $g \in C(\mathcal{P},w)$ be an element. By construction of $M(\mathcal{P},w)$, we can write, up to switching $x$ and $y$, $x=[A,\beta,\mu,P]$ and $y=[A,\beta,\nu,Q]$ where $Q$ is obtained from $P$ by removing a $\mathcal{P}$-interior vertex (which has to be a root). By adding dipoles, we can write $g \equiv (R,\alpha,A')$, $x=[A',\beta',\mu',P']$, and $y=[A',\beta',\nu',Q']$; and the key observation is that $Q'$ can still be obtained from $P'$ by removing a $\mathcal{P}$-interior vertex. We have $g \cdot x = [R, \beta'\alpha,\mu',P']$ and $g \cdot y = [R,\beta' \alpha,\nu', Q']$, so $g \cdot x$ and $g \cdot y$ are still adjacent in $M(\mathcal{P},w)$, as desired. 

\medskip \noindent
Thus, we have constructed an action of the Chambord group $C(\mathcal{P},w)$ on the cube complex $M(\mathcal{P},w)$ by automorphisms. We conclude this section by recording a few properties of this action.

\begin{lemma}\label{lem:Translate}
Every vertex $x$ in $M(\mathcal{P},w)$ admits a $C(\mathcal{P},w)$-translate that is of the form $[T, \mathrm{id}, \emptyset, \eta(m)]$, where $\emptyset$ is the empty map and where $m$ is the word obtained by reading the $\mathcal{A}$-labels from left to right of the roots of the bottom forest in any representative of~$x$. 
\end{lemma}

\noindent
We recall that, given a word $m \in \mathcal{A}^+$, $\eta(m)$ denotes the $(\mathcal{P},m)$-forest with no edges. In other words, $\eta(m)$ is a collection of isolated vertices such that the word $m$ is obtained by reading their $\mathcal{A}$-labels from left to right. Lemma~\ref{lem:Translate} will be often used in the next section in order to simplify the proofs. It justifies why we defined vertices of $M(\mathcal{P},w)$ as classes of clopen diagrams instead of more general open diagrams.

\begin{proof}[Proof of Lemma~\ref{lem:Translate}.]
Let $x \in M(\mathcal{P},w)$ be a vertex. Write $x$ as $[A,\beta,\sigma,B]$ for some clopen $(\mathcal{P},w)$-diagram $(A,\beta,\sigma,B)$. By definition of being closable, there must exist a (closed) $(\mathcal{P},w)$-diagram $(A,\beta,B^+)=(A,\beta, \sigma^+,B^+)$ such that $B$ can be obtained from $B^+$ by removing $\mathcal{P}$-interior vertices and such that $\sigma,\sigma^+$ agree on the $\mathcal{P}$-interior vertices of $B$. Then the reduction $g$ of $(B^+,\beta^{-1},A)$ is an element of $C(\mathcal{P},w)$ and we have
$$g \cdot x = (B^+, \beta^{-1}, A) \cdot [A,\beta,\sigma, B] = [B^+, \mathrm{id}, \sigma, B].$$
Observe that $(B^+,\mathrm{id},\sigma,B)$ can be obtained from $(B^+,\mathrm{id},B^+)$ by removing $\mathcal{P}$-interior vertices in the bottom forest. Because $(B^+,\mathrm{id},B^+)$ has reduction $\epsilon(w)$, this shows that the $\mathcal{P}$-interior vertices of $B$ in $(B^+,\mathrm{id},\sigma,B)$ disappear in the reduction, i.e. there exists a $(\mathcal{P},w)$-forest $T$ such that $(T, \mathrm{id},\emptyset,{\eta(m)})$ is the reduction of $(B^+,\mathrm{id},\sigma,B)$, where $m$ denotes the word obtained by reading from left to right the $\mathcal{A}$-labels of the roots of $B$. We conclude that $g \cdot x = [T,\mathrm{id},\emptyset,\eta(m)]$, as desired.
\end{proof}

\noindent
As it will be observed later, it turns out that the stabiliser in $C(\mathcal{P},w)$ of a cube in $M(\mathcal{P},w)$ must fix a vertex. Consequently, in order to understand cube-complexes it is sufficient to understand vertex-stabilisers. This is done by our next lemma (combined with Lemma~\ref{lem:Translate}).

\begin{lemma}\label{lem:Stab}
Let $x \in M(\mathcal{P},w)$ be a vertex of the form $[A,\mathrm{id},\emptyset,\eta(m)]$. Then, with respect to the action $C(\mathcal{P},w) \curvearrowright M(\mathcal{P},w)$,
$$\mathrm{stab}(x) = \{ g \in C(\mathcal{P},w) \mid g \equiv (A,\beta,A), \text{ $\beta$ arbitrary} \}.$$
As a consequence, $\mathrm{stab}(x)$ is isomorphic to the braid group $\mathcal{B}_{i(A),\ell(A)}$.
\end{lemma}

\begin{proof}
Let $g \in C(\mathcal{P},w)$ be an element stabilising $x$. By adding dipoles to $g$, we know that $g$ is equipotent to a diagram whose bottom forest contains $A$ as a prefix, i.e. we can write $g \equiv (R, \alpha , A \otimes P)$. Here, the notation $A \otimes P$ refers to a $\mathcal{P}$-forest obtained from $A$ and another $\mathcal{P}$-forest $P$, such that the word obtained by reading from left to right the $\mathcal{A}$-labels of the roots of $P$ coincides with the word obtained by reading from left to right the $\mathcal{A}$-labels of the $\mathcal{P}$-leaves of $A$, by identifying the roots of $P$ with the $\mathcal{P}$-leaves of $A$ from left to right (so by gluing $P$ below $A$). Observe that, by adding dipoles to $(A, \mathrm{id},\emptyset, \eta(m))$, we can write $x$ as $[A \otimes P, \mathrm{id},\sigma,P]$. In $(A \otimes P, \mathrm{id}, \sigma,P)$, the braid connected to a $\mathcal{P}$-interior vertex $p \in P$ in the bottom forest must be connected to the same $p$ in the top forest $A \otimes P$ (identifying $P$ with its image in $A \otimes P$). Consequently, 
\begin{description}
	\item[$(\ast)$] $\sigma : \iota(P) \hookrightarrow \iota(A \otimes P)$ is induced by the obvious embedding $P \hookrightarrow A \otimes P$. 
\end{description}
Because $g$ stabilises $x$, we have
$$[A, \mathrm{id},\emptyset,\eta(m)] = x = g \cdot x = [R, \alpha, \sigma,P],$$
which implies, according to Corollary~\ref{cor:Classes}, that the reduction of $(R,\alpha,\sigma,P)$ is isocephalese to $(A, \mathrm{id},\emptyset,\eta(m))$, or, equivalently, is of the form $(A,\beta,\emptyset,\eta(m))$. Thus, there exists a sequence of dipole reductions transforming $(R,\alpha,\sigma,P)$ into $(A,\beta, \emptyset,\eta(m))$. As a consequence of the observation $(\ast)$ above, it is possible to add vertices to the bottom forest of $(R,\alpha,\sigma,P)$ in order to get the closed diagram $(R,\alpha,A \otimes P)$. The previous sequence of dipole reductions applied to $(R,\alpha, A\otimes P)$ then leads to the diagram $(A,\beta, A)$. We conclude that $g \equiv (R,\alpha, A\otimes P)$ must be equipotent to $(A,\beta,A)$. Conversely, every element of $C(\mathcal{P},w)$ equipotent to a diagram of this form clearly stabilises $x$.
\end{proof}

\subsection{CAT(0)ness}

\noindent
From now on, we fix an arboreal semigroup presentation $\mathcal{P} = \langle \mathcal{A} \mid \mathcal{R} \rangle$ and a baseword $w \in \mathcal{A}^+$. This section is dedicated to the proof of the following statement:

\begin{thm}\label{thm:BigCC}
$M(\mathcal{P},w)$ is a CAT(0) cube complex. 
\end{thm}

\noindent
Recall that a cube complex is \emph{CAT(0)} if it is simply connected and if the links of its vertices are simplicial flag complexes. An alternative description is given by the proposition below. In its statement, given a graph $X$, we refer to the \emph{cube-completion} $X^\square$ as the cube complex obtained from $X$ by filling in every subgraph isomorphic to the one-skeleton of a cube with cube of the corresponding dimension.

\begin{prop}\label{prop:CriterionCC}
Let $X$ be a connected graph. Assume that
\begin{itemize}
	\item the cube-completion $X^\square$ is simply connected;
	\item $X^\square$ satisfies the \emph{$3$-cube condition} (i.e. three pairwise distinct squares that share a vertex and that pairwise share an edge must span a $3$-cube);
	\item $X$ does not contain copy of the complete bipartite graph $K_{3,2}$.
\end{itemize}
Then $X^\square$ is a CAT(0) cube complex.
\end{prop}

\noindent
The converse of the proposition also holds, but this observation will not be used in the sequel. Although not stated explicitly, Proposition~\ref{prop:CriterionCC} is contained in the proof of \cite[Theorem~6.1]{mediangraphs}. 
\medskip

\noindent
Our goal below is to verify that Proposition~\ref{prop:CriterionCC} applies to the cube complex $M(\mathcal{P},w)$. We begin by proving that our cube complex is connected.

\begin{prop}\label{prop:Mconnected}
Let $(A,\beta,\sigma,B)$ be a $(w,\ast)$-diagram. The distance between $[A,\beta,\sigma,B]$ and $[\epsilon(w)]$ in $M(\mathcal{P},w)$ is equal to $i(A) + i(B)$. As a consequence, $M(\mathcal{P},w)$ is connected.
\end{prop}

\begin{proof}
By removing all the $\mathcal{P}$-interior vertices in $B$, we construct a path in $M(\mathcal{P},w)$ of length $i(B)$ from $[A,\beta,\sigma, B]$ to $[A,\beta, {\emptyset},\eta(m)]$ for some word $m$. (Recall that $\eta(m)$ denotes the $(\mathcal{P},w)$-forest with no edges.) Observe that $[A,\beta,{\emptyset},\eta(m)]=[A,\mathrm{id}, {\emptyset},\eta(m)]$. Now, by adding $\mathcal{P}$-interior vertices to $\eta(m)$ we obtain in order to create dipoles with the $\mathcal{P}$-interior vertices in $A$, we construct a path in $M(\mathcal{P},w)$ of length $i(A)$ from $[A,\mathrm{id},\emptyset,\eta(m)]$ to  $[A,\mathrm{id}, \mathrm{id},A]=[\epsilon(w)]$. Our path from $[A,\beta,\sigma,B]$ to $[\epsilon(w)]$ has total length $i(A)+i(B)$ as desired. Conversely, it is clear from Lemma~\ref{lem:Adjacent} that a path from $[A,\beta, \sigma,B]$ to $[\epsilon(w)]$ must have length at least $i(A)+i(B)$. 
\end{proof}

\noindent
The second step towards the proof of Theorem~\ref{thm:BigCC} is to show that our cube complex is simply connected.

\begin{prop}\label{prop:Msimplyconnected}
$M(\mathcal{P},w)$ is simply connected.
\end{prop}

\noindent
In order to prove our proposition, we need to introduce a convenient \emph{height function}: given a $(w,\ast)$-diagram $\Delta=(A,\beta,\sigma,B)$, we set $h(\Delta):= i(A)-i(B)$. Observe that the height remains the same after reduction of a dipole and that it is constant on a class of isocephalese diagrams. As a consequence, it makes sense to refer to the height of a vertex of $M(\mathcal{P},w)$. Observe that the heights of any two adjacent vertices always differ by one. Also, notice that the height of a vertex is preserved by the action of $C(\mathcal{P},w)$. Indeed, given a vertex $x \in M(\mathcal{P},w)$ and an element $g \in C(\mathcal{P},w)$, fix two representatives $(A,\alpha, B)$ and $(B, \beta, \sigma, C)$ of $g$ and $x$ respectively. Then $(A,\beta \alpha,\sigma,C)$ represents $gx$. Hence
$$h(gx)= i(A)-i(C)=i(B)-i(C)= h(x),$$
proving the desired assertion. 

\begin{proof}[Proof of Proposition~\ref{prop:Msimplyconnected}.]
Given a loop $\gamma$ in the one-skeleton of $M(\mathcal{P},w)$, we define its \emph{support} in the following way. Let $[\Phi_0], \ldots, [\Phi_{n-1}]$ denote the vertices of $\gamma$, where each $\Phi_i$ is reduced. For every $0 \leq i \leq n-1$, write $\Phi_i$ as $(A_i,\beta_i,\sigma_i,B_i)$. Now, define $\mathrm{supp}(\gamma)$ as the forest that is the union of all the $A_i$. Of course, $\Phi_i$ is not the unique reduced diagram representing the vertex $[\Phi_i]$, but, according to Corollary~\ref{cor:Classes}, the $(\mathcal{P},w)$-forest $A_i$ does not depend on this choice.  Thus, the support of $\gamma$ is well-defined.

\medskip \noindent
Our goal is to show that $\gamma$ can be modified, up to homotopy, in order to get a single vertex.  

\medskip \noindent
If $\gamma$ backtracks along an edge, then the modification just consists in removing this backtrack. Now, assume that $\gamma$ does not backtrack along an edge and is not reduced to a single vertex. Then there exists some $i \in \mathbb{Z}_n$ such that $h(\Phi_i)<h(\Phi_{i-1}), h(\Phi_{i+1})$. As a consequence of Lemma~\ref{lem:Adjacent}, $\Phi_{i-1}$ and $\Phi_{i+1}$ are isocephalese (in fact, without loss of generality, equal) to either a top expansion or a bottom contraction of $\Phi_i$. So four cases may happen.
\begin{itemize}
	\item $\Phi_{i-1}$ and $\Phi_{i+1}$ are two top expansions of $\Phi_i$, i.e. there exist two distinct $\mathcal{P}$-leaves $a_{i-1},a_{i+1} \in A_i$ such that $\Phi_{i-1},\Phi_{i+1}$ are respectively the top expansions of $\Phi_i$ over $a_{i-1},a_{i+1}$. Let $\Phi_i'$ denote the diagram obtained by applying simultaneously the top expansions over $a_{i-1},a_{i+1}$. Observe that the top $(\mathcal{P},w)$-forest of $\Phi_i'$ coincides with $A_{i-1} \cup A_{i+1}$, that $h(\Phi_i')>h(\Phi_i)$, and that $[\Phi_{i-1}]$, $[\Phi_i]$, $[\Phi_{i+1}]$, and $[\Phi_i']$ are the vertices of a square in $M(\mathcal{P},w)$.
	\item $\Phi_{i-1}$ and $\Phi_{i+1}$ are two bottom contractions of $\Phi_i$, i.e. there exist two distinct roots $b_{i-1},b_{i+1} \in B_i$ such that $\Phi_{i-1},\Phi_{i+1}$ are respectively the bottom contractions of $\Phi_i$ over $b_{i-1},b_{i+1}$. Let $\Phi_i'$ denote the diagram obtained by applying simultaneously the bottom contractions over $a_{i-1},a_{i+1}$. Observe that the top $(\mathcal{P},w)$-forest of $\Phi_i'$ coincides with $A_i$, and that $[\Phi_{i-1}]$, $[\Phi_i]$, $[\Phi_{i+1}]$, that $h(\Phi_i')>h(\Phi_i)$, and $[\Phi_i']$ are the vertices of a square in $M(\mathcal{P},w)$.
	\item $\Phi_{i-1}$ is a top expansion of $\Phi_i$ and $\Phi_{i+1}$ is a bottom contraction of $\Phi_i$, i.e. there exist a $\mathcal{P}$-leaf $a \in A_i$ (resp. a root $b \in B_i$) such that $\Phi_{i-1}$ (resp. $\Phi_{i+1}$) is the top expansion of $\Phi_i$ over $a$ (resp. the bottom contraction of $\Phi_i$ over $b$). Let $\Phi_i'$ denote the diagram obtained by applying simultaneously our top expansion and bottom contraction. Observe that the top $(\mathcal{P},w)$-forest of $\Phi_i'$ coincides with $A_i$, that $h(\Phi_i')>h(\Phi_i)$, and that $[\Phi_{i-1}]$, $[\Phi_i]$, $[\Phi_{i+1}]$, and $[\Phi_i']$ are the vertices of a square in $M(\mathcal{P},w)$.
	\item The remaining case is symmetric to the previous one, switching the indices $i$ and $i+1$.
\end{itemize}
The key observation is that, by replacing $[\Phi_i]$ with $[\Phi_i']$ along $\gamma$, one obtains a new loop $\gamma'$ that is homotopy equivalent to $\gamma$, whose length is equal to that of $\gamma$, whose support lies in $\mathrm{supp}(\gamma)$, and such that the sum of the heights of its vertices is larger than the same sum for $\gamma$. 

\medskip \noindent
We iterate the process, creating a sequence of loops $\gamma,\gamma',\gamma'', \ldots$ that are all homotopy equivalent. Observe that this sequence must be finite because the total height of each loop (i.e. the sum of the heights of all the vertices) is bounded above by $\mathrm{length}(\gamma) \cdot |\mathrm{supp}(\gamma)|$. But the only possibility for our sequence to stop is that one of the loops is reduced to a single vertex.
\end{proof}

\noindent
The third step towards the proof of Theorem~\ref{thm:BigCC} is to verify that the one-skeleton of $M(\mathcal{P},w)$ is $K_{3,2}$-free. 

\begin{prop}\label{prop:Bipartite}
The one-skeleton of $M(\mathcal{P},w)$ does not contain a copy of the complete bipartite graph $K_{3,2}$. 
\end{prop}

\noindent
We begin by proving a preliminary lemma.
But before even stating it, let us observe that the vertices of $M(\mathcal{P},w)$ are naturally labelled by pairs of $\mathcal{P}$-forests. Indeed, if $x \in M(\mathcal{P},w)$ is a vertex, then there exists a reduced diagram $(A,\beta,\sigma,B)$ such that $x=[A,\beta,\sigma,B]$. As a consequence of Corollary~\ref{cor:Classes}, the pair $(A,B)$ does not depend on the reduced diagram we chose. Therefore, the pair $(A,B)$ naturally labels the vertex $x$. Of course, distinct vertices may be labelled by the same pair. It follows from Lemma~\ref{lem:Adjacent} that, when passing from a vertex to one of its neighbours, one adds/removes all the children of a $\mathcal{P}$-leaf to/from the top forest $A$ (for convenience, we denote by $c(a)$ the collection of all the children of such a $\mathcal{P}$-leaf $a \in A$) or one adds/removes a $\mathcal{P}$-interior vertex to/from the bottom forest $B$. Such a labelling of the vertices of $M(\mathcal{P},w)$ will be quite useful in the next proofs.

\begin{lemma}\label{lem:ParallelOrientation}
Let $(p,q,r,s)$ be a simple $4$-cycle in the one-skeleton of $M(\mathcal{P},w)$. If $h(p)<h(q)$ then $h(s)<h(r)$.
\end{lemma}

\begin{proof}
Assume for contradiction that $h(p)<h(q)$ and $h(s)>h(r)$. By orienting the edges in our cycle from low height to high height, four orientations are possible, as illustrated by Figure~\ref{SquaresOne}.  Three of them are clearly impossible:
\begin{itemize}
	\item if $h(p)>h(s)$ and $h(q)>h(r)$, then we must have $h(q)=h(r)+1$ but also $h(q)=h(p)+1=h(s)+2=h(r)+3$, a contradiction;
	\item if $h(s)>h(p)$ and $h(r)>h(q)$, the configuration is symmetric to the previous one;
	\item if $h(p)>h(s)$ and $h(r)>h(q)$, then $h(p)=h(s)+1=h(r)+2=h(q)+3=h(p)+4$, a contradiction.
\end{itemize}
From now on, we assume that $h(s)>h(p)$ and $h(q)>h(r)$. As a consequence of Lemma~\ref{lem:Adjacent}, a directed edge corresponds to applying a top expansion or a bottom contraction (in order to pass from the initial vertex to the terminal vertex of the edge); we refer to top and bottom edges accordingly. As a consequence of Lemma~\ref{lem:Translate}, we can assume without loss of generality that the edges containing $p$ are top edges. Indeed, Lemma~\ref{lem:Translate} shows that we can assume that $p$ is represented by a diagram whose bottom forest is trivial, and the only way to increase the height of such a vertex is to perform a top contraction. However, the two remaining edges, those not containing $p$, can be either top or bottom. As illustrated by Figure~\ref{SquaresOne}, there are four cases to consider. 
\begin{figure}
\begin{center}
\includegraphics[width=0.8\linewidth]{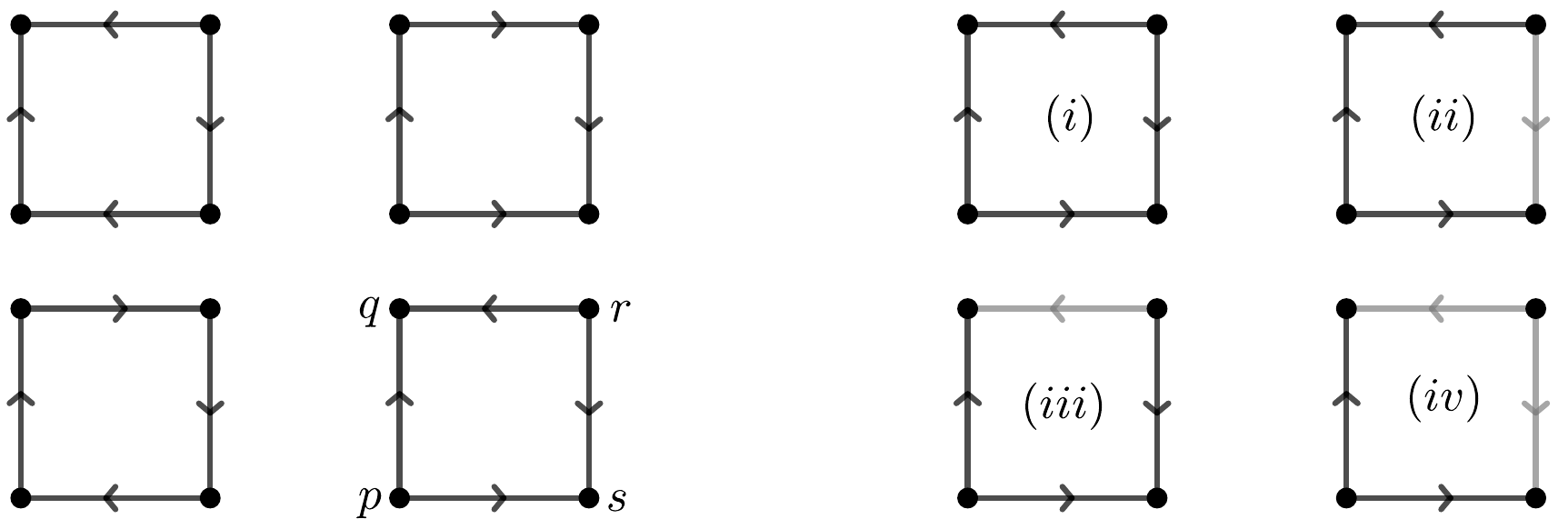}
\caption{On the left, the four possible orientations of our square; on the right, the four possible colourings of our square (the top edges are black, the bottom edges grey).}
\label{SquaresOne}
\end{center}
\end{figure}


\medskip \noindent
\emph{Case $(i)$.} In the first configuration, the vertices are labelled by pairs of $\mathcal{P}$-forests given by reduced representatives in such a way that the sequence of top forests is of the form (following the cycle clockwisely, starting from $p$)
$$A, \ A \cup c(a), \ (A\cup c(a)) \backslash c(a'), \ (A \cup c(a) \cup c(a'')) \backslash c(a'), \ (A \cup c(a) \cup c(a'')) \backslash (c(a''') \cup c(a''')).$$
Because we must have $A = (A \cup c(a) \cup c(a'')) \backslash (c(a') \cup c(a'''))$, necessarily $\{a,a''\}=\{a', a'''\}$. If $a=a'$ and $a''=a'''$, then we can write $q=[\Phi]=[\Psi]$, for some isocephalese reduced diagrams $\Phi,\Psi$, such that $p$ and $r$ are represented respectively by diagrams $\Phi'$ and $\Psi'$ obtained from $\Phi$ and $\Psi$ by applying a top contraction over $a=a'$. Because the strands in $\Phi$ and $\Psi$ starting from $a$ are free and parallel to the wires indexed by the children of $a$ in both $\Phi$ and $\Psi$, it follows that the braid applied in order to get $\Psi$ from $\Phi$ can be applied in order to get $\Psi'$ from $\Phi'$, proving that $\Psi$ and $\Phi$ are isocephalese. Thus, we have found the contradiction $p=[\Phi']=[\Psi']=r$. Otherwise, assume that $a=a'''$. Here, the situation is simpler because we know from Lemma~\ref{lem:Adjacent} that diagrams representing $q$ and $s$ can be obtained from a fixed representative of $p$ by applying top expansions over $a$ and $a'''$. (We do not need to consider diagrams isocephalese to our representative of $p$.) The equality $a=a'''$ immediately implies that $p=q$, again a contradiction.

\medskip \noindent
\emph{Cases $(ii)$ and $(iii)$.} Following the cycle clockwisely starting from $s$ in case $(ii)$, and anticlockwisely starting from $q$ in case $(iii)$, the vertices are labelled by pairs of $\mathcal{P}$-forests in such a way that the sequence of bottom forests is of the form $B,B,B,B,B\backslash \{b\}$. Because we cannot have $B=B \backslash \{b\}$, we get a contradiction.

\medskip \noindent
\emph{Case $(iv)$.} Following the cycle clockwisely starting from $q$, the vertices are labelled by pairs of $\mathcal{P}$-forests in such a way that the sequence of top forests is of the form $A,A,A, A \backslash c(a), (A\backslash c(a) ) \cup c(a')$. Because we must have $A=(A\backslash c(a) ) \cup c(a')$, necessarily $a=a'$. But this implies that $q=s$, a contradiction.
\end{proof}

\begin{proof}[Proof of Proposition~\ref{prop:Bipartite}.]
As a consequence of Lemma~\ref{lem:Adjacent}, the edges of $M(\mathcal{P},w)$ can be naturally labelled as \emph{top} and \emph{bottom}, depending on whether they correspond to a top or bottom expansion/contraction. We begin by proving the following general observation:

\begin{claim}\label{claim:ParallelTopBottom}
Two opposite edges in a square are either both top or both bottom.
\end{claim}
\begin{figure}
\begin{center}
\includegraphics[width=0.7\linewidth]{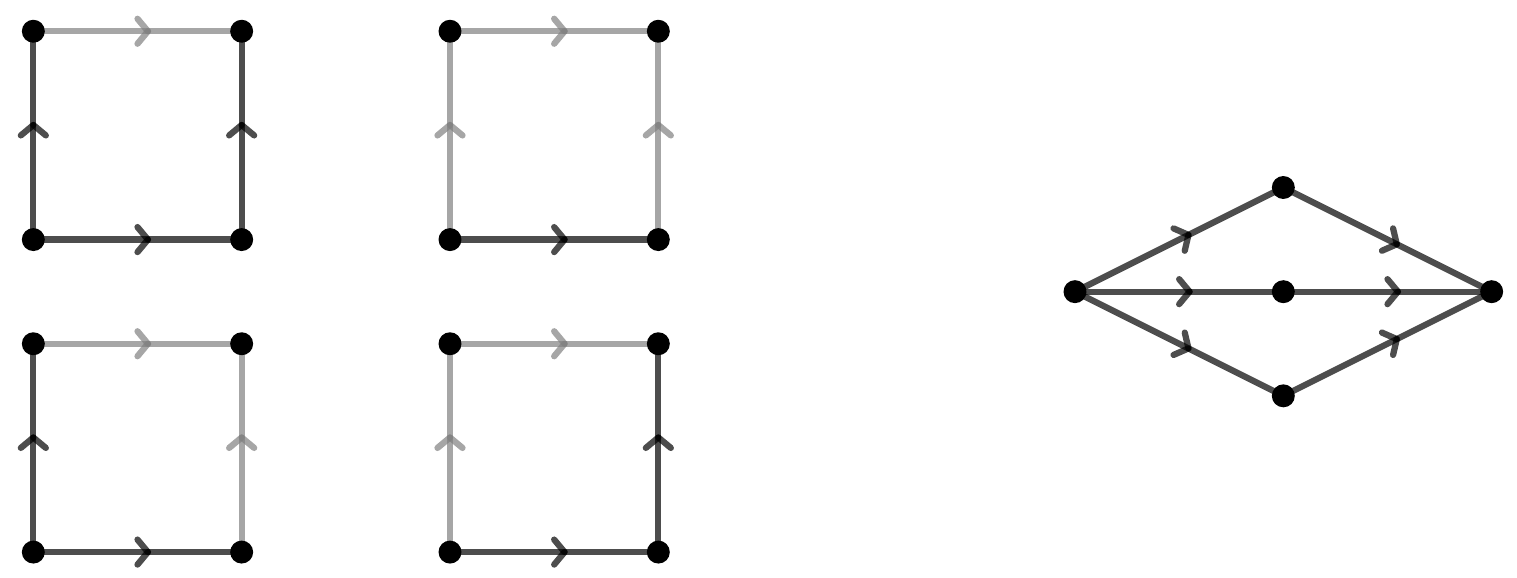}
\caption{On the left, the four possible colourings of our square; on the right, our colouring of $K_{3,2}$. (The top edges are black, the bottom edges grey.)}
\label{SquaresTwo}
\end{center}
\end{figure}

\noindent
Assume for contradiction that there exists a square in $M(\mathcal{P},w)$ with two opposite edges one of them being top and the other bottom. Then one of the configurations illustrated by Figure~\ref{SquaresTwo} occurs.
\begin{itemize}
	\item In the first configuration, the sequence of bottom forests indexing the vertices of our cycle is of the form $B,B,B,B,B \backslash \{b\}$. Since $B$ cannot coincide with $B \backslash \{b\}$, we get a contradiction.
	\item In the second configuration, the sequence of top forests indexing the vertices of our cycle is of the form $A, A,A,A, A \backslash c(a)$. Since $A$ cannot coincide with $A{\setminus} c(a)$, we get a contradiction.
	\item In the third configuration, the sequence of top forests indexing the vertices of our cycle, starting from bottom left vertex $v$, is of the form $A,A \cup c(a), A \cup c(a), A\cup c(a), (A \cup c(a)) \backslash c(a')$. Because $A$ must agree with $(A \cup c(a) ) \backslash c(a')$, necessarily $a=a'$. It follows that the two edges starting from $v$ coincide, a contradiction.
	\item In the fourth configuration, the sequence of top forests indexing the vertices of our cycle is of the form $A,A,A,A \cup c(a),A \cup c(a) \cup c(a')$. Since $A$ cannot coincide with $A \cup c(a) \cup c(a')$, we get a contradiction.
\end{itemize}
This concludes the proof of our claim.

\medskip \noindent
Now assume that $M(\mathcal{P},w)$ contains a copy of $K_{3,2}$. Let $a$ and $b$ denote the two vertices of degree $3$ in $K_{3,2}$. As a consequence of Lemma~\ref{lem:Translate}, we can assume without loss of generality that the three oriented edges containing $a$ are top and go away from $a$. Indeed, Lemma~\ref{lem:Translate} shows that we can assume that $a$ is represented by a diagram whose bottom forest is trivial, and the only way to increase the height of such a vertex is to perform a top contraction. It follows from  Lemma~\ref{lem:ParallelOrientation} and Claim~\ref{claim:ParallelTopBottom} that our configuration is all illustrated by Figure~\ref{SquaresTwo}.
Then we can find a $\mathcal{P}$-forest $A$ and vertices $u_1,u_2,v_1,v_2,w_1,w_2$ such that $a$ is indexed by the top forest $A$, $b$ by the top forest
$$A \cup c(u_1) \cup c(u_2)= A \cup c(v_1) \cup c(v_2)=A \cup c(w_1) \cup c(w_2),$$
while the three other vertices are indexed by the top forests $A \cup c(u_1)$, $A \cup c(v_1)$, and $A \cup c(w_1)$. Necessarily, $\{u_1,u_2\}=\{v_1,v_2\}=\{w_1,w_2\}$. So $u_1=v_1$ or $v_1=w_1$ or $u_1=w_1$. In each case, we deduce that, in our copy of $K_{3,2}$, two vertices of degree $2$ coincide, a contradiction. This completes the proof of our lemma.
\end{proof}

\noindent
Finally, we prove that our cube complex $M(\mathcal{P},w)$ satisfies the $3$-cube condition.

\begin{prop}\label{prop:MThreeCube}
$M(\mathcal{P},w)$ satisfies the $3$-cube condition.
\end{prop}

\begin{proof}
Let us consider three squares pairwise intersecting along an edge with a vertex common to all of them. As a consequence of Lemma~\ref{lem:ParallelOrientation}, the edges of these squares can be oriented by the height function (from low to high height) in four ways, as illustrated by Figure~\ref{3Cube}. In the three first configurations, there is a unique vertex $x$ that is of minimal height. As a consequence of Lemma~\ref{lem:Translate}, we can assume without loss of generality that $x$ is of the form $[A,\mathrm{id},\emptyset, \eta(m)]$. Then, it follows from Lemma~\ref{lem:Adjacent} that our configurations are as described by Figure~\ref{Conf}. In the three cases, the missing vertex of our $3$-cube is respectively, $[A \cup c(a_1) \cup c(a_2) \cup c(a_3),\mathrm{id},\emptyset, \eta(m)]$, $[A \cup c(a_1) \cup c(a_3),\mathrm{id},\emptyset, \eta(m)]$, and $[A \cup c(a_3),\mathrm{id},\emptyset, \eta(m)]$. 
\begin{figure}
\begin{center}
\includegraphics[width=0.8\linewidth]{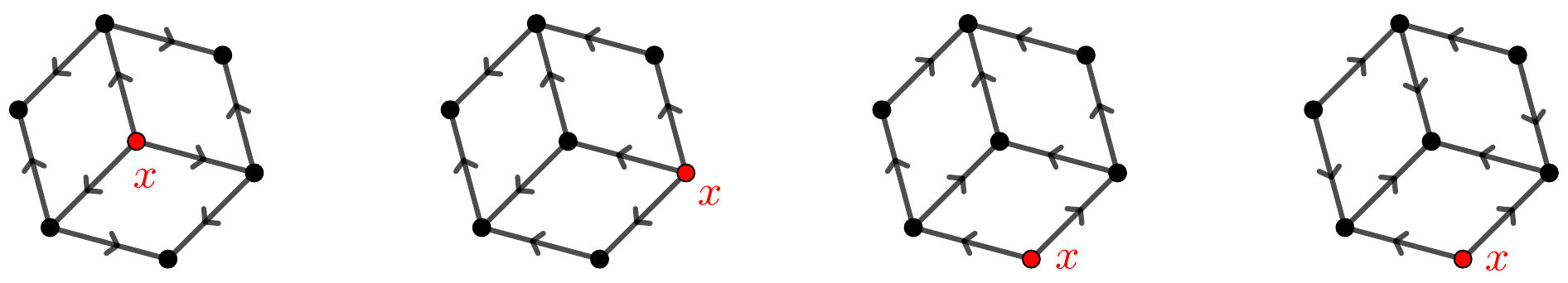}
\caption{Possible orientations of a cycle of three squares.}
\label{3Cube}
\end{center}
\end{figure}
\begin{figure}
\begin{center}
\includegraphics[width=\linewidth]{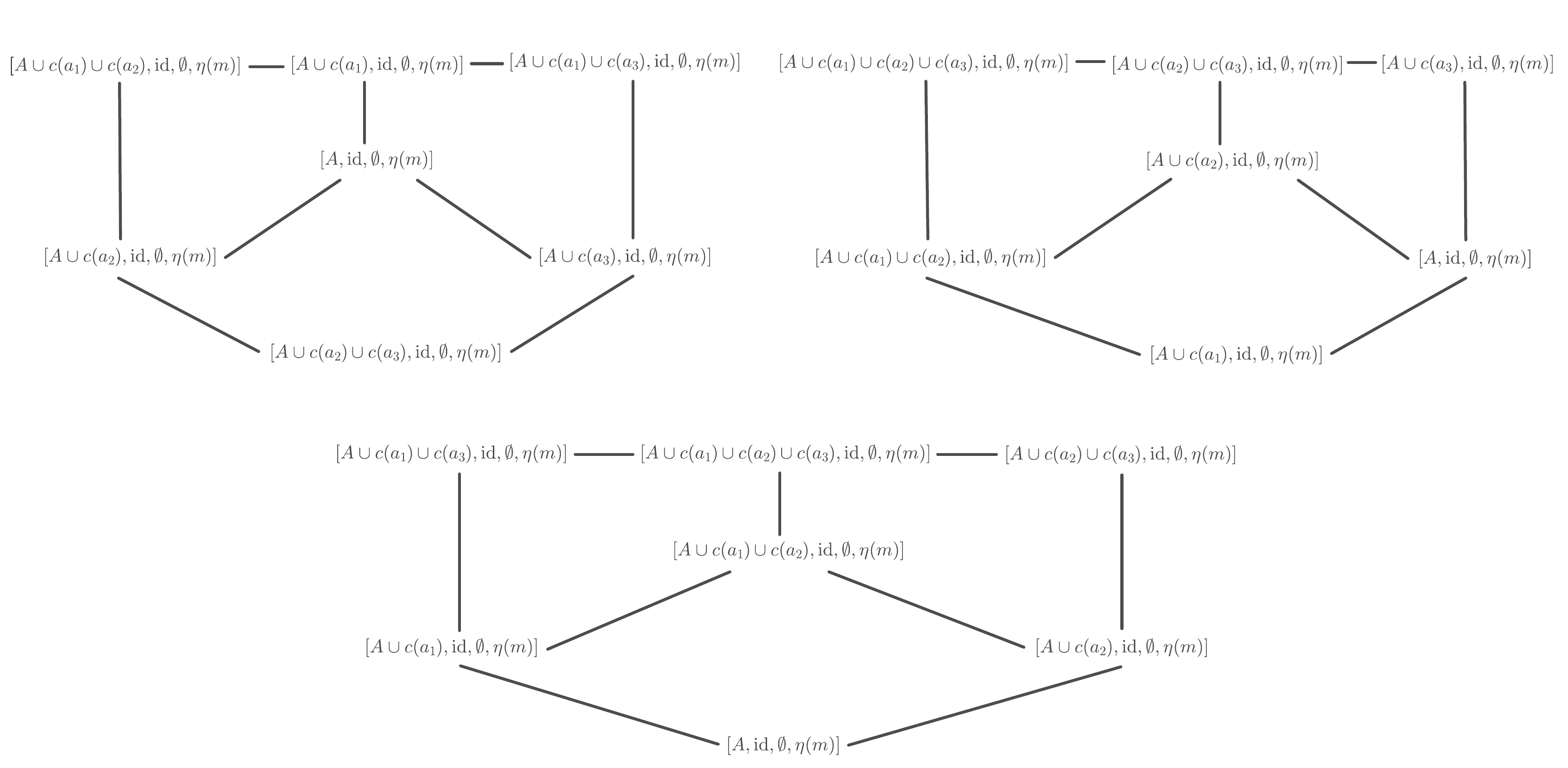}
\caption{The three first possible configurations.}
\label{Conf}
\end{center}
\end{figure}

\medskip \noindent
From now on, we focus on the fourth configuration of Figure~\ref{3Cube}. Let $x$ denote the bottom vertex. As a consequence of Lemma~\ref{lem:Translate}, we can assume without loss of generality that $x$ is of the form $[A, \mathrm{id},\emptyset,\eta(m)]$. We deduce from Lemma~\ref{lem:Adjacent} that the vertices of the lower square are as given by Figure~\ref{Configuration}. According to Claim~\ref{claim:ParallelTopBottom}, two cases may happen, depending on whether the three vertical edges are all top or all bottom. In the former case, the remaining vertices of our three squares are given by the left configuration of Figure~\ref{Configuration}, and $[A\backslash c(a_3),\mathrm{id},\emptyset,\eta(m)]$ is the missing vertex of our $3$-cube. In the latter case, the remaining vertices of our three squares are given by the right configuration of Figure~\ref{Configuration}, where $B$ is a forest with a single $\mathcal{P}$-interior vertex and where $(A \cup c(a_1), \alpha,\sigma, B)$ (resp. $(A \cup c(a_2), \alpha_2, \sigma_2,B)$, $(A \cup c(a_1) \cup c(a_2), \alpha_{12},\sigma_{12},B)$) is a diagram obtained from $(A \cup c(a_1), \mathrm{id},\emptyset, \eta(m))$ (resp. $(A \cup c(a_2), \mathrm{id}, \emptyset, \eta(m))$, $(A \cup c(a_1) \cup c(a_2), \mathrm{id}, \emptyset, \eta(m))$) by a bottom expansion. If $(A,\alpha_1,\sigma_1,B)$ denote the diagram obtained from $(A \cup c(a_1),\alpha,\sigma,B)$ by a top contraction over $a_1$, we claim that $[A,\alpha_1,\sigma_1,B]$ is the missing vertex of our $3$-cube. It is clearly adjacent to $[A,\mathrm{id},\emptyset, \eta(m)]$ and $[A \cup c(a_1),\alpha, \sigma,B]$, but it remains to show that it is also adjacent to $[A \cup c(a_2),\alpha_2, \sigma_2,B]$. We know that
\begin{itemize}
	\item from $(A,\alpha_1,\sigma_1,B)$, applying a top expansion over $a_1$ yields $(A \cup c(a_1), \alpha, \sigma,B)$;
	\item from $(A \cup c(a_1),\alpha,\sigma,B)$, applying a bottom contraction over the $\mathcal{P}$-interior vertex of $B$ yields $(A \cup c(a_1),\mathrm{id},\emptyset, \eta(m))$;
	\item from $(A \cup c(a_1),\mathrm{id},\emptyset, \eta(m))$, applying a top contraction over $a_1$ yields $(A,\mathrm{id},\emptyset, \eta(m))$;
	\item from $(A,\mathrm{id},\emptyset, \eta(m))$, applying a top expansion over $a_2$ yields $(A \cup c(a_2),\mathrm{id},\emptyset, \eta(m))$;
	\item from $(A \cup c(a_2),\mathrm{id},\emptyset, \eta(m))$, applying a bottom expansion yields $(A \cup c(a_2),\alpha_2,\sigma_2,B)$.
\end{itemize}
Along this sequence of transformations of diagrams, the two top expansion and contraction over $a_1$ cancel out, as well as the two bottom expansion and contraction, showing that $(A \cup c(a_2), \alpha_2, \sigma_2,B)$ can be obtained from $(A, \alpha_1,\sigma_2,B)$ by a top expansion over $a_2$. Thus, the vertices $[A,\alpha_1,\sigma_1,B]$ and $[A \cup c(a_2),\alpha_2,\sigma_2,B]$ are adjacent, as desired.
\begin{figure}
\begin{center}
\includegraphics[width=\linewidth]{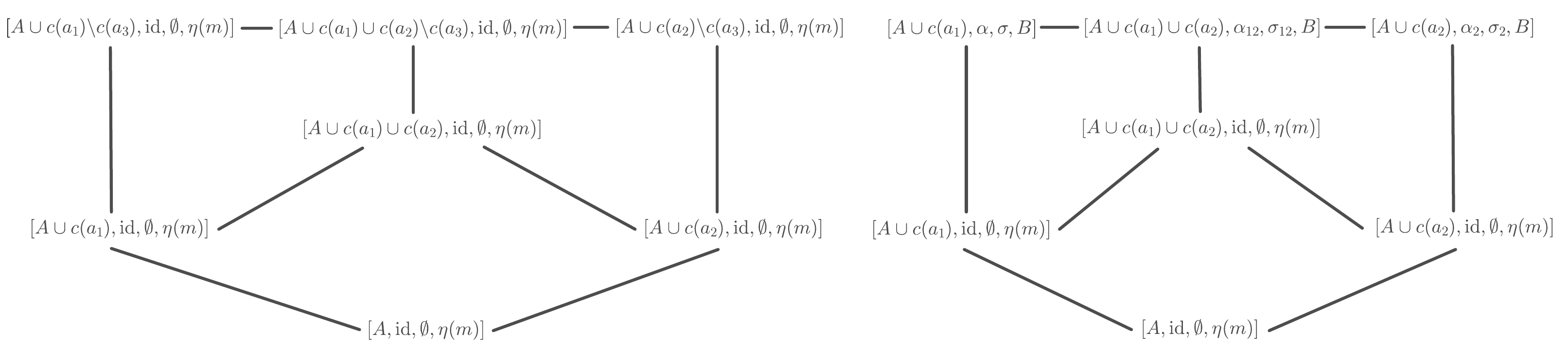}
\caption{The fourth possible configuration.}
\label{Configuration}
\end{center}
\end{figure}
\end{proof}

\begin{proof}[Proof of Theorem~\ref{thm:BigCC}.]
As a consequence of Propositions~\ref{prop:Mconnected},~\ref{prop:Msimplyconnected},~\ref{prop:Bipartite}, and~\ref{prop:MThreeCube}, we know that Proposition~\ref{prop:CriterionCC} applies, proving that $M(\mathcal{P},w)$ is a CAT(0) cube complex, as desired.
\end{proof}

\begin{remark}
The construction of our CAT(0) cube complex can be adapted in order to construct a CAT(0) cube complex $Q_a(\mathcal{P},w)$ (resp. $Q_p(\mathcal{P},w)$) on which the annular diagram group $D_a(\mathcal{P},w)$ (resp. the planar diagram group $D_p(\mathcal{P},w)$ described in Section~\ref{section:DiagramGroups} acts. Such a cube complex coincides with the complex constructed in \cite{MR1978047, MR2136028}, illustrated in Section~\ref{section:CCV} for Thompson's group $V$. Then, the labelling of the vertices of $M(\mathcal{P},w)$ by pairs of forests actually comes from a natural projection $M(\mathcal{P},w) \to Q_p(\mathcal{P},w)$, which turns out to be equivariant with respect to the canonical epimorphism $C(\mathcal{P},w) \to D_p(\mathcal{P},w)$. It can be shown that the quotients $M(\mathcal{P},w)/ C(\mathcal{P},w)$ and $Q_p(\mathcal{P},w)/ D_p(\mathcal{P},w)$ are isomorphic, providing a complex of braid groups whose fundamental group is $C(\mathcal{P},w)$ and whose underlying complex has fundamental group $D_p(\mathcal{P},w)$. Then, the previous projection $M(\mathcal{P},w) \to Q_p(\mathcal{P},w)$ turns out to be a development (see \cite[Chapter II.12]{MR1744486}). 
\end{remark}

\section{A few applications}\label{section:Afew}

\noindent
In this section, we record a few direct applications of the actions of Chambord groups on the CAT(0) cube complexes constructed in Section~\ref{section:CCC}.

\paragraph{Fixed-point properties.} First, we turn to the proof of Theorem~\ref{thm:IntroFW}, which states that the only groups satisfying the property $(FW_\omega)$ that can be embedded into Chambord groups must be finite. Recall that a group satisfies the property $(FW_\omega)$ if all its actions on CAT(0) cube complexes without infinite cubes are elliptic. {Let us observe that a group with property $(FW_\omega)$ in particular does not surject onto $\mathbb{Z}$.} We begin by proving the theorem for braid groups.

\begin{lemma}\label{lem:FWomega}
For all $p,q \geq 0$, every infinite subgroup in $\mathcal{B}_{p,q}$ virtually surjects onto $\mathbb{Z}$.
\end{lemma}

\begin{proof}
Because the pure subgroup $P\mathcal{B}_p \leq \mathcal{B}_{p,q}$ has finite index $q\cdot p!$, it suffices to show that every non-trivial subgroup $H \leq P\mathcal{B}_p$ surjects onto $\mathbb{Z}$. We argue by induction over $p$. If $p \leq 1$, then $P\mathcal{B}_p$ is trivial and there is nothing to prove. If $p \geq 2$, then we know thanks to a famous theorem due to Artin that the morphism $P \mathcal{B}_p \twoheadrightarrow P \mathcal{B}_{p-1}$ that forgets the last strand has a free kernel. Either $H$ has a non-trivial image under this morphism, and our induction hypothesis implies that $H$ surjects onto $\mathbb{Z}$, or $H$ lies in the free kernel, and the conclusion is clear. 
\end{proof}

\begin{cor}\label{cor:FWomega}
For all $p,q \geq 0$, every subgroup in $\mathcal{B}_{p,q}$ satisfying the property $(FW_\omega)$ must be finite and hence in particular cyclic.
\end{cor}

\begin{proof}
It is well-known that, if a group $G$ contains a finite-index subgroup $K$ acting on a cell complex $X$, then $G$ itself naturally acts on the Cartesian product $X^{[G:K]}$. This implies that the property $(FW_\omega)$ passes from a group to its finite-index subgroups. Our corollary follows from this observation and from Lemma~\ref{lem:FWomega}. Using the exact sequence from Section~\ref{sec:braidedstrand} we see that finite subgroups of $\mathcal{B}_{p,q}$ have to be cyclic. 
\end{proof}

\begin{proof}[Proof of Theorem~\ref{thm:IntroFW}.]
Let $\mathcal{P}=\langle \mathcal{A} \mid \mathcal{R} \rangle$ be an arboreal semigroup presentation, $w \in \mathcal{A}^+$ a baseword, and $H \leq C(\mathcal{P},w)$ a subgroup satisfying the property $(FW_\omega)$. Because a group acting on a median graph with bounded orbits has to stabilise a cube, necessarily $H$ stabilises a cube with respect to the action of $C(\mathcal{P},w)$ on $M(\mathcal{P},w)$. In fact, it has to fix a vertex:

\begin{claim}\label{claim:VertexFixed}
The stabiliser in $C(\mathcal{P},w)$ of a cube of $M(\mathcal{P},w)$ must fix a vertex.
\end{claim}

\noindent
Indeed, it follows from Lemma~\ref{lem:ParallelOrientation} that a cube contains a unique vertex of minimal height. Since the height function is preserved by the action of $C(\mathcal{P},w)$, our claim follows.

\medskip \noindent
Therefore, we deduce from Lemmas~\ref{lem:Stab} and~\ref{lem:Translate} that $H$ embeds into $\mathcal{B}_{p,q}$ for some $p,q \geq 0$. We conclude from Lemma~\ref{cor:FWomega} that $H$ must be finite. 
\end{proof}

\paragraph{Torsion elements.} Following the lines of \cite[Theorem 4.3]{GLU}, one can exploit the CAT(0) cube complexes we constructed in order to show that finite subgroups in Chambord groups are necessarily cyclic and to classify the possible orders of finite-order elements (depending on the arboreal semigroup presentation and the baseword). Because this is already done for asymptotically rigid mapping class groups in~\cite{GLU}, we do not pursue this direction further and leave the details to the reader. In this article, we restrict ourselves to the following observation about pure subgroups, as defined at the end of Section~\ref{section:GroupLaw}:

\begin{prop}
Pure subgroups in Chambord groups are torsion-free.
\end{prop}

\begin{proof}
Let $\mathcal{P}=\langle \mathcal{A} \mid \mathcal{R} \rangle$ be an arboreal semigroup presentation and $w \in \mathcal{A}^+$ a baseword. As a consequence of Claim~\ref{claim:VertexFixed}, a finite-order element in $C(\mathcal{P},w)$ must fix a vertex of the cube complex $M(\mathcal{P},w)$. But it follows from Lemmas~\ref{lem:Stab} and~\ref{lem:Translate} that $PC(\mathcal{P},w)$ acts on $M(\mathcal{P},w)$ with torsion free vertex-stabilisers. The desired conclusion follows.
\end{proof}

\paragraph{Finite presentability.} Following \cite{MR4033504}, an action on a CAT(0) cube complex can be used in order to prove that some group is not finitely presented. For instance, such a strategy applies to the braided lamplighter group $\mathrm{br} \mathcal{L}$ briefly discussed in Section~\ref{section:Examples}. In fact, thanks to \cite{GT}, it can be proved that the lamplighter group $\mathbb{Z}\wr \mathbb{Z}$ (as well as many other wreath products) does not admit a finitely presented braided version. More precisely:

\begin{prop}\label{prop:notFP}
Let $A$ be a non-trivial group and $H$ an infinite group. Assume that they do not contain non-abelian free subgroups. If a group $G$ satisfies a short exact sequence
$$1 \longrightarrow B_\infty \longrightarrow G \longrightarrow A \wr H \longrightarrow 1,$$
where $B_\infty$ denotes the group of finitely supported braids on infinitely many strands, then $G$ is not finitely presented.
\end{prop}

\begin{proof}
Assume for contradiction that $G$ is finitely presented. As a consequence of \cite{GT}, the quotient map $G \twoheadrightarrow A \wr H$ factors through a quotient $\bar{G}$ that is acylindrically hyperbolic. Let $B$ denote the image of $B_\infty$ in $\bar{G}$. Observe that $B$ is a normal subgroup in $\bar{G}$. We distinguish two cases.

\medskip \noindent
Assume first that $B$ is finite. Then $G/ B_\infty$, which is isomorphic to $A \wr H$, surjects onto $\bar{G}/B$. Because $\bar{G}$ contains a non-abelian free subgroup, as any acylindrically hyperbolic group, so does $\bar{G}/B$ since $B$ is finite. It follows that $A \wr H$, and a fortiori $A$ or $H$, contains a non-abelian free subgroup, contradicting our assumptions.

\medskip \noindent
Next, assume that $B$ is infinite. Let $X$ be a hyperbolic space on which $\bar{G}$ acts non-elementarily and acylindrically. As an infinite normal subgroup, the induced action of $B$ on $X$ must be non-elementary \cite[Lemma~7.2]{OsinAcyl}, which implies that there exist two loxodromic elements $a,b \in B$ generating a free group. Let $\alpha,\beta \in B_\infty$ be two pre-images respectively of $a,b$. If $\gamma \in B_\infty$ is a braid that sends the supports of $\alpha,\beta$ sufficiently far away, then $\gamma \alpha \gamma^{-1}$ commutes with $\alpha$ and $\beta$. A fortiori, if $g$ denotes the image of $\gamma$ in $B$, then so does $gag^{-1}$ with $a$ and $b$. Since $gag^{-1}$ must stabilise quasi-axes of $a$ and $b$, it follows that $gag^{-1}$ is elliptic. We conclude that $a$ is elliptic, a contradiction.
\end{proof}

\section{Polycyclic subgroups}\label{section:Polycyclic}

\noindent
This last section is dedicated to the main application of the CAT(0) cube complexes constructed in Section~\ref{section:CCC}. Namely, we want to prove that, in Chambord groups, polycyclic groups are virtually abelian and undistorted. For this purpose, the following statement, which shows that every action of a polycyclic group on a CAT(0) cube complex essentially factorises through an abelian quotient, will be central.

\begin{thm}[\cite{CubicalFlat}]\label{thm:PolycyclicCC}
Let $G$ be a polycyclic group acting on a CAT(0) cube complex $X$. Then $G$ contains a finite-index subgroup $H$ such that $$\mathcal{E}:=\{g \in H \mid \text{$g$ is elliptic in $X$} \}$$ defines a normal subgroup of $H$ with $H/\mathcal{E}$ free abelian. 
\end{thm}

\noindent
Observe that replacing $G$ with the finite-index subgroup $H$ is necessary because the product of two elliptic isometries may not be elliptic. For instance, consider the usual action of the infinite dihedral group $D_\infty$ on a bi-infinite line. However, the product of two commuting elliptic isometries is again elliptic, making the conclusion of Theorem~\ref{thm:PolycyclicCC} clear when $G$ is abelian.

\subsection{Abelian subgroups are undistorted}

\noindent
In this section, we focus on the distortion of specific subgroups in Chambord groups. Recall that, given a group $G$ and a finitely generated subgroup $H \leq G$, $H$ is \emph{undistorted} if, for every finitely generated subgroup $K \leq G$ containing $H$, the inclusion $H \hookrightarrow K$ is a quasi-isometric embedding. The section is dedicated to the proof of the following statement:

\begin{thm}\label{thm:DistAbelian}
Let $\mathcal{P}= \langle \mathcal{A} \mid \mathcal{R} \rangle$ be an arboreal semigroup presentation and $w \in \mathcal{A}^+$ a baseword. Finitely generated abelian subgroups in $C(\mathcal{P},w)$ are undistorted.
\end{thm}

\noindent
The strategy is to apply the following general observation, which is of independent interest:

\begin{prop}\label{prop:CCundistorted}
Let $G$ be a group acting on a CAT(0) cube complex. Assume that vertex-stabilisers are undistorted in $G$ and have their finitely generated abelian subgroups undistorted. Then finitely generated abelian subgroups in $G$ are undistorted.
\end{prop}

\noindent
In order to prove Proposition~\ref{prop:CCundistorted} we need the next preliminary lemma. Given a CAT(0) cube complex $X$, we denote by $\|\cdot\|$ the corresponding \emph{translation length function}, defined by 
$$\|g\|:= \lim\limits_{n \to + \infty} \frac{1}{n} d(x_0,g^n \cdot x_0), \text{ $g \in \mathrm{Isom}(X)$, $x_0 \in X$ basepoint}.$$
It is well-known that the limit exists and does not depend on the basepoint $x_0$ we choose.

\begin{lemma}\label{lem:CCundistorted}
Let $A$ be a finitely generated abelian group acting properly on a CAT(0) cube complex. There exists some $C>0$ such that $\|a\| \geq C \cdot |a|_A$ for every $a \in A$, where $| \cdot |_A$ denotes the word length with respect to $A$.
\end{lemma}

\begin{proof}
As a consequence of \cite{MR3704240}, $A$ stabilises a subcomplex $Y$ that is CAT(0) and finite-dimensional, and it follows from the flat torus theorem \cite[Theorem II.7.1]{MR1744486} that $Y$ contains a CAT(0)-convex flat $F$ on which $A$ acts properly and cocompactly. Fix a point $x_0 \in F$. As a consequence of the Milnor-\v Svarc lemma, there exist constants $C_1,C_2>0$ such that $d_2(a \cdot x_0, b \cdot x_0) \geq C_1 \cdot |a^{-1}b|_A - C_2$ for all $a,b \in A$, where $d_2$ denotes the CAT(0) metric. Because $Y$ is finite-dimensional, we also know that the metrics $d$ and $d_2$ are biLipschitz equivalent on $Y$. In particular, there exists a constant $C_3>0$ such that $d(x,y) \geq C_3 d_2(x,y)$ for all vertices $x,y \in Y$. Then
$$\begin{array}{lcl} \|a\| & \geq & \displaystyle C_3 \lim\limits_{n \to + \infty} \frac{1}{n} d_2(x_0,a^n \cdot x_0) \geq C_1 C_3  \lim\limits_{n \to + \infty} \frac{1}{n} \left( |a^n|_A - C_2/C_3 \right) \\ \\ & \geq & \displaystyle C_1 C_3 \cdot |a|_A \end{array}$$
for all $a \in A$. This concludes the proof of our lemma.
\end{proof}

\begin{proof}[Proof of Proposition~\ref{prop:CCundistorted}.]
Assume for contradiction that $G$ contains a distorted finitely generated abelian subgroup $A$. Up to replacing $G$ with a finitely generated subgroup containing $A$, we assume that $G$ is finitely generated. The fact that $A$ is distorted in $G$ precisely means that there exist elements $g_n \in A$ such that $|g_n|_G / |g_n|_A \to 0$ as $n \to + \infty$. 

\medskip \noindent
As a consequence of Proposition~\ref{thm:PolycyclicCC}, up to replacing $A$ by a finite index subgroup, we can assume that $A$ is a direct product $A_1 \oplus A_2$ where $A_1$ is elliptic and $A_2$ purely loxodromic. (This can also be shown directly without using Proposition~\ref{thm:PolycyclicCC}). Again up to replacing $A$ with a finite-index subgroup, we assume without loss of generality that $A_1$ fixes a vertex. For every $n \geq 0$, fix $a_n \in A_1$ and $b_n \in A_2$ such that $g_n=a_nb_n$. We distinguish two cases.

\medskip \noindent
First, assume that there exists some $C>0$ such that $|b_n|_A \geq C \cdot |g_n|_A$ for every $n \geq 0$. Because $\|\cdot\|$ satisfies a triangle inequality, we have $\|g_n\| \leq C_1 \cdot |g_n|_G$ for every $n \geq 0$, where $C_1$ is the maximum of the $\|s\|$ where $s$ is one of our generators of $G$. Also, observe that, since $A_1$ is elliptic, the equality $\|g_n\|=\|b_n\|$ holds for every $n \geq 0$. By combining these two observations with Lemma~\ref{lem:CCundistorted}, we find some constant $C_2>0$ such that
$$C_2 \cdot |g_n|_A \leq C_2 \cdot |b_n|_A \leq \|b_n\| = \|g_n\| \leq C_1 \cdot |g_n|_G$$
for every $n \geq 0$. We obtain a contradiction by dividing by $|g_n|_A$ and by letting $n$ go to infinity.

\medskip \noindent
Next, assume that $|b_n|_A/|g_n|_A \to 0$ as $n \to + \infty$. Up to extracting a subsequence, this amounts to assuming the previous case does not happen. Because vertex-stabilisers are undistorted in $G$ and have their finitely generated abelian subgroups undistorted, there exists some $C>0$ such that $|a_n|_G \geq C \cdot |a_n|_A$ for every $n \geq 0$. We have
$$\begin{array}{lcl} |g_n|_G & \geq & |a_n|_G- |b_n|_G \geq C \cdot |a_n|_A - |b_n|_A \\ \\ & \geq & C \left( |a_n|_A + |b_n|_A \right) - (C+1) \cdot |b_n|_A = C \cdot |g_n|_A - (C+1) \cdot |b_n|_A \end{array}$$
for every $n \geq 0$. We obtain a contradiction by dividing by $|g_n|_A$ and by letting $n$ go to infinity.
\end{proof}

\noindent
Thus, in order to prove Theorem~\ref{thm:DistAbelian}, it remains to verify that the vertex-stabilisers characterised by Lemma~\ref{lem:Stab} are undistorted in the Chambord group under consideration. This is done by our next proposition.

\begin{prop}\label{prop:BraidUndistorted}
Let $\mathcal{P}= \langle \mathcal{A} \mid \mathcal{R} \rangle$ be an arboreal semigroup presentation, $w \in \mathcal{A}^+$ a baseword, and $A$ a $(\mathcal{P},w)$-forest. The braid subgroup $$\mathcal{B}_A:= \{g \in C(\mathcal{P},w) \mid g \equiv (A,\alpha,A) \text{ for some braid $\alpha\in\mathcal{B}_{i(A),\ell(A)}$} \}$$ is undistorted in $C(\mathcal{P},w)$.
\end{prop}

\begin{proof}
Recall that, for all $p,q \geq 0$, the subgroup $\mathcal{B}_p$ of $\mathcal{B}_{p,q}$ is given by the braids that fix each marked point on the boundary $\partial \mathscr{D}_{p,q}$ (or, in other words, whose wires are vertical straight lines). In the sequel, denote the ``true'' braid associated to an $\alpha \in \mathcal{B}_p$, namely the braid obtained from $\alpha$ by keeping its strands but removing its wires, as $\alpha^\circ$. 

\medskip \noindent
First, we define a projection $\pi_A : C(\mathcal{P},w) \to \mathcal{B}_A^-$, where $$\mathcal{B}_A^-:=\{g \in C(\mathcal{P},w) \mid g \equiv (A,\alpha,A) \text{ for some $\alpha \in \mathcal{B}_{i(A)} \leq \mathcal{B}_{i(A),\ell(A)}$} \},$$ as follows. Given a reduced diagram $g \in C(\mathcal{P},w)$, we add dipoles in order to write $g\equiv (R,\alpha,S)$ where $R$ contains $A$ as a prefix. Next, we set $\pi_A(g):= (A,\hat{\alpha},A)$ where $\hat{\alpha} \in \mathcal{B}_{i(A)}$ is the braid such that $\hat{\alpha}^\circ$ is obtained from $\alpha^\circ$ by removing the strands that are not connected to a $\mathcal{P}$-interior vertex of the prefix $A$ in the top forest $R$. It is worth noticing that $\hat{\alpha}$ does not depend on the representative $(R,\alpha,S)$ we chose, because strands connected to possible extra dipoles are removed anyway, so $\pi_A$ is well-defined. Also observe that $\pi_A$ restricts to the identity on $\mathcal{B}_A^-$.

\medskip \noindent
Next, given an element $g \in C(\mathcal{P},w)$ and an integer $k \geq 0$, we set
$$\mathscr{B}_k(g):= \{ \alpha \in \mathcal{B}_k \mid \text{$\alpha^\circ$ subbraid of $\beta^\circ$ of size $k$}, \ g \equiv (R,\beta,S) \},$$
where $\alpha^\circ$ being a \emph{subbraid of $\beta^\circ$ of size $k$} means that $\alpha^\circ$ is obtained from $\beta^\circ$ by removing all but $k$ strands. Observe that:

\begin{claim}\label{claim:ForRetraction}
For all $a,b \in C(\mathcal{P},w)$, there exists some $\beta \in \mathscr{B}_{i(A)}(b)$ such that the equality $\pi_A(ab)=\pi_A(a) \cdot (A,\beta,A)$ holds.
\end{claim}

\noindent
Write $a \equiv (R,\mu,S)$ and $b\equiv (S,\nu,T)$ where $R$ contains $A$ as a prefix. Let $\hat{\mu} \in \mathcal{B}_{i(A)}$ be such that $\hat{\mu}^\circ$ is the subbraid of $\mu^\circ$ containing the strands connected to the $\mathcal{P}$-interior vertices of $A$ in $R$, and let $\hat{\nu} \in \mathcal{B}_{i(A)}$ be such that $\hat{\nu}^\circ$ is the subbraid of $\nu^\circ$ containing the strands connected to the $\mathcal{P}$-interior vertices of $S$ that are the endpoints of the strands in $\hat{\mu}$. Then we have 
$$\pi_A(ab) = (A, \hat{\nu} \hat{\mu}, A) = (A,\hat{\mu},A) \cdot (A, \hat{\nu},A) = \pi_A(a) \cdot (A,\hat{\nu},A)$$
where $\hat{\nu} \in \mathscr{B}_{i(A)}(b)$. This concludes the proof of our claim.

\medskip \noindent
Our next observation is the key point of the argument:

\begin{claim}\label{claim:Bfinite}
For all $g \in C(\mathcal{P},w)$ and $k \geq 0$, the set $\mathscr{B}_k(g)$ is finite.
\end{claim}

\noindent
During the proof of the claim, the following terminology will be needed. For all $p,q \geq 0$, let us say that two strands $s_1,s_2$ in a braid $\alpha \in \mathcal{B}_{p,q}$ are \emph{parallel} if there exist two arcs $a_1 \subset \mathscr{D}_{p,q} \times \{0\}$, resp. $a_2 \subset \mathscr{D}_{p,q} \times \{1\}$, connecting the endpoints of $s_1, s_2$ in $\mathscr{D}_{p,q} \times \{0\}$, resp. in $\mathscr{D}_{p,q} \times \{1\}$, and lying in the lower part of the disc such that $a_1 \cup s_2 \cup a_2$ and $s_1$ are homotopic (with fixed endpoints) in the complement of the strands distinct from $s_1,s_2$ in $\mathscr{D}_{p,q} \times [0,1]$. (This amounts to saying that, up to isotopy, there exists a ribbon in $(\text{lower part of } \mathscr{D}_{p,q}) \times [0,1]$ disjoint from the strands different from $s_1,s_2$ such that two of its opposite sides are delimited by $s_1,s_2$ while its two remaining sides connect the endpoints of $s_1,s_2$ respectively in $\mathscr{D}_{p,q} \times \{0\}$, $\mathscr{D}_{p,q}\times\{1\}$.)

\medskip \noindent
Now, fix some $g \in C(\mathcal{P},w)$, some $k \geq 0$, and some finite subset $\mathscr{B} \subset \mathscr{B}_k(g)$. Write $g=(R,\alpha,S)$. By adding dipoles, we then write $g \equiv (U,\beta,V)$ such that all the braids in $\mathscr{B}$ are given by subbraids of $\beta^\circ$. Observe that $U$ has to contain $R$ as a prefix. As a consequence of the following three observations, we know that, if $r$ is a $\mathcal{P}$-leaf of $R$ and if $u \in U$ is one of its descendants, then the strands in $\beta$ starting from $r$ and $u$ are parallel.
\begin{itemize}
	\item In a diagram, if a strand is parallel to a wire and if we expand this wire to get a dipole, then, in the new diagram, the previous strand is parallel to the new wires.
	\item In a diagram, if a strand is parallel to a wire and if we expand a distinct wire to get a dipole, then, in the new diagram, the previous strand is still parallel to the wire.
	\item In a diagram, if two strands are parallel to the same wire, then they are parallel.
\end{itemize}
Indeed, it follows from the first two points that, in $(U,\beta,V)$, a strand that is connected to a vertex $v$ not $\mathcal{P}$-interior in $R \leq U$ must be parallel to the wires indexed by the $\mathcal{P}$-leaves of $U$ that are descendants of $v$. Next, we deduce from the third point that, for every vertex $v$ not $\mathcal{P}$-interior in $R \leq U$, the strand connected to $v$ must be parallel to the strand connected to the $\mathcal{P}$-leaf of $R$ that is a parent of $v$, as desired. Thus, up to parallelism, the number of strands in $\beta$ is at most the number of vertices $\# R$  of $R$. 

\medskip \noindent
Let $\mathscr{C}$ denote the parallel classes of strands in $\beta$, and define a map $\xi : \mathscr{B} \to \mathscr{C}^k$ that sends a subbraid of $\beta$ of size $k$ to the sequence of parallel classes of its $k$ strands (from left to right, seeing strands of an element in $\mathscr{B}$ as strands of $\beta$). Observe that, given an element of $\mathscr{B}$, replacing a strand with a parallel strand in $\beta$ may only modify the braid we get by a pre- and post-composition with an elementary twist that permutes two punctures along an arc in the lower halfdisc. So, for every $\mu \in \mathscr{B}$, there exist at most $(k-1)^2$ braids $\nu \in \mathscr{B}$ such that $\nu$ is obtained from $\mu$ by replacing a fixed strand with another parallel strand. It follows that the pre-image of a single point under $\xi$ has cardinality at most $(k-1)^{2k}$. We conclude that 
$$|\mathscr{B}| \leq k (k-1)^{2k} \cdot |\mathscr{C}| \leq k (k-1)^{2k} \cdot \# R.$$
Because our upper bound does not depend on $\mathscr{B}$, we deduce that $\mathscr{B}_k(g)$ must be finite. This completes the proof of Claim~\ref{claim:Bfinite}.

\medskip \noindent
We are now ready to conclude the proof of our proposition. So let $G \leq C(\mathcal{P},w)$ be a finitely generated subgroup containing $\mathcal{B}_A$. Because $\mathcal{B}_A^-$ has finite index in $\mathcal{B}_A$, it suffices to show that $\mathcal{B}_A^-$ is quasi-isometrically embedded in $G$. Fix a finite symmetric generating set $S \subset G$ and a finite generating set $R \subset \mathcal{B}_A^-$ containing $$\bigcup\limits_{s \in S} \{ g \in C(\mathcal{P},w) \mid g \equiv (A,\alpha,A) \text{ for some } \alpha \in \mathscr{B}_{i(A)}(s)\},$$
which is possible thanks to Claim~\ref{claim:Bfinite}. Let $\rho : G \to \mathcal{B}_A^-$ denote the restriction of $\pi_A$ to $G$. As a consequence of Claim~\ref{claim:ForRetraction}, for every $g \in G$ and every $s \in S$, there exists some $\varsigma \in \mathscr{B}_{i(A)}(s)$ such that $\rho(gs) = \rho(g) \cdot (A,\varsigma,A)$; observe that $(A,\varsigma,A)$ represents an element of $R$. Thus, $\rho$ defines a Lipschitz quasi-retraction of $G$ onto $\mathcal{B}_A^-$, which implies that $\mathcal{B}_A^-$ is quasi-isometrically embedded in $G$.

\end{proof}

\begin{proof}[Proof of Theorem~\ref{thm:DistAbelian}.]
As biautomatic groups, braid groups have their (finitely generated) abelian subgroups undistorted (see for instance \cite[Theorem~III.$\Gamma$.4.10]{MR1744486} and \cite[Theorem~9.3.4]{MR1161694}). As a consequence, the combination of Lemmas~\ref{lem:Translate} and~\ref{lem:Stab} with Proposition~\ref{prop:BraidUndistorted} shows that Proposition~\ref{prop:CCundistorted} applies, proving our theorem.
\end{proof}

\subsection{Polycyclic subgroups are virtually abelian}

\noindent
This section is dedicated to the proof of the following statement:

\begin{thm}\label{thm:PolyMod}
For every locally finite planar tree $A$, polycyclic subgroups in $\mathfrak{mod}(A)$ are virtually abelian.
\end{thm}

\noindent
The theorem is stated for asymptotically rigid mapping class groups instead of more general Chambord groups because the proof is much easier and because our study deals with asymptotically rigid mapping class groups anyway. However, it is worth noticing that, since Remark~\ref{remark:ChambordInMod} shows that a Chambord group always embeds into an asymptotically rigid mapping class group, the conclusion of Theorem~\ref{thm:PolyMod} automatically holds for Chambord groups.

\medskip \noindent
The three following preliminary lemmas will be needed in the proof of Theorem \ref{thm:PolyMod}.

\begin{lemma}\label{lem:NilpotentAbelian}
Let $G$ be a group all of whose infinite cyclic subgroups are undistorted. Then every finitely generated nilpotent subgroup in $G$ must be virtually abelian.
\end{lemma}

\begin{proof}
Every finitely generated nilpotent subgroup that is not virtually abelian contains distorted cyclic subgroups, see for instance \cite{MR1872804}.
\end{proof}

\begin{lemma}\label{lem:Normaliser}
For every $p \geq 0$, the centraliser of an abelian subgroup in the braid group $\mathcal{B}_p$ has finite index in its normaliser.
\end{lemma}

\begin{proof}
Let $H \leq \mathcal{B}_p$ be an abelian subgroup. Observe that, necessarily, $H$ must be finitely generated. Because $\mathcal{B}_p$ is biautomatic \cite[Theorem~9.3.4]{MR1161694}, it follows from \cite[Propositions~4.2 and~4.4]{MR3483604} that, given a Cayley graph $X$ of $\mathcal{B}_p$, there exists a vertex $x \in X$ such that, for every $h \in H$, $h$ acts as a translation of some length $\ell(h)$ on a bi-infinite geodesic passing through $x$. Fix some $R \geq 1$ sufficiently large such that $S:= \{h \in H \mid \ell(h) \leq R\}$ generates $H$. Observe that $S$ is finite, since it lies in $\{h \in H \mid d(x,h \cdot x) \leq R\}$, and that the action by conjugation on $H$ of the normaliser of $H$ in $\mathcal{B}_p$ stabilises $S$, since $\ell(\cdot)$ is conjugacy-invariant. By considering the kernel of this action, we conclude that the centraliser of $H$ in $\mathcal{B}_p$ has finite index in its normaliser, as desired.
\end{proof}

\begin{lemma}[\cite{GLUcremona}]\label{lem:FixedPoint}
Let $G$ be a finitely generated group acting on a CAT(0) cube complex without infinite cube. If the action is purely elliptic then it must be elliptic.
\end{lemma}

\begin{proof}[Proof of Theorem \ref{thm:PolyMod}.] 
Let $P \leq \mathfrak{mod}(A)$ be a polycyclic subgroup. According to Theorem~\ref{thm:BigChambordMCG}, $\mathfrak{mod}(A)$ is isomorphic to some Chambord group $C(\mathcal{P},w)$; so it acts on the CAT(0) cube complex $M(\mathcal{P},w)$. Then, we deduce from Theorem~\ref{thm:PolycyclicCC} that there exists a finite-index subgroup $\dot{P} \leq P$ such that $\mathcal{E}:= \{ g \in \dot{P} \mid \text{elliptic in $M(\mathcal{P},w)$} \}$ is a normal subgroup in $\dot{P}$ and such that $\dot{P}/ \mathcal{E}$ is free abelian. 

\medskip \noindent
It follows from the fact that $\dot{P}$ is polycyclic that $\mathcal{E}$ must be finitely generated. As a consequence of Lemma~\ref{lem:FixedPoint} and Claim~\ref{claim:VertexFixed}, $\mathcal{E}$ fixes a vertex in $\mathrm{M}(\mathcal{P},w)$. It follows from the characterisation of vertex-stabilisers given by Lemmas~\ref{lem:Translate} and~\ref{lem:Stab} (combined with Theorem~\ref{thm:BigChambordMCG}) that $\mathcal{E}$ lies in $\mathrm{Mod}(\Sigma)$ for some admissible subsurface $\Sigma$. From now on, we assume that $\Sigma$ is the unique admissible subsurface that is minimal with respect to the inclusion among the admissible subsurfaces $\Xi$ satisfying $\mathcal{E} \subset \mathrm{Mod}(\Xi)$. 

\medskip \noindent
Because $\dot{P}$ normalises $\mathcal{E}$, it must stabilise $\Sigma$. (More precisely, every element in $\dot{P}$ can be represented by a homeomorphism $g$ satisfying $g \Sigma= \Sigma$.) Consequently, the morphism $\dot{P} \to \mathrm{Aut}(\mathcal{E})$ given by the action of $\dot{P}$ on $\mathcal{E}$ by conjugation factorises through the inner automorphism group of $\mathrm{Mod}(\Sigma)$, i.e. for every $g \in \dot{P}$ there exists some $h \in \mathrm{Mod}(\Sigma)$ such that $g\beta g^{-1} = h \beta h^{-1}$ for every $\beta \in \mathrm{Mod}(\Sigma)$. Let $\ddot{P}$ denote the pre-image under this morphism of the inner automorphism group of $\mathrm{Mod}(\Sigma,\partial \Sigma)$.

\medskip \noindent
So far, we have a finite-index subgroup $\ddot{P} \leq P$ that contains the normal subgroup $\ddot{\mathcal{E}} := \mathcal{E} \cap \ddot{P} \leq \mathrm{Mod}(\Sigma,\partial \Sigma)$ such that $\ddot{P}/\ddot{\mathcal{E}}$ is free abelian and such that the action by conjugation of $\ddot{P}$ on $\ddot{\mathcal{E}}$ factorises through the inner automorphism group of $\mathrm{Mod}(\Sigma,\partial \Sigma)$. But we know that, in $\mathrm{Mod}(\Sigma, \partial \Sigma)$, solvable subgroups are virtually abelian \cite{MR726319}, so $\ddot{\mathcal{E}}$ must be virtually abelian. Because polycyclic groups are subgroup separable \cite{PolyLERF}, there exists a finite-index subgroup $\bar{P} \leq \ddot{P}$ such that $\bar{\mathcal{E}}:= \ddot{\mathcal{E}} \cap \bar{P}$ is abelian. As a consequence of Lemma~\ref{lem:Normaliser}, $\bar{P}$ virtually centralises $\bar{\mathcal{E}}$, i.e. there exists a finite-index subgroup $\hat{P}\leq \bar{P}$ containing $\bar{\mathcal{E}}$ in its center. 

\medskip \noindent
Now, $\hat{P}$ is nilpotent and it follows from Lemma \ref{lem:NilpotentAbelian} (which applies thanks to Theorems~\ref{thm:DistAbelian} and~\ref{thm:BigChambordMCG}) that it must be virtually abelian, concluding the proof of our theorem.
\end{proof}

\section{Concluding remarks and open questions}\label{section:last}

\noindent
Let us summarise what we know about the structure of subgroups in asymptotically rigid mapping class groups.

\paragraph{Finitely generated solvable subgroups.} Our main result shows that polycyclic subgroups must be virtually abelian and undistorted, so it is natural to ask if the conclusion holds for a more general family of finitely generated solvable groups. However, our next statement shows that this does not happen even for metabelian groups.

\begin{prop}\label{prop:Wreath}
Let $A$ be a locally finite planar tree. If $A$ admits a loxodromic isometry, then, for every $d \geq 1$, $\mathfrak{mod}(A)$ contains a copy of $\mathbb{Z}\wr \mathbb{Z}$ that has distortion at least polynomial of degree $d$.
\end{prop}

\noindent
The proposition applies for instance to the braided lamplighter group $\mathrm{br}\mathcal{L}$ and the braided Ptolemy-Thompson groups $\mathrm{br}T_n:= \mathrm{br}T_{n,n+1}$ (see Section~\ref{section:Examples} for definitions).

\begin{proof}[Proof of Proposition~\ref{prop:Wreath}.]
Let $g \in \mathfrak{mod}(A)$ be a rigid homeomorphism induced by a loxodromic isometry of $A$ and let $\tau$ be a braid twisting some pair of punctures represented by two consecutive vertices along the axis of the previous isometry. Then $\{g^{2k} \tau g^{-2k} \mid k \in \mathbb{Z}\}$ is a collection of braids with pairwise disjoint supports, so it defines the basis of a free abelian subgroup of infinite rank. Moreover, $\langle g^2 \rangle$ acts on it by conjugation: it shifts the generators of the free basis. It is clear that the subgroup $\langle \tau , g \rangle \leq \mathfrak{mod}(A)$ is isomorphic to $\mathbb{Z}\wr \mathbb{Z}$. 

\medskip \noindent
It is proved in \cite{MR2811580} that, for every $d \geq 1$, $\mathbb{Z} \wr \mathbb{Z}$ contains a copy of itself with distortion at least polynomial of degree $d$. This concludes the proof of our proposition.
\end{proof}

\noindent
Also, observe that, in the conclusion of our main result, virtually abelian subgroups cannot be replaced with abelian subgroups since braid groups themselves contain virtually abelian subgroups that are not abelian. For instance, if $\alpha,\beta$ are two braids having disjoint supports $A,B$ and if $\gamma$ is a twist that permutes $A,B$, then $\langle \alpha,\beta,\gamma \rangle$ decomposes as $\mathbb{Z}^2 \rtimes \mathbb{Z}$ where $\mathbb{Z}$ acts on $\mathbb{Z}^2$ by permuting the two direct factors. Clearly, $\langle \alpha,\beta,\gamma \rangle$ is not abelian (indeed, $\gamma \alpha \gamma^{-1}=\beta$) but $\langle \alpha , \beta, \gamma^2 \rangle$ is a free abelian subgroup of index two.

\paragraph{Large abelian subgroups.} As shown by Proposition~\ref{prop:Wreath}, asymptotically rigid mapping class groups may contain free abelian subgroups of infinite rank. It is natural to ask what kind of abelian subgroups can be found.

\begin{question}
Let $A$ be a locally finite planar tree. Is an abelian subgroup in $\mathfrak{mod}(A)$ necessarily a direct sum of cyclic groups? Can it contain $\mathbb{Q}$ for instance?
\end{question}

\noindent
The question about $\mathbb{Q}$ is particularly intriguing. If we compare with other Thompson-like groups, close to our Chambord groups, we know that every abelian subgroup in a diagram group must be free \cite[Theorem 16]{MR1725439} but, on the other hand, the cyclic extension $\hat{T}$ of Thompson's group $T$ contains $\mathbb{Q}$ \cite{BelkQ}.

\paragraph{Torsion subgroups.} Finite subgroups in asymptotically rigid mapping class groups have been classified in \cite{GLU}, and Theorem~\ref{thm:IntroFW} proves that finitely generated torsion subgroups must be finite. This is sufficient to show that, in some cases, torsion subgroups are always finite.

\begin{prop}\label{prop:TorsionFinite}
For all $m \geq 1, n \geq 2$ satisfying $m \neq n-1$, every torsion subgroup in $\mathrm{br}T_{n,m}$ must be finite.
\end{prop}

\begin{proof}
As a consequence of Theorem~\ref{thm:IntroFW}, every finitely generated torsion subgroup in $\mathrm{br}T_{n,m}$ must be finite. Therefore, any torsion subgroup in $\mathrm{br}T_{n,m}$ must be a union of finite subgroups. However, we know from \cite[Theorem~4.7]{GLU} that a finite subgroup in $\mathrm{br}T_{n,m}$ cannot be arbitrarily large, so the desired conclusion follows.
\end{proof}

\noindent
The case $m=n-1$ is less clear because $\mathrm{br}T_{n,n-1}$ contains finite-order elements of arbitrarily large order. Proposition~\ref{prop:TorsionFinite} contrasts with the fact that Thompson's group $T$ contains infinite torsion subgroups such as $\mathbb{Q}/\mathbb{Z}$ (see for instance \cite{MR2847521}). This leads naturally to the question:

\begin{question}
Does there exist a locally finite planar tree $A$ such that $\mathfrak{mod}(A)$ contains an infinite torsion subgroup?
\end{question}

\paragraph{Co-hopfian groups.} The diagrammatic representations of Thompson's groups $F,T,V$ can be used in order to prove that these groups are not co-hopfian, i.e. they contain proper copies of themselves. Following the terminology introduced in Section~\ref{section:CCV}, this can be done by replacing the $(1,2)$-diagram from the generating set of the groupoid with an arbitrary $(1,2)$-diagram. However, as mentioned earlier, our braided diagrams are not naturally endowed with the structure of a groupoid, and a similar argument fails. Hence the question: when are asymptotically rigid mapping class groups co-hopfian? As a particular case:

\begin{question}
Are the braided Higman-Thompson groups $\mathrm{br}T_{n,m}$ co-hopfian?
\end{question}

\addcontentsline{toc}{section}{References}

\bibliographystyle{alpha}
{\footnotesize\bibliography{Biblio}}

\newcommand{\etalchar}[1]{$^{#1}$}
\def\polhk#1{\setbox0=\hbox{#1}{\ooalign{\hidewidth
  \lower1.5ex\hbox{`}\hidewidth\crcr\unhbox0}}}
\begin{thebibliography}{ECH{\etalchar{+}}92}

\bibitem[BC08]{GuidoFW}
A.~Barnhill and I.~Chatterji.
\newblock Property ({T}) versus {P}roperty ({FW}).
\newblock {\em Section 5 in \emph{Guido's book of conjectures}. Collected by I.
  Chatterji. Enseign. Math.}, 54(2), 2008.

\bibitem[Bel04]{MR2706280}
J.~Belk.
\newblock {\em Thompsons' group {F}}.
\newblock ProQuest LLC, Ann Arbor, MI, 2004.
\newblock Thesis (Ph.D.)--Cornell University.

\bibitem[BF19]{MR4009393}
J.~Belk and B.~Forrest.
\newblock Rearrangement groups of fractals.
\newblock {\em Trans. Amer. Math. Soc.}, 372(7):4509--4552, 2019.

\bibitem[BG00]{MR1841750}
L.~Bartholdi and R.~Grigorchuk.
\newblock On the spectrum of {H}ecke type operators related to some fractal
  groups.
\newblock {\em Tr. Mat. Inst. Steklova}, 231(Din. Sist., Avtom. i Beskon.
  Gruppy):5--45, 2000.

\bibitem[BH99]{MR1744486}
M.~Bridson and A.~Haefliger.
\newblock {\em Metric spaces of non-positive curvature}, volume 319 of {\em
  Grundlehren der Mathematischen Wissenschaften [Fundamental Principles of
  Mathematical Sciences]}.
\newblock Springer-Verlag, Berlin, 1999.

\bibitem[BHM22]{BelkQ}
James Belk, James Hyde, and Francesco Matucci.
\newblock Embedding {$\Bbb{Q}$} into a finitely presented group.
\newblock {\em Bull. Amer. Math. Soc. (N.S.)}, 59(4):561--567, 2022.

\bibitem[BKM11]{MR2847521}
C.~Bleak, M.~Kassabov, and F.~Matucci.
\newblock Structure theorems for groups of homeomorphisms of the circle.
\newblock {\em Internat. J. Algebra Comput.}, 21(6):1007--1036, 2011.

\bibitem[BLM83]{MR726319}
J.~Birman, A.~Lubotzky, and J.~McCarthy.
\newblock Abelian and solvable subgroups of the mapping class groups.
\newblock {\em Duke Math. J.}, 50(4):1107--1120, 1983.

\bibitem[Bri04]{BrinnV}
M.~Brin.
\newblock Higher dimensional {T}hompson groups.
\newblock {\em Geom. Dedicata}, 108:163--192, 2004.

\bibitem[Bri07]{BrinbrV}
M.~Brin.
\newblock The algebra of strand splitting. {I}. {A} braided version of
  {T}hompson's group {$V$}.
\newblock {\em J. Group Theory}, 10(6):757--788, 2007.

\bibitem[BZ22]{TwistedThompson}
James Belk and Matthew C.~B. Zaremsky.
\newblock Twisted {B}rin-{T}hompson groups.
\newblock {\em Geom. Topol.}, 26(3):1189--1223, 2022.

\bibitem[Che00]{mediangraphs}
V.~Chepoi.
\newblock Graphs of some {$\rm CAT(0)$} complexes.
\newblock {\em Adv. in Appl. Math.}, 24(2):125--179, 2000.

\bibitem[Deg00]{Degenhardt}
F.~Degenhardt.
\newblock {\em Endlichkeitseigeinschaften gewisser {G}ruppen von {Z}\"{o}pfen
  unendlicher {O}rdnung}.
\newblock PhD thesis, Frankfurt 2000.

\bibitem[Deh06]{DehornoybrV}
P.~Dehornoy.
\newblock The group of parenthesized braids.
\newblock {\em Adv. Math.}, 205(2):354--409, 2006.

\bibitem[DL16]{MR3483604}
D.~Descombes and U.~Lang.
\newblock Flats in spaces with convex geodesic bicombings.
\newblock {\em Anal. Geom. Metr. Spaces}, 4(1):68--84, 2016.

\bibitem[DO11]{MR2811580}
T.~Davis and A.~Olshanskii.
\newblock Subgroup distortion in wreath products of cyclic groups.
\newblock {\em J. Pure Appl. Algebra}, 215(12):2987--3004, 2011.

\bibitem[ECH{\etalchar{+}}92]{MR1161694}
D.~A. Epstein, J.~Cannon, D.~Holt, S.~Levy, M.~Paterson, and W.~Thurston.
\newblock {\em Word processing in groups}.
\newblock Jones and Bartlett Publishers, Boston, MA, 1992.

\bibitem[Far03]{MR1978047}
D.~Farley.
\newblock Finiteness and {$\rm CAT(0)$} properties of diagram groups.
\newblock {\em Topology}, 42(5):1065--1082, 2003.

\bibitem[Far05]{MR2136028}
D.~Farley.
\newblock Actions of picture groups on {CAT}(0) cubical complexes.
\newblock {\em Geom. Dedicata}, 110:221--242, 2005.

\bibitem[FK04]{FunarUniversal}
L.~Funar and C.~Kapoudjian.
\newblock On a universal mapping class group of genus zero.
\newblock {\em Geom. Funct. Anal.}, 14(5):965--1012, 2004.

\bibitem[FK08]{FunarKapoudjian}
L.~Funar and C.~Kapoudjian.
\newblock The braided {P}tolemy-{T}hompson group is finitely presented.
\newblock {\em Geom. Topol.}, 12(1):475--530, 2008.

\bibitem[FKS12]{Survey}
L.~Funar, C.~Kapoudjian, and V.~Sergiescu.
\newblock Asymptotically rigid mapping class groups and {T}hompson's groups.
\newblock In {\em Handbook of {T}eichm\"{u}ller theory. {V}olume {III}},
  volume~17 of {\em IRMA Lect. Math. Theor. Phys.}, pages 595--664. Eur. Math.
  Soc., Z\"{u}rich, 2012.

\bibitem[Fun07]{FunarHoughton}
L.~Funar.
\newblock Braided {H}oughton groups as mapping class groups.
\newblock {\em An. \c{S}tiin\c{t}. Univ. Al. I. Cuza Ia\c{s}i. Mat. (N.S.)},
  53(2):229--240, 2007.

\bibitem[Gen22]{CubicalFlat}
Anthony Genevois.
\newblock Median sets of isometries in {$\rm CAT(0)$} cube complexes and some
  applications.
\newblock {\em Michigan Math. J.}, 71(3):487--532, 2022.

\bibitem[GK12]{MR3027509}
R.~Grigorchuk and Y.~Krylyuk.
\newblock The spectral measure of the {M}arkov operator related to 3-generated
  2-group of intermediate growth and its {J}acobi parameters.
\newblock {\em Algebra Discrete Math.}, 13(2):237--272, 2012.

\bibitem[GLU22]{GLU}
Anthony Genevois, Anne Lonjou, and Christian Urech.
\newblock Asymptotically rigid mapping class groups, {I}: {F}initeness
  properties of braided {T}hompson's and {H}oughton's groups.
\newblock {\em Geom. Topol.}, 26(3):1385--1434, 2022.

\bibitem[GLU24]{GLUcremona}
Anthony Genevois, Anne Lonjou, and Christian Urech.
\newblock Cremona groups over finite fields, {N}eretin groups, and
  non-positively curved cube complexes.
\newblock {\em Int. Math. Res. Not. IMRN}, (1):554--596, 2024.

\bibitem[GS91]{MR1090167}
P.~Greenberg and V.~Sergiescu.
\newblock An acyclic extension of the braid group.
\newblock {\em Comment. Math. Helv.}, 66(1):109--138, 1991.

\bibitem[GS97]{MR1396957}
V.~Guba and M.~Sapir.
\newblock Diagram groups.
\newblock {\em Mem. Amer. Math. Soc.}, 130(620):viii+117, 1997.

\bibitem[GS99]{MR1725439}
V.~Guba and M.~Sapir.
\newblock On subgroups of the {R}. {T}hompson group {$F$} and other diagram
  groups.
\newblock {\em Mat. Sb.}, 190(8):3--60, 1999.

\bibitem[GT21]{GT}
A.~Genevois and R.~Tessera.
\newblock A note on morphisms to wreath products.
\newblock {\em arXiv:2110.09822}, 2021.

\bibitem[Hig74]{HigmanV}
G.~Higman.
\newblock {\em Finitely presented infinite simple groups}.
\newblock Department of Pure Mathematics, Department of Mathematics, I.A.S.
  Australian National University, Canberra, 1974.
\newblock Notes on Pure Mathematics, No. 8 (1974).

\bibitem[Hou79]{Houghton}
C.~H. Houghton.
\newblock The first cohomology of a group with permutation module coefficients.
\newblock {\em Arch. Math. (Basel)}, 31(3):254--258, 1978/79.

\bibitem[Hug09]{Sim}
B.~Hughes.
\newblock Local similarities and the {H}aagerup property.
\newblock {\em Groups Geom. Dyn.}, 3(2):299--315, 2009.
\newblock With an appendix by Daniel S. Farley.

\bibitem[Kil97]{MR1448329}
V.~Kilibarda.
\newblock On the algebra of semigroup diagrams.
\newblock {\em Internat. J. Algebra Comput.}, 7(3):313--338, 1997.

\bibitem[LM16]{LM}
Y.~Lodha and J.~Moore.
\newblock A nonamenable finitely presented group of piecewise projective
  homeomorphisms.
\newblock {\em Groups Geom. Dyn.}, 10(1):177--200, 2016.

\bibitem[Mal58]{PolyLERF}
A.~Mal'cev.
\newblock On homomorphisms onto finite groups.
\newblock {\em Fluchen. Zap. Ivanovskogo Gos. Ped. Inst}, 18:49--60, 1958.

\bibitem[Mon13]{Monod}
N.~Monod.
\newblock Groups of piecewise projective homeomorphisms.
\newblock {\em Proc. Natl. Acad. Sci. USA}, 110(12):4524--4527, 2013.

\bibitem[Nek04]{Nekrashevych}
V.~Nekrashevych.
\newblock Cuntz-{P}imsner algebras of group actions.
\newblock {\em J. Operator Theory}, 52(2):223--249, 2004.

\bibitem[New42]{MR7372}
M.~Newman.
\newblock On theories with a combinatorial definition of ``equivalence.''.
\newblock {\em Ann. of Math. (2)}, 43:223--243, 1942.

\bibitem[NR98]{MR1459140}
G.~Niblo and M.~Roller.
\newblock Groups acting on cubes and {K}azhdan's property ({T}).
\newblock {\em Proc. Amer. Math. Soc.}, 126(3):693--699, 1998.

\bibitem[NSJG18]{QV}
B.~A. Nucinkis and S.~St. John-Green.
\newblock Quasi-automorphisms of the infinite rooted 2-edge-coloured binary
  tree.
\newblock {\em Groups Geom. Dyn.}, 12(2):529--570, 2018.

\bibitem[Osa18]{MR3786300}
D.~Osajda.
\newblock Group cubization.
\newblock {\em Duke Math. J.}, 167(6):1049--1055, 2018.
\newblock With an appendix by Mika\"{e}l Pichot.

\bibitem[Osi01]{MR1872804}
D.~Osin.
\newblock Subgroup distortions in nilpotent groups.
\newblock {\em Comm. Algebra}, 29(12):5439--5463, 2001.

\bibitem[Osi16]{OsinAcyl}
D.~Osin.
\newblock Acylindrically hyperbolic groups.
\newblock {\em Trans. Amer. Math. Soc.}, 368:851--888, 2016.

\bibitem[R{\"{o}}v99]{Rover}
C.~R{\"{o}}ver.
\newblock Constructing finitely presented simple groups that contain
  {G}rigorchuk groups.
\newblock {\em J. Algebra}, 220(1):284--313, 1999.

\bibitem[Ste92]{Stein}
M.~Stein.
\newblock Groups of piecewise linear homeomorphisms.
\newblock {\em Trans. Amer. Math. Soc.}, 332(2):477--514, 1992.

\bibitem[SWZ19]{MR3910073}
R.~Skipper, S.~Witzel, and M.~Zaremsky.
\newblock Simple groups separated by finiteness properties.
\newblock {\em Invent. Math.}, 215(2):713--740, 2019.

\bibitem[Thu17]{Thumann}
W.~Thumann.
\newblock Operad groups and their finiteness properties.
\newblock {\em Adv. Math.}, 307:417--487, 2017.

\bibitem[Woo17]{MR3704240}
D.~Woodhouse.
\newblock A generalized axis theorem for cube complexes.
\newblock {\em Algebr. Geom. Topol.}, 17(5):2737--2751, 2017.

\bibitem[WZ18]{Witzel}
S.~Witzel and M.~Zaremsky.
\newblock Thompson groups for systems of groups, and their finiteness
  properties.
\newblock {\em Groups Geom. Dyn.}, 12(1):289--358, 2018.

\bibitem[WZ19]{MR4033504}
S.~Witzel and M.~Zaremsky.
\newblock The {B}asilica {T}hompson group is not finitely presented.
\newblock {\em Groups Geom. Dyn.}, 13(4):1255--1270, 2019.

\end{thebibliography}

\Address

%

\end{document}